\documentclass[a4paper, 11pt, final]{amsart}
\usepackage{multicol}
\usepackage{multirow}
\usepackage[pdftex]{graphicx}
\usepackage{amsmath, amsfonts, amssymb, amsthm, float}
\usepackage{mathtools}
\usepackage{mathrsfs}
\usepackage{enumitem}
\usepackage[a4paper, left=2cm, bottom=2cm]{geometry}
\usepackage{array}
\usepackage{lmodern}
\usepackage[most]{tcolorbox}
\usepackage{xcolor}
\usepackage{url}
\usepackage{tikz-cd}
\usepackage{bm}
\usepackage{tabularx}
\usepackage{etoolbox}
\usepackage[colorlinks=true,linkcolor=black,citecolor=blue!80!black]{hyperref}
\usepackage{xcolor}

\usepackage[capitalize]{cleveref}

\newtheorem{definition}{Definition}[section]
\newtheorem{lemma}[definition]{Lemma}
\newtheorem{theorem}[definition]{Theorem}
\newtheorem{proposition}[definition]{Proposition}
\newtheorem{corollary}[definition]{Corollary}

\newtheorem{hypothesis}[definition]{Hypothesis}

\newcounter{claim}[definition]

\newtheorem{notation}[definition]{Notation}
\newtheorem*{notation*}{Notation}

\theoremstyle{remark}
\newtheorem{remark}[definition]{Remark}

\newtheorem{example}[definition]{Example}

\AtBeginEnvironment{theorem}{\setlist[enumerate,1]{label=(\roman*),font=\upshape}}
\AtBeginEnvironment{example}{\setlist[enumerate,1]{label=(\roman*),font=\upshape}}
\AtBeginEnvironment{definition}{\setlist[enumerate,1]{label=(\roman*),font=\upshape}}
\AtBeginEnvironment{lemma}{\setlist[enumerate,1]{label=(\roman*),font=\upshape}}
\AtBeginEnvironment{proposition}{\setlist[enumerate,1]{label=(\roman*),font=\upshape}}
\AtBeginEnvironment{corollary}{\setlist[enumerate,1]{label=(\roman*),font=\upshape}}
\AtBeginEnvironment{theorem}{\setlist[enumerate,2]{label=(\alph*),font=\upshape}}
\AtBeginEnvironment{lemma}{\setlist[enumerate,2]{label=(\alph*),font=\upshape}}
\AtBeginEnvironment{claim}{\setlist[enumerate,2]{label=(\alph*),font=\upshape}}

\renewcommand{\Gamma}{\varGamma}
\renewcommand{\epsilon}{\varepsilon}
\renewcommand{\bar}{\overline}
\renewcommand{\hat}{\widehat}
\renewcommand{\leq}{\leqslant}
\renewcommand{\geq}{\geqslant}

\newcommand{\normaleq}{\trianglelefteq}

\newcommand{\divides}{\bigm|}

\newcommand{\N}{\mathbb{N}}

\newcommand{\KK}{\mathbb{K}}
\newcommand{\FF}{\mathbb{F}}
\newcommand{\SL}{\operatorname{SL}}
\newcommand{\SU}{\mathrm{SU}}

\newcommand{\syl}{\mathrm{Syl}}
\newcommand{\Syl}{\operatorname{Syl}}
\newcommand{\GL}{\mathrm{GL}}
\newcommand{\Sp}{\mathrm{Sp}}

\newcommand{\PSL}{\mathrm{PSL}}

\newcommand{\PSU}{\mathrm{PSU}}

\newcommand{\Sz}{\mathrm{Sz}}
\newcommand{\Ree}{\mathrm{Ree}}
\newcommand{\Sym}{\mathrm{Sym}}
\newcommand{\Alt}{\mathrm{Alt}}

\newcommand{\Stab}{\mathrm{Stab}}
\newcommand{\Aut}{\mathrm{Aut}}
\newcommand{\Out}{\mathrm{Out}}

\newcommand {\End}{\mathrm{End}}
\newcommand {\Gal}{\mathrm{Gal}}

\def \wt {\widetilde}

\def \ov {\overline}
\newcommand{\gen}[1]{\langle #1 \rangle}

\renewcommand{\hat}{\widehat}
\newcommand{\norm}{\trianglelefteq}
\newcommand{\rdim}{\operatorname{rd}}
\newcommand{\pdeg}{\operatorname{pd}}

\numberwithin{equation}{section}
\makeatletter
\let\c@theorem\c@equation
\makeatother

\begin{document}

\author[V. Grazian]{Valentina Grazian}
\address{Dipartimento di Matematica, Universit\`{a} di Padova, 35121, Italia}
\email{valentina.grazian@math.unipd.it}

\author[J. Lynd]{Justin Lynd}
\address{Department of Mathematics, University of Louisiana at Lafayette, Maxim Doucet Hall, Lafayette, LA 70504}
\email{lynd@louisiana.edu}

\author[C. Parker]{Chris Parker}
\address{School of Mathematics, University of Birmingham, B15 2TT, United Kingdom}
\email{c.w.parker@bham.ac.uk}

\author[J. Semeraro]{Jason Semeraro}
\address{Department of Mathematical Sciences, Loughborough University, LE11 3TT, United Kingdom}
\email{j.p.semeraro@lboro.ac.uk}

\author[M. van Beek]{Martin van Beek}
\address{Department of Mathematics, University of Manchester, Manchester, M13 9PL, United Kingdom}
\email{martin.vanbeek@manchester.ac.uk}


\begin{abstract}
Suppose that $p$ is a prime and $X$ is a finite group with a strongly $p$-embedded subgroup, for example a rank $1$ group of Lie type in characteristic $p$. Let $m$ denote the $p$-rank of $X$ and assume that $m \ge 2$. We say that a faithful $\FF_p X$-module is $k$-active if some element of order $p$ in $X$ acts with exactly $k$ non-trivial Jordan blocks. In this paper, we determine the non-trivial composition factors of the faithful $\FF_pX$-modules which are $k$-active for some $k\leq m$.
\end{abstract}

\title{Modules with few Jordan blocks for rank $1$ groups of Lie type and related groups}

\maketitle

\section{Introduction}
Let $p$ be a prime and let $X$ be a finite group with a strongly $p$-embedded subgroup. Recall that a proper subgroup  $M$ of $X$ of order divisible by $p$ is \emph{strongly $p$-embedded} in $X$ if and only if for every $x \in X\setminus M$, $M\cap M^x$ is a $p'$-group. The most prominent examples of groups with a strongly $p$-embedded subgroup are the rank $1$ groups of Lie type in characteristic $p$. Throughout, we write $m:=m_p(X)$ for the \emph{$p$-rank} of $X$, the largest rank of an elementary abelian subgroup of $X$.

\begin{definition}
Suppose that $p$ is a prime, $Y$ is a finite group and $V$ is a faithful $\FF_pY$-module. Let $y \in Y$ have order $p$. Then $y$ is \emph{$k$-active} on $V$ if $y$ acts on $V$ with at exactly $k$ non-trivial Jordan blocks. We say that $V$ is $k$-active if there exists $y\in Y$ of order $p$ such that $y$ is $k$-active on $V$.
\end{definition}

Restrictions on the Jordan block structure of $p$-elements often impose strong constraints on both the representation and the ambient group, and have proved to be an effective tool in the study of modular representations and the $p$-local structure of finite groups. As examples, see \cite{Suprunenko}, \cite{MinActive} and \cite{TestZale}.

In this article, we study faithful $\FF_pX$-modules which are $k$-active, for some $k\leq m$. Our main theorem demonstrates that, with a few exceptions, such modules for groups with a strongly $p$-embedded subgroup arise as explicitly described irreducible modules of rank $1$ groups of Lie type in characteristic $p$. Much of the analysis in this paper is devoted to these groups, alongside the alternating groups of degree $2p$, and their associated $\mathbb{F}_p$-representations. Consequently, our results also contribute to the representation theory of small-rank groups of Lie type. Since these groups form fundamental building blocks in the analysis of larger quasisimple groups and their representations, a detailed understanding of their $k$-active modular representations is important and therefore of independent interest. The methodology developed to prove our theorems could be generalized to include larger classes of groups, complementing existing research in this area. As an initial step in this direction, in Section \ref{SU3Section} we determine the irreducible $\FF_{p^n}\SL_3(p^{n})$-modules for which a unipotent element has at most two non-trivial Jordan blocks (see Proposition \ref{SL3Split}).

Groups with a strongly $p$-embedded subgroup arise naturally in local group theory through what is now called the Alperin--Goldschmidt theorem \cite{AlperinFusion,GoldschmidtFusion} and as such play a significant role in the classification of the finite simple groups as does, of course, Bender's fundamental determination of the groups with a strongly $2$-embedded subgroup. This phenomenon is naturally phrased in the language  of saturated fusion systems. By the Alperin--Goldschmidt theorem, a saturated fusion system is largely determined by the automorphism groups of its essential subgroups. The key point being that the outer automorphism group $A$ of an essential subgroup $E$ has a strongly $p$-embedded subgroup, and, as $O_p(A)=1$, $E/\Phi(E)$ is a faithful $\FF_pA$-module that under some conditions must be $k$-active for some $k\leq m_p(S)$. If the essential subgroup is not normal in the group of the fusion system, then the structure of $A$ can often be studied using the theory of failure of factorization modules (see \cite{GeneralFF}). Failure of factorization modules are $k$-active for some $k\leq m$. More importantly, a result of Oliver \cite[\S 2]{nonrealizability} imposes restrictions on the Jordan block structure of a $p$-element which acts on a weakly closed abelian subgroup in \emph{any} simple saturated fusion system. An immediate application of our theorems appears in upcoming work of the authors alongside  Oliver \cite{NormalAbelianI}, where reduced fusion systems which have  a weakly closed abelian essential subgroup are classified.

We briefly describe some of the most important examples which appear in our main theorems.

Set $X=\SL_2(p^n)$ and $1\leq i\leq p-1$. Let $V_i(p^n)$ denotes the $\FF_{p^n}$-space spanned by homogeneous polynomials of degree $i$ in two commuting variables $y$ and $z$. Then $V_i(p^n)$ is made into an $(i+1)$-dimensional $\FF_{p^n}X$-module by enforcing that \[y\cdot\left(\begin{smallmatrix} a & b \\ c & d\end{smallmatrix}\right)=a y + b z\,\,\, \text{and}\,\,\, z\cdot \left(\begin{smallmatrix} a & b \\ c & d\end{smallmatrix}\right)= cy + dz\] and extending this action to $V_i(p^n)$. Notice that  for $\tau\in \Aut(\FF_{p^n})$,  $V_i(p^n)$ is isomorphic to $V_i(p^n)^\tau$ as an $\FF_{p}X$-module. When we take tensor products  $V_i(p^n)\otimes V_j(p^n)$, this tensor product is over $\FF_{p^n}$.

We call $V_1(p^n)|_{\FF_p}$ the \emph{natural module} for $\SL_2(p^n)$, of dimension $2n$. The \emph{natural module} for $\SU_3(p^n)$ is the restriction of any $3$-dimensional $\FF_{p^{2n}}\SL_3(p^{2n})$-module to $\SU_3(p^n)$ and then regarded as an $\FF_p\SU_3(p^n)$-module of dimension $6n$. Finally, the \emph{natural module} for $\Ree(3^n)$ is the restriction of any $7$-dimensional $\FF_{3^n}\mathrm{G}_2(3^n)$-module to $\Ree(3^n)$, viewed as an $\FF_3\Ree(3^n)$-module, of dimension $7n$.

We may now state our main theorem.

\begin{theorem}\label[theorem]{ModResult}
Suppose that $p$ is a prime, $X=O^{p'}(X)$ is a group with a strongly $p$-embedded subgroup and $V$ is a faithful $\FF_pX$-module. Let $m:=m_p(X)$ and $W$ be a non-trivial composition factor of $V$.

Assume that $m \geq 2$, $x \in X$ has order $p$ and $x$ is $k$-active on $V$ for some $k\leq m$. Then one of the following holds.
\begin{enumerate}
\itemsep0.5em
\item\label{modsl2} $X\cong\SL_2(p^m)$ or $\PSL_2(p^m)$ and, for $\sigma\in \Aut(\FF_{p^m})$ of order $m$, either
\begin{enumerate}
\itemsep0.5em
\item\label{modbasic} $W = V_i(p^m)$ for some $1\leq i \leq p-1$ considered as a $(i+1)m$-dimensional $\FF_pX$-module;
\item $p$ is odd and $W=(V_1(p^m)\otimes V_1(p^m)^{\sigma^j})|_{\FF_p}$ is a $4m$-dimensional $\FF_pX$-module, for some $j\in \{1,\dots, m-1\}$ with $j \ne m/2$;
\item\label{Omega4} $m$ is even and $W$ is an irreducible summand of $(V_1(p^m)\otimes V_1(p^m)^{\sigma^{m/2}})|_{\FF_p}$ of dimension $2m$;
\item\label{trialityintro} $p$ is odd, $m$ is divisible by $3$ and $W$ is an irreducible summand of $(V_1(p^m)\otimes V_1(p^m)^\tau\otimes V_1(p^m)^{\tau^2})|_{\FF_p}$, where $\tau = \sigma^{m/3}$, of dimension $8m/3$;
\item  $p\geq 5$, $m$ is divisible by $4$ and $W$ is an irreducible summand of $(V_1(p^m)\otimes V_1(p^m)^\tau\otimes V_1(p^m)^{\tau^2}\otimes V_1(p^m)^{\tau^3})|_{\FF_p}$, where $\tau = \sigma^{m/4}$, of dimension $4m$; or
\item $p \geq 5$, $m$ is even and $W$ is an irreducible summand of $(V_2(p^m)\otimes V_2(p^m)^{\sigma^{m/2}})|_{\FF_p}$ of dimension $9m/2$.
\end{enumerate}
\item\label{modsu3} $p$ is odd, $m=2n$, $X/Z(X)\cong \PSU_3(p^n)$, $M$ is a natural $3$-dimensional $\FF_{p^{2n}}\SU_3(p^n)$-module, $W=[V, X]/C_{[V, X]}(X)$ and either
\begin{enumerate}
\itemsep0.5em
    \item $x$ is any element of order $p$ and $W=M|_{\FF_p}$ is a natural $\SU_3(p^n)$-module of dimension $6n$;
    \item $p=3$, $x$ is not $p$-central and $W$ is any non-trivial, irreducible composition factor of $(M\otimes_{\FF_{p^{2n}}}  M^*)|_{\FF_3}$ of dimension $7n$;
    \item $p\geq 5$, $x$ is not $p$-central and $W$ is any non-trivial, irreducible composition factor of $(M\otimes_{\FF_{p^{2n}}} M^*)|_{\FF_p}$ of dimension $8n$; or
    \item $p\geq 5$, $x$ is not $p$-central and $W=\Sym^2(M)|_{\FF_p}$  of dimension $12n$.
\end{enumerate}
\item $p=3$, $x$ is not $3$-central, $m=2n$, $X\cong\Ree(3^n)$ and $W=[V, X]/C_{[V, X]}(X)$ is a natural $\Ree(3^n)$-module of dimension $7n$.
\item\label{M11Main} $p=3$, $m=2$, $x$ is any element of order $3$, $X\cong \mathrm{M}_{11}$ and $W=[V, X]/C_{[V, X]}(X)$ is either the code or the cocode module of dimension $5$.
\item\label{L34Main1} $p=3$, $m=2$, $x$ is any element of order $3$, $X\cong 2\cdot\PSL_3(4)$, $W=[V, X]$ and $\dim_{\FF_3} W=6$.
\item\label{L34Main2} $p=3$, $m=2$, $x$ is any element of order $3$, $X\cong 4\cdot\PSL_3(4)$, $W=[V, X]$ and $\dim_{\FF_3} W=8$.
\item\label{altmainintro} $p$ is odd, $m=2$, $X/O_{p'}(X)\cong \Alt(2p)$ and there is $L\le X$ such that $x\in S\in \syl_p(L)\subseteq \syl_p(X)$ and either $L\cong \Alt(2p)$; $p=3$ and $L\cong \SL_2(9)$; or $p=5$ and $L\cong 2\cdot\Alt(10)$.
\end{enumerate}
Furthermore, if \ref{altmainintro} holds then every irreducible composition factor of $V|_L$ is known and if $p\geq 5$ then, unless $L\cong \Alt(2p)$ and $x$ is a $p$-cycle, we have that $[V,L]/C_{[V, L]}(L)$ is irreducible.
\end{theorem}

For all cases  of Theorem \ref{ModResult}, we also provide explicit Jordan block structure of $W|_{\langle x\rangle}$ when these cases arise in the analysis. This data is conveniently displayed in Table~\ref{Thm1 Table} in Appendix~\ref{AppB}. If \ref{altmainintro} holds, then the possible composition factors of $V|_L$ are described in Proposition \ref{CompFactorsAlt2p} when $p\geq 5$. If instead we have that $p=3$ then $L/Z(L)\cong \PSL_2(9)$ and the possible composition factors of $V|_L$ are described in \ref{modsl2} of Theorem \ref{ModResult}.

We may extract the following consequence from Theorem~\ref{ModResult}.

\begin{corollary}
Suppose that $p$ is a prime, $X=O^{p'}(X)$ is a group with a strongly $p$-embedded subgroup and $V$ is a faithful $\FF_pX$-module with $V=[V, X]$ and $C_V(X)=0$. Assume that $m:=m_p(X) \geq 2$, $x \in X$ has order $p$ and $x$ is $k$-active on $V$ for some $k<m$.

Then $X\cong \PSL_2(p^m)$, $p$ is odd, $m$ is even, $V$ is any irreducible summand of $(V_1(p^m)\otimes V_1(p^m)^{\sigma^{\frac{m}{2}}})|_{\FF_p}$ of  dimension $2m$, and $x$ acts with $\frac m2$ trivial blocks and $ \frac{m}{2}$ blocks of size $3$.
\end{corollary}

Suppose that $x\in X$ has order $p$ and is $k$-active on $V$ for some $k\leq m$. Then $x$ centralizes a subspace of $V$ of codimension at most $m(p-1)$. A straightforward argument (see Lemma \ref{dimbound}) then shows that for $K$ a certain subgroup of $X$ which also has a strongly $p$-embedded subgroup,
\[\dim_{\FF_p} V/C_V(K)\leq 4m(p-1).\]
Consequently, there exists a non-trivial $K$-composition factor $W$ of $V$ such that $W$ has dimension at most $4m(p-1)$. This means that a central ingredient in the proof of Theorem~\ref{ModResult} is Theorem~\ref{lem:with one boundintro} applied to the pair $(K/C_K(W), W)$. When combined with the known classification of low-dimensional representations of quasisimple groups, this theorem significantly sharpens earlier bounds on the minimal dimension of faithful modular representations for groups with a strongly $p$-embedded subgroup (see, for example, \cite[Lemma 1.7]{OV}, \cite[Lemma 4.6]{parkersemerarocomputing}). Results of this type have already found repeated applications in group theory, representation theory and the study of fusion systems, and we expect that the strengthened bounds obtained here will prove similarly useful in future work.

\begin{theorem}\label[theorem]{lem:with one boundintro}
Assume that $X=O^{p'}(X)$ is a group with a strongly $p$-embedded subgroup with $m:=m_p(X)\geq 2$. Let $V$ be a faithful, irreducible $\FF_pX$-module and assume that $\dim_{\FF_p}V \leq 4m(p-1)$. Then one of the following holds:
\begin{enumerate}
\item\label{quasisimple} $O^{p'}(O^p(X))$ is quasisimple.
\item\label{5Caseintro} $(m,p) =(2,5)$, $\dim_{\FF_5} V=16$ and $X\cong (4\circ 2^{1+8}).\Alt(10)$.
\item\label{3Case2intro} $(m,p) =(2,3)$, $\dim_{\FF_3} V\in\{8, 16\}$ and $X\cong (4\circ 2^{1+4}).\Alt(6)$.
\item\label{3Caseintro} $(m,p) =(2,3)$, $\dim_{\FF_3} V=16$ and $X$ is isomorphic to $(4\circ 2^{1+6}).\Ree(3)$, $(4\circ 2^{1+6}).\SU_3(3)$, $2^{1+8}_+.\Ree(3)$ or $2^{1+8}_+.\Alt(6)$.
\item\label{permcase1intro} $(m,p) =(2,5)$, $X/O_{5'}(X) \cong \PSL_2(25)$, $O_{5'}(X)$ is an abelian $2$-group of order at most $4^{26}$, $\dim_{\FF_5} V=26$ and $V|_{O_{5'}(X)}$ is a direct sum of $26$ $1$-dimensional homogeneous components which $X$ permutes transitively.
\item\label{permcase2intro} $(m,p) =(2,3)$, and either 
\begin{enumerate}
    \item $X/O_{3'}(X)\cong \Ree(3)$ and  $\dim_{\FF_3} V=9$; or
   \item $X/O_{3'}(X)\cong \mathrm{M}_{11}$ and $\dim_{\FF_3} V\in \{11,12\}$.
    \end{enumerate} 
Moreover, in both cases, $O_{3'}(X)$ is an elementary abelian $2$-group of order $2^{\dim V-1}$, and $V|_{O_{3'}(X)}$ is a direct sum of $\dim V$ $1$-dimensional homogeneous components which $X$ permutes transitively.
\item\label{Alt2pCaseintronew} $p$ is odd, $m=2$, $X/O_{p'}(X) \cong \Alt(2p)$, $V|_{O_{p'}(X)}=V_1\oplus \dots \oplus V_{r}$ where $\dim_{\FF_p} V_i\leq 3$, and $X$ acts transitively on the set $\{V_1,\dots,V_r\}$. If $r\ne 2p$ then $p=3$, $r\in \{10, 15\}$ and $\dim_{\FF_p} V_i=1$. Furthermore, there is $L\le X$ such that $X=LO_{p'}(X)$ and $L$ is quasisimple.
\end{enumerate}
\end{theorem}

We have the following remark concerning case (vii) of Theorem \ref{lem:with one boundintro} where we do not provide an exact description of $O_{p'}(X)$.

\begin{remark}
 Let  $X:=\PSL_2(7)\wr \Alt(22)$ and write $O_{11'}(X)=K_1\times \dots \times K_{22}$ where $K_i\cong \PSL_2(7)$ for each $i$. Since $\PSL_2(7)$ has a $3$-dimensional faithful $\FF_{11}$-module, we can form the $\FF_{11}X$-module $V$ such that $V|_{O_{11'}(X)}=V_1\oplus \dots \oplus V_{22}$, $V_i=[V,K_i]$ is a $3$-dimensional $\FF_{11}K_i$-module and $X$ acts on the set $\{V_1,\dots, V_{22}\}$ in its natural $22$-point permutation action. Since $3\times 22=66<80=4\times 2\times 10$, this is an unavoidable example which is captured in Theorem~\ref{lem:with one boundintro} (vii).
\end{remark}

With some additional analysis of smaller cases, Theorem~\ref{lem:with one boundintro} reduces the proof of Theorem \ref{ModResult} to the situation where $X$ is a rank $1$ quasisimple group of Lie type in characteristic $p$, where powerful tools such as Steinberg's Tensor Product Theorem are available.

The main reduction for Theorem \ref{ModResult} is given in Proposition \ref{blocks lemma2}. From there, we initiate an analysis of the various cases which are treated in Proposition \ref{M11Prop}, Proposition \ref{PSL34Prop}, Proposition \ref{Alt2pp>5New}, Proposition \ref{SL2qmods}, Proposition \ref{SU3mods} and Proposition \ref{prop:ReeMods}.

The following corollary may be extracted from Theorem \ref{lem:with one boundintro}.

\begin{corollary}\label{cor:mc1}
Assume that $X=O^{p'}(X)$ is a group with a strongly $p$-embedded subgroup such that $m:=m_p(X)\geq 3$. If $V$ is a faithful $\FF_pX$-module and $\dim_{\FF_p} V\le 4m(p-1)$, then $O^{p'}(O^p(X))$ is quasisimple.
\end{corollary}

Finally, we note that the classification of the finite simple groups is used sparingly in this work. In Proposition \ref{SE2} we describe the collection $\mathcal S$ of finite groups of $p$-rank at least $2$ with a strongly $p$-embedded subgroup, and in Lemma \ref{NoComps} we invoke the Schreier property. All results of this paper, including Theorems \ref{ModResult} and \ref{lem:with one boundintro}, only concern the groups in $\mathcal{S}$. Thus, if readers are primarily concerned with the groups in the class $\mathcal{S}$ and are willing to impose mild additional hypotheses (for example, that $O_{p'}(X)$ is solvable), then they may apply the results of this work without appealing to the classification of the finite simple groups.  
 
The structure of the paper is as follows. Section \ref{speSec} collects various facts about strongly $p$-embedded subgroups, their generation properties and the structure of their representations of small dimensions. We reserve a specific subsection for the groups $\Alt(2p)$, and in particular some facts about their cohomology. These groups often cause us the most difficulty, and consequently, the results we prove about them in Theorem \ref{ModResult} and Theorem \ref{lem:with one boundintro} are often weaker when compared with other classes of groups. Section \ref{BoundSec} is devoted to the proof of Theorem \ref{lem:with one boundintro}, making use of the facts collected in Section \ref{speSec}. This is the most technical section of the paper, where we undertake an intensive analysis of the structure of a minimal counterexample to Theorem \ref{lem:with one boundintro}, ultimately showing that such a counterexample does not exist.

In Section \ref{ReductionSec}, we use Theorem \ref{lem:with one boundintro} to reduce the proof of Theorem \ref{ModResult} to the case where $X$ is isomorphic to a quasisimple group of Lie type of rank $1$ and in defining characteristic $p$. The remaining sections of the paper then analyze the possible representations of the linear, unitary and Ree groups respectively, which satisfy the hypothesis of Theorem \ref{ModResult}. We make liberal use of Steinberg's Tensor Product Theorem, and the Jordan block structure of tensor products and symmetric powers, as surveyed in Appendix \ref{JordanSec}. This paper also makes use of {\sc Magma} to eliminate small cases in the proofs of various results. The associated code is attached as an ancillary file.

\begin{notation*}
Our notation for groups is mostly standard and follows \cite{AschbacherFG}. We will use variants of the bar notation for quotients. Thus for $G$ a finite group, $A\le G$ and $N\norm G$, defining $\bar{G}:=G/N$ we have that $\bar{A}:=AN/N$. For $\mathbb{K}$ a field, $A$ a finite group and $V$ a $\mathbb{K}A$-module, we define the $n$-fold commutator $[V,A;n]$ by first setting $[V,A;1] := [V,A]$ and then inductively defining $[V,A;n]= [[V,A;n-1],A]$ for $n>1$. We let $\Pi(G)$ denote the set of primes which divide the order of the finite group $G$.

Suppose that $p$ is a prime, $\mathbb{K}$ a field of characteristic $p$ and $H$ is a cyclic group of order $p$. For $1\leq n \leq p$, the indecomposable $\mathbb{K}H$-module of dimension $n$ is denoted by $J_n(\mathbb{K})$ and we represent a direct sum of $m\geq 0$ copies of $J_n(\mathbb{K})$ by $mJ_n(\mathbb{K})$. Furthermore, if $|\mathbb{K}|= p^a$ for some $a\in \N$, we will adapt our notation and write $J_n(p^a)$ for $J_n(\mathbb{K})$. If $|\mathbb{K}|=p$, we shall write $J_n=J_n(p)$. Then for $V$ a finite dimensional $\FF_p H$-module, we have $V= \bigoplus\limits_{i=1}^p n_iJ_i$. Typically, we omit the terms with $n_i=0$.

Notice that if $V$ is an indecomposable $\FF_{p^n}H$-module, then there exists $k$ such that $V=J_k\otimes_{\FF_p}\FF_{p^n}$. In particular, when $V$ is considered as an $\FF_pH$-module, we have $V=nJ_k$.
\end{notation*}

\textit{Acknowledgements:} We are all grateful to the University of Birmingham and the Heilbronn Institute for Mathematical Research for
organizing and funding the workshop ``Patterns in Exotic Fusion Systems'' where this work was begun. We thank Bob Oliver for his contributions to this project and for providing the motivation to complete this work. We would also like to thank David Craven and Ellen Henke for their input earlier in the project.

The first author is a member of the GNSAGA INdAM research group and kindly acknowledges its support. The fourth author gratefully acknowledges funding from the UK Research Council EPSRC for the project EP/W028794/1. The fifth author is supported by the Heilbronn Institute for Mathematical Research.

\section{Generation and representations of groups with a strongly $p$-embedded subgroup}\label[section]{speSec}

In this section, we analyze the generation and representation theory of groups with a strongly $p$-embedded subgroup. We begin by presenting the groups of $p$-rank at least two which contain a strongly $p$-embedded subgroup.

\begin{proposition}\label[proposition]{SE2}
Suppose that $X=O^{p'}(X)$ is a group with a strongly $p$-embedded subgroup and set $\wt{X}:=X/O_{p'}(X)$. If $m_p(X)\geq 2$ then $\wt{X}$ is isomorphic to one of the following:
\begin{enumerate}
\item $\PSL_2(p^{a})$ for $p$ arbitrary and $a\geq 2$;
\item $\PSU_3(p^a)$ for $p$ arbitrary and $p^a>2$;
\item $\Sz(2^{2a+1})$ for $p=2$ and $a\geq 1$;
\item $\mathrm{Ree}(3^{2a+1})$ for $p=3$ and $a\geq 0$;
\item $\Alt(2p)$ for $p>3$;
\item\label{exception1} $\PSL_3(4)$ or $\mathrm{M}_{11}$ for $p=3$;
\item ${}^2\mathrm{F}_4(2)'$, $\mathrm{Fi}_{22}$, $\Sz(32):5$ or $\mathrm{McL}$ for $p=5$;
\item\label{exception2} $\mathrm{J}_4$ for $p=11$.
\end{enumerate}
\end{proposition}
\begin{proof}
Let $K$ be strongly $p$-embedded in $X$. If $X\ne KO_{p'}(X)$, then this follows from \cite[(2.5), (3.3)]{parkerSE} which in turn uses \cite[Theorem 7.6.1]{GLS3} (with the appropriate erratum). See also \cite[Proposition 4.5]{parkersemerarocomputing}. So assume that $X=KO_{p'}(X)$. By \cite[Theorem 5.3.16]{GOR}, using that $m_p(X)\geq 2$, we have \[O_{p'}(X)=\langle C_{O_{p'}(X)}(a) : a\in Q^\#\rangle\] for some $Q\le S\in\syl_p(X)$ which is elementary abelian of order $p^2$. Then, since $K$ is strongly $p$-embedded in $X$, we have $O_{p'}(X)\le K$ and $X=K$, a contradiction.
\end{proof}

\begin{notation}\label{SNotation}
Let $\mathcal{S}$ be the collection of finite groups $X$ such that $X=O^{p'}(X)$, $m_p(X)\geq 2$ and $X/O_{p'}(X)$ is isomorphic to one of the groups described in the outcome of Proposition \ref{SE2}. Whenever $X\in \mathcal{S}$, we set $T\in\syl_p(X)$, $m:=m_p(X)$, $n=\log_p(|Z(T)|)$, $R=O_{p'}(X)$ and $\wt{X}:=X/R$.
\end{notation}

We shall repeatedly use the following observation.

\begin{lemma}
Suppose that $X$ has a strongly $p$-embedded subgroup, $m_p(X)\geq 2$, $Y \le X$ and $N$ is a normal subgroup of $X$. Set $R:=O_{p'}(X)$.
 \begin{enumerate}
 \item If $X=YR$, then $m_p(Y)=m_p(X)$, and both $Y$ and $O^{p'}(Y)$ have strongly $p$-embedded subgroups.
 \item If $N \le R$, then $m_p(X/N)=m_p(X)$ and $X/N$ has a strongly $p$-embedded subgroup.
 \end{enumerate}
\end{lemma}
\begin{proof}
Since $X=YR$, comparing with the groups in Proposition \ref{SE2}, we have $X=O^{p'}(Y)R$. Hence, we may as well prove (i) only for $Y$. Let $T_0\in \syl_p(Y)\subseteq \syl_p(X)$. Then $T_0 \in \Syl_p(X)$ so that $m_p(Y)=m_p(X)$. Since $X$ has a strongly $p$-embedded subgroup and $m_p(X) \geq 2$, we have $X>\Gamma:=\langle N_X(U)\mid 1<U\le T_0\rangle \ge R$. In particular, $Y \not \le \Gamma$. Hence $\gen{N_{Y}(U)\mid 1<U\le T_0}\le \Gamma \cap Y< Y$ and so $Y$ has a strongly $p$-embedded subgroup. Therefore, (i) holds.

Assume that $N \le R$, $T \in \syl_p(X)$ and write $\ov X=X/N$. Then $\ov T \in \syl_p(\ov X)$ and $m_p(X)=m_p(\ov X)$. Assume $1 < \ov U \le \ov T$. Then $N_{\ov X}(\ov U)= \ov{N_X(U)}$ by the Frattini Argument. Hence, as $\ov X > \ov {\langle N_X(U)\mid 1<U\le T\rangle}= \langle N_{\ov X}(\ov U)\mid 1<\ov U\le \ov T\rangle$ and $O^{p'}(\ov X)= \ov{O^{p'}(X)}=\ov X$, we have that $\ov X$ has a strongly $p$-embedded subgroup. This is (ii).
\end{proof}

\begin{remark}
In the above lemma, we may replace the condition ``has a strongly $p$-embedded subgroup" by ``belongs to $\mathcal{S}$." Hence, $\mathcal{S}$ is closed under taking certain subgroups and quotients.
\end{remark}

\subsection{Generation of groups with a strongly $p$-embedded subgroup}

In this subsection, we provide results concerning the generation of the groups in $\mathcal{S}$ by conjugate $p$-elements. These will be pivotal to the proof of Theorem \ref{ModResult} via Lemma \ref{dimbound}.

\begin{lemma}\label[lemma]{SU3Generation}
Assume that $G= \SU_3(p^b)$, $p^b \geq 5$, and let $x$ be an element of order $p$ with Jordan form $J_3(p^{2b})$ on the natural $\FF_{p^{2b}}$-module for $G$. Then there exists an involution $t\in G$ such that $G=\langle x,t\rangle=\langle x,x^t\rangle$. Furthermore, if $p^b=3$, then $G$ is generated by two conjugates of $x$, but it is not generated by $x$ and an involution.
\end{lemma}
\begin{proof}
Set $q=p^b$. Since $x$ has Jordan form $J_3(q^2)$ on the natural $\FF_{q^2}$-module for $G$, we have that $p$ is odd and $x$ is a regular unipotent element. Then the lemma follows from \cite[Lemma 4.12]{PellegriniZalesski}.
\end{proof}

In Section \ref{SU3Section}, in order to understand $\FF_p$-representations of $\SU_3(p^b)$ we first must understand $\FF_{p^a}$-representations of $\SL_3(p^{2b})$ for certain $a\in \N$. Thus, while $\SL_3(p^b)$ is not in $\mathcal{S}$, it is necessary for our later methods to have structural information about $\SL_3(p^{2b})$, which we obtain in a similar manner as for those groups with a strongly $p$-embedded subgroup. For this reason, we record the following lemma here.

\begin{lemma}\label[lemma]{SL3Generation}
Assume that $G= \SL_3(p^b)$, $p$ is odd, and let $x$ be an element of order $p$ in $G$. If $x$ has Jordan form $J_2(p^{b})\oplus J_1(p^{b})$ on the natural $\FF_{p^{b}}$-module for $G$, then $G$ is generated by three conjugates of $x$. If $x$ has Jordan form $J_3(p^{b})$ on the natural $\FF_{p^{b}}$-module for $G$, then $G$ is generated by two conjugates of $x$.
\end{lemma}
\begin{proof}
Set $q=p^b$. If $x$ has Jordan form $J_3(q)$ on the natural $\FF_{q}$-module for $G$ then $x$ is a regular unipotent element of $G$. Then the result follows from \cite[Lemma 4.12]{PellegriniZalesski}.

Suppose now that $x$ has Jordan form $J_2(\FF_{q})\oplus J_1(\FF_{q})$ on the natural $\FF_{{q}}$-module for $G$. Then $x$ is conjugate to a root element of $G$. Assume that $q\ne 9$. Set $x\in S\in\syl_p(G)$ and let $P$ be a maximal parabolic subgroup of $G$ containing $S$. Then we can arrange, up to conjugation, that $x$ lies in a Levi subgroup of $P$. Since $q\ne 9$, by \cite[Proposition 2.16]{vbbook}, there is $g\in G$ such that $H:=\langle x, x^g\rangle\cong \SL_2(q)$ with $H\cap O_p(P)=1$ and $O^{p'}(P)=HO_p(P)$. Hence, $H$ lies in exactly two maximal subgroups of $G$, namely $P$ and its opposite parabolic. Now, let $P_2$ be the maximal parabolic of $G$ which contains $S$ but is not equal to $P$, and choose $\hat{S}\in \syl_p(P_2)\setminus \{S\}$ with $x\not\in Z(\hat{S})$. Then $O_p(P_2)=(H\cap S)\times Z(\hat{S})$ and there are $h_1, h_3\in P_2$ and $h_2\in P$ such that $x^{h_1}\in Z(S)=O_p(P)\cap O_p(P_2)$, $x^{h_1h_2}\in S\setminus O_p(P_2)$ and $x^{h_1h_2h_3}\in \hat{S}\setminus (S\cap \hat{S})$. Set $h:=h_1h_2h_3$. Since $[H\cap S, x^h]=Z(\hat{S})$, we have that $O_p(P_2)=[H\cap S, x^h](H\cap S)$. Then $O^{p'}(P)=\langle O_p(P_2)^H\rangle$ from which we deduce that $O^{p'}(P)\le \langle x, x^g, x^h\rangle$. Since $x^h\not\in P$, we conclude that $G=\langle x, x^g, x^h\rangle$, as desired. We verify the result when $q=9$ using {\sc Magma} \cite{Magma}.
\end{proof}

\begin{lemma}\label[lemma]{ReeGeneration}
Suppose that $G= \Ree(3^{2a+1})$, for $a\geq 1$, and $x$ is an element of order $3$. Then there exists an involution $t$ such that $G= \langle x,t\rangle=\langle x,x^t\rangle$.
\end{lemma}
\begin{proof}
Set $m_0:=3^a$ and $q=3m_0^2=3^{2a+1}$. We adopt the notation of \cite{Ward} and use the generic character table therein. We ascertain that $G$ has one class of involutions, $J$, and three classes of elements of order $3$: $\mathcal{X}$ whose elements are central in a Sylow $3$-subgroup, and $T$ and $T^{-1}$ whose elements centralize an involution. Let $t\in G$ be an involution lying in the class $J$ and $z$ an element of order $q+1+3m_0$ lying in the class $W$. By \cite[Table Chapter V]{Ward}, the only irreducible characters of $G$ which do not vanish on $x$ or $z$ are $\xi_1$, $\xi_3$, $\xi_6$ and $\xi_8$. We first aim to show that $xt\in W=z^G$.

From the structure constants formula \cite[Theorem 4.2.12]{GOR} we need to show \[0\ne a_{xtw}:=\sum_{\chi \in \mathrm{Irr}(G)} \frac{\chi(x)\chi(t)\overline{\chi(z^{-1})}}{\chi(1)} = \sum_{i\in\{1,3,6,8\}} \frac{\xi_i(x)\xi_i(t)\overline{\xi_i(z^{-1})}}{\xi_i(1)}\] for $x \in \{\mathcal{X},T\}$. We calculate
\[a_{xtw}=\begin{cases}
1+ \frac{(q-m_0)(q-1)}{ (q-1)m_0(q+1-3m_0)}+ \frac{(q-m_0)(q-1)}{2(q-1)m_0(q+1-3m_0)}=  1+ \frac{3(q-m_0) }{2m_0(q+1-3m_0)}, & x\in \mathcal{X}\\
1+ \frac{(-m_0+im_0^2\sqrt 3)(q-1)}{2(q-1)m_0(q+1-3m_0)} +  \frac{(-m_0-im_0^2\sqrt 3)(q-1)}{2(q-1)m_0(q+1-3m_0)}=1-\frac{1}{(q+1-3m_0)},& x\in T,\end{cases}\]
both of which are non-zero. Hence, $|\langle x,t \rangle|$ is divisible by $6(q+1+3m_0)$. The case where $x\in T^{-1}$ is similar.

Appealing to \cite[Theorem 6.5.5]{GLS3}, we deduce that either $\langle x,t\rangle= G$ or that $\langle x,t\rangle= N_G(\langle z\rangle)$ where $t$ is an involution chosen so that $xt=z$. In the latter case, however, we see that $x=zt\in \langle t, z\rangle$ which is a $3'$-group, a contradiction. Hence $G=\langle x,t\rangle$ and as $\langle x, x^t\rangle\norm G$, we have that $G= \langle x,x^t\rangle$, which proves the claim.
\end{proof}

We summarize the results about the generation of groups in $\mathcal{S}$ in the following proposition. We make use of {\sc Magma} \cite{Magma} in the proof.

\begin{proposition}\label[proposition]{generation}
Suppose that $X\in \mathcal{S}$. Then for $x\in T$ of order $p$, either there exists $g\in X$ such that $\wt{\langle x,x^g\rangle} \ge F^*(\wt X)$ or one of the following holds for $a\in \N$:
\begin{enumerate}
\item $\wt X \cong \PSL_2(2^a)$, $p=2$ and $a\geq 2$;
\item $\wt X \cong \PSL_2(9)$ and $p=3$;
\item $\wt X\cong \Sz(2^{2a+1})$, $p=2$ and $a\geq 1$;
\item $\wt X \cong \PSU_3(p^a)$, $p$ is arbitrary, $p^a>2$ and $\wt x \in Z(\wt T)$;
\item $\wt X\cong \Alt(2p)$, $p \geq 5$ and $\wt x$ is a $p$-cycle.
\end{enumerate}
Furthermore, in all cases, there exist $g,h \in X$ such that $F^*(\wt X) \le \wt{\langle x, x^g, x^h\rangle}$.
\end{proposition}
\begin{proof}
We work through the groups listed in $\mathcal{S}$.

For $\wt X \cong \PSL_2(p^a)$ the result holds by \cite[Proposition 2.16]{vbbook}. In particular, if $p^a=9$ then there is $g\in X$ with $\wt{\langle x,x^g\rangle}\cong \Alt(5)$, a maximal subgroup of $\wt X\cong \PSL_2(9)\cong \Alt(6)$. It is then easy to check that there is $h\in X$ such that $\wt{\langle x, x^g, x^h\rangle}=\wt X$. For $\wt X\cong \Sz(2^{2a+1})$ we appeal to \cite[Proposition 2.18(vi)]{vbbook} for the result. For $\wt X\cong \Ree(3^{2a+1})$ the result holds by Lemma \ref{ReeGeneration} if $a>0$. If $a=0$, then we verify the result using {\sc Magma} \cite{Magma}.

Suppose now that $\wt X\cong \PSU_3(p^a)$. Assume that $x\in Z(T)$. If $p\ne 2$ and $p^a\ne 9$, then the result holds by \cite[Proposition 2.17(vi)]{vbbook}. If $p^a=9$ then we appeal to {\sc Magma} \cite{Magma} for the result. If $p=2$, then we appeal to a result of Wagner \cite{WagnerInvolution}. If $x\not\in Z(T)$ then $p$ is odd and the claim holds by Lemma \ref{SU3Generation}.

Suppose that $\wt X\cong \Alt(2p)$ and that $p\geq 5$. Assume that $\wt x$ is of cycle type $(p,p)$. Let $a=(1,\dots,p)(p+1, \dots, 2p)\in \Alt(2p)$ and take $b=(p+2,p+1,p, \dots 3)(1,2,2p,2p-1, \dots p+4,p+3)\in\Alt(2p)$. Notice that $a$ and $b$ are conjugate in $\Alt(2p)$ by the element $(3, 2p, 4, 2p-1,\dots, p+3, p+1)$.

We calculate $ab=(1,2p,p,2,p+2)$ is a $5$-cycle, $ab^{-1}$ fixes $1$ in the natural action of $\Alt(2p)$ and \[ab^{-1}=(2,4,\dots,p+1,3,5,\dots,p, p+3,p+5,\dots ,2p, p+2,p+4,\dots,2p-1)\] is a $(2p-1)$-cycle. Hence $Y=\langle a,b\rangle$ is $2$-transitive and hence primitive in $\Alt(2p)$. Since $Y$ contains a $5$-cycle and $p \geq 5$,  \cite[Theorem 13.9]{wielandt} implies $Y=\Alt(2p)$. Up to relabeling, we may as well assume $\wt x=a$ and $\wt x^g=b$, and so the result holds in this case.

If $\wt x$ is a $p$-cycle, then up to relabeling we may as well assume that $\wt x$ corresponds to the element $(1, \dots, p)\in\Alt(2p)$. Then $(1, \dots, p)$, $(p+1, \dots,2p)$ and $(2, 3,\dots p+1)$ are conjugate elements which generate $\Alt(2p)$, and the result holds.

Finally, that the groups listed in parts \ref{exception1} to \ref{exception2} of Proposition \ref{SE2} are all generated by $\wt x$ and a conjugate is quickly verified using {\sc Magma} \cite{Magma}.
\end{proof}

\begin{remark}
In all of the above exceptions, two conjugate $p$-elements with the appropriate restrictions fail to generate the entire group.
\end{remark}

\subsection{Representations of groups with a strongly $p$-embedded subgroup}

We now collect results concerning the representation theory of groups in $\mathcal{S}$. We fix the following notation.

\begin{notation}
Suppose that $X$ is a finite group and $p$ is a prime. Then
\begin{enumerate}
\item $\rdim_p(X)$ is the minimal dimension of a faithful $\FF_p$-representation of any perfect central $p'$-extension of $X$;
\item $\rdim_{p'}(X):=\mathrm{min}_{r\in \Pi(X)\setminus \{p\}} \rdim_r(X)$; and
\item $\pdeg(X)$ is the minimal degree of a non-trivial transitive permutation representation of $X$.
\end{enumerate}
\end{notation}

\begin{lemma}\label[lemma]{ReeBound}
Suppose that $X\cong \Ree(3^n)$ and $V$ is a faithful $\FF_3X$-module. Then $\dim_{\FF_3} V\geq 7n$.
\end{lemma}
\begin{proof}
Observe that $n$ is odd. Since $\Ree(3^n)$ is a maximal subgroup of $\mathrm{G}_2(3^n)$, $\Ree(3^n)$ has a $7n$-dimensional faithful $\FF_3$-module. Aiming for a contradiction, let $V$ be a faithful $\FF_3X$-module with $\dim_{\FF_3} V<7n$. Consider the maximal subgroup $H$ of $X$ with $H\cong 2\times \PSL_2(3^n)$ \cite[Theorem 6.5.5]{GLS3}. Set $K:=H'\cong \PSL_2(3^n)$ and let $\tau$ be an involution such that $\langle \tau\rangle=O_2(H)$. Choose $\hat{\tau}$ an involution in $K$. Note that every involution in $X$ is conjugate (by \cite[Theorem 6.5.5]{GLS3} the normalizer of a Sylow $2$-subgroup is a group of shape $2^3 : F_{21}$, where $F_{21}$ is the Frobenius group of order $21$). In particular, we see that $\dim_{\FF_3} C_V(\tau)=\dim_{\FF_3} C_V(\hat{\tau})$.

By coprime action, we have that $V=[V, \tau]\oplus C_V(\tau)$ is a $K$-invariant decomposition, and $[V, \tau]\ne 0$ as $V$ is faithful. By \cite[Theorem 6.5.5]{GLS3}, $X$ has maximal subgroups of shape $(3^n\pm 3^{\frac{n+1}{2}}+1) : 6$ which we denote $M_{\pm}$. We note that $3^{3n}+1 \divides |X|$ and comparing with the maximal subgroups of $X$, we deduce by Zsigmondy's theorem \cite{Zsigmondy} that one of $M_+$ or $M_-$ contains a Sylow $r$-subgroup of $X$, where $r$ is a primitive prime divisor of $3^{6n}-1$. Denote this maximal subgroup by $M$. It follows that a faithful $\FF_3M$-module has dimension at least $6n$.

Now, $O^3(M)$ is a dihedral group and so is generated by two involutions. Hence, in any faithful, irreducible $\FF_3O^3(M)$-module $W$, an involution fixes at most half the space and so inverts at least half of the space. Since all involutions in $X$ are conjugate, we conclude that $\dim_{\FF_3} [V, \tau]\geq 3n$. Hence, $\dim_{\FF_3} C_V(\tau)<4n$.

Before proceeding, we need to describe the candidates for $[V, \tau]$ and $C_V(\tau)$ as $\FF_3K$-modules.

\begin{claim}
Suppose that $p=3$, $n$ is odd and $W$ is a faithful, irreducible $\FF_3\PSL_2(3^n)$-module. If $\dim_{\FF_3} W<4n$ then $W$ is isomorphic to a natural $\Omega_3(3^n)$-module. In particular, $\dim_{\FF_3} C_W(t)=n$ for $t\in \PSL_2(3^n)$ an involution.
\end{claim}
\begin{proof}
Let $\mathbb{K}=\FF_{3^n}$ so that $\mathbb{K}$ is a splitting field for $\PSL_2(3^n)$. Set $\mathbb{L}= \End_{\FF_3\PSL_2(3^n)}(W)$, assume that $\mathbb{L}=\FF_{3^a}$ and write $n=ra$. Since $n$ is odd, both $a$ and $r$ are odd. We may regard $W$ as an $\mathbb{L}\PSL_2(3^n)$-module and it has $\mathbb{L}$-dimension $d:=\dim_{\FF_3} W/a$. Now $\bar{W}=W\otimes_{\mathbb L}\mathbb K$ is an irreducible $\mathbb{K}\PSL_2(3^n)$-module of $\mathbb{K}$-dimension $d$. We have that $\dim_{\mathbb{L}} W<4r$ and so $\dim_\mathbb{K} \bar{W}<4r$. It follows from \cite[Proposition 5.4.6]{KleidmanLiebeck} that $r\leq 3$. If $r=3$ then $\bar{W}=M\otimes M^{\sigma^{\frac{n}{3}}}\otimes M^{\sigma^{\frac{2n}{3}}}$ where $M$ is a $2$-dimensional $\mathbb{K}\SL_2(3^n)$-module. But in this case the involution in $\SL_2(3^n)$ acts non-trivially on $\bar{W}$, and as $\bar{W}$ is a $\mathbb{K}\PSL_2(3^n)$-module, we have a contradiction. Since $r$ is odd, we see that $r=1$ and $n=a$. Hence, $W$ may be regarded as an $\mathbb{L}\PSL_2(3^n)$-module, $W$ is basic and described in \cite{BrauerNesbitt}, and the claim holds.
\end{proof}

Assume first that $K$ centralizes $[V, \tau]$ so that $K$ acts non-trivially on $C_V(\tau)$. By the claim, we see that $\dim_{\FF_3} C_{C_V(\tau)}(\hat{\tau})\geq n$ and so \[\dim_{\FF_3} C_V(\hat{\tau})=\dim_{\FF_3} [V, \tau] +\dim_{\FF_3} C_{C_V(\tau)}(\hat{\tau})\geq 3n+n\geq 4n\] from which we deduce that $\dim_{\FF_3} C_V(\tau)\geq 4n$, a contradiction.

Assume now that $K$ centralizes $C_V(\tau)$. Then $C_V(\tau)\le C_V(K)\le C_V(\hat{\tau})$ and as $\tau$ and $\hat{\tau}$ are conjugate, we conclude that $C_V(\tau)=C_V(\hat{\tau})$. Since $\hat{\tau}$ normalizes $[V, \tau]$, we deduce that $[V, \tau]=[V, \hat{\tau}]$ so that $\tau$ acts fixed point freely on $[V, \tau]$. Since $H$ acts faithfully on $[V, \tau]$, we deduce that $\hat{\tau}\le Z(H)$, a contradiction.

Hence, $K$ acts non-trivially on $[V, \tau]$ and $C_V(\tau)$. By the claim, we see that $\dim_{\FF_3} [V, \tau, \hat{\tau}]=2n=\dim_{\FF_3} [C_V(\tau), \hat{\tau}]$ from which we conclude that $\dim_{\FF_3} [V, \hat{\tau}]=4n>\dim_{\FF_3} [V, \tau]$. Since $\tau$ and $\hat{\tau}$ are conjugate, this yields a final contradiction.
\end{proof}

The following table provides bounds on the minimal degree of various faithful representations of quasisimple groups in $\mathcal{S}$.

\begin{table}[H]
    \centering
    \begin{tabular}{|c|c|c|c|c|c|}\hline
    $\wt{X}$ & $p$ & $m_p(X)$ & $\pdeg(\wt X)$ & $\rdim_{p'}(\wt{X})$ & $\rdim_p(\wt{X})$ \\\hline
    $\PSL_2(p^n)$, $p^n\ne 9$ & all $p$ & $n$ & $p^n+1$ & $\lceil\frac{p^n-1}{2}\rceil$ & $2n$\\
    $\PSL_2(9)$& $p=3$ & $2$ & $6$ & $4$ & $4$\\
    $\PSU_3(p^n)$ & all $p$ & $\frac{2n}{(p,2)}$ & $p^{3n}+1$ & $p^n(p^n-1)$ & $6n$\\
    $\Alt(2p)$ & $p\geq 5$ & $2$ & $2p$ & $2p-2$ & $2p-2$ \\
    $\Sz(2^n)$ & $p=2$ & $n$ & $2^{2n}+1$ & $2^{\frac{n-1}{2}}(2^n-1)$ & $4n$\\
    $\Ree(3)$ & $p=3$ & $2$ & $9$ & $6$ & $7$ \\
    $\Ree(3^n)$, $n>1$ & $p=3$ & $2n$ & $3^{3n}+1$ & $3^n(3^n-1)$ & $7n$ \\
    $\mathrm{M}_{11}$ & $p=3$ & $2$ & $11$ & $10$ & $5$\\
    $\PSL_3(4)$ & $p=3$ & $2$ & $21$ & $6$ & $6$\\
    $\Sz(32):5$ & $p=5$ & $2$ & $1025$ & $20$ & $124$ \\
    $\mathrm{Fi}_{22}$ & $p=5$ & $2$ & $3510$ & $54$ & $78$ \\
    ${}^2\mathrm{F}_4(2)'$ & $p=5$ & $2$ & $1600$ & $26$ & $26$ \\
    $\mathrm{McL}$ & $p=5$ & $2$ & $275$ & $21$ & $21$ \\
    $\mathrm{J}_4$ & $p=11$ & $2$ & $173067389$ & $112$ & $1333$\\\hline
    \end{tabular}
    \caption{Degrees of faithful representations of groups with a strongly $p$-embedded subgroup}
    \label{tab:data}
\end{table}

\begin{proof}[Justification of Table \ref{tab:data}]
Suppose that $\wt X$ is a rank $1$ simple group of Lie type. Then the $p$-rank of $\wt X$ is given in \cite[Table 3.3.1]{GLS3} and the minimal faithful permutation degrees are given in \cite{MinPermRank1} and \cite{MinPermTwisted}. The values for $\rdim_{p'}(\wt X)$ are taken from \cite{Seitz2}. Following the proofs of \cite[Lemma 1.7]{OV} and \cite[Lemma 4.7]{parkersemerarocomputing}, using Zsigmondy's theorem \cite{Zsigmondy}, we deduce the relevant values of $\rdim_p(\wt X)$, except when $\wt X\cong \Ree(3^n)$ in which case we have that $\rdim_3(\Ree(3^n))\geq 6n$. For this case, we appeal to Lemma \ref{ReeBound}.

For the remainder of the groups, we appeal to a combination of \cite[Table 3.3.1, Table 5.6.1]{GLS3} and {\sc Magma} \cite{Magma} for the $p$-ranks of $\wt X$ and we appeal to \cite{MinPermRank1, MinPermTwisted, Sporadics} for the minimal faithful permutation degree of the relevant groups. We appeal to a combination of \cite{ModAt} and \cite{Jansen} for the minimal dimensions of faithful representations.
\end{proof}

\begin{proposition}\label[proposition]{cor:data}
Suppose that $X\in \mathcal{S}$. Let $x\in T$ have order $p$. Then either $\wt X \cong \Alt(2p)$, $\wt X\cong \Ree(3)$ or
\begin{enumerate}
\item $O^p(\wt X)$ is contained in a subgroup generated by three conjugates of $\wt x$;
\item $\pdeg(\wt X)\geq p^m+1$; and
\item $\rdim_{p'}(\wt{X})\geq \frac{p^m-1}{2}$.
\end{enumerate}
\end{proposition}
\begin{proof}
This follows from Proposition \ref{generation} and Table \ref{tab:data}.
\end{proof}

\subsection{Small representations of $\Alt(2p)$}\label[section]{sec: Alt2p}

In proving Theorem \ref{ModResult} and Theorem \ref{lem:with one boundintro}, the case where $X/O_{p'}(X)\cong \Alt(2p)$ causes the most difficulty. In this subsection, we record necessary information about the representation theory of $\Alt(2p)$, and certain cohomology groups associated to small $\Alt(2p)$-modules.

\begin{lemma}\label[lemma]{permalt(2p)}
Suppose that $X$ is quasisimple with $X/Z(X)\cong \Alt(2p)$ where $p\geq 3$ is a prime. Assume that $X$ acts non-trivially and transitively on a set of size $d$ with $d\leq 8(p-1)$. Then either $d=2p$, or $d\in\{10, 15\}$ and $p=3$.
\end{lemma}
\begin{proof}
Notice that if $p\geq 5$, then $d\leq 8(p-1) < \binom{2p}{2}$. Thus applying \cite[Theorem 5.2 A (i)]{dixon-mortimer} when $p\geq 5$ with $r=2$ yields that the stabilizer of a point in the action of $X$ has index $d=2p$. If $p=3$ then the proof follows by a consideration of the subgroup structure of $\Alt(6)$. This proves the claim.
\end{proof}

Whenever $p\geq 5$ and $r$ is an arbitrary prime, the natural $\FF_r$-module for $\Alt(2p)$ is the unique non-trivial composition factor of the $\FF_r\Alt(2p)$ permutation module of dimension $2p$. Hence, the natural module has $\FF_r$-dimension $2p-2$ when $r\in\{2,p\}$ and dimension $2p-1$ otherwise.

\begin{lemma}\label[lemma]{Alt2pModBound}
Suppose that $X$ is quasisimple with $X/Z(X)\cong \Alt(2p)$ where $p\geq 5$ is a prime, and let $S\in\syl_p(X)$. Let $V$ be a faithful, irreducible $\FF_pX$-module with $\dim_{\FF_p} V\leq 8(p-1)$. Then either
\begin{enumerate}
    \item $X\cong \Alt(2p)$ and $V$ is a natural module for $X$;
    \item $X\cong \Alt(10)$, $\dim_{\FF_5} V=28$ and each element of $S^\#$ has more than two non-trivial Jordan blocks in its action on $V$; or
    \item $p=5$, $X\cong 2\cdot \Alt(10)$, $\dim_{\FF_5} V=8$, if $x\in S\#$ corresponds to a $5$-cycle then $V|_{\langle x\rangle} = J_4\oplus J_4$ and otherwise  $V|_{\langle x\rangle} = J_3\oplus J_5$. 
\end{enumerate}
\end{lemma}
\begin{proof}
If $X\cong \Alt(2p)$ then we appeal to \cite[Theorem 7]{James} to see that $V$ is a natural module of dimension $2p-2$ provided $p\geq 7$. If $p=5$ then \cite{ModAt} reveals that either $V$ is a natural module of dimension $8$, or $\dim_{\FF_5} V = 28$. If $Z(X)\ne 1$ then we appeal to \cite{KleschevTiep} to see that $p\leq 7$. Moreover, if $p\leq 7$ then we appeal to \cite{ModAt} to see that $p=5$ and $V$ is irreducible of dimension $8$.

Suppose that $p=5$. Since the size of each non-trivial Jordan block of $V$ under the action of a $p$-element is at most $5$, if $x\in S^\#$ has at most two non-trivial Jordan blocks in its action on $V$, we have that $\dim_{\FF_5} V/C_V(x)\leq 8$. Then by Proposition \ref{generation}, we have that $X$ is generated by three conjugates of $x$, and as $V$ is irreducible, we deduce that $\dim_{\FF_5} V\leq 24$. Thus, to complete the proof we may assume that $X\cong 2\cdot \Alt(10)$ and $\dim_{\FF_5} V=8$. We calculate the Jordan block structure of the faithful, irreducible $8$-dimensional $\FF_5 2\cdot\Alt(10)$-modules using {\sc Magma} \cite{Magma}.
\end{proof}

\begin{proposition}\label[proposition]{CompFactorsAlt2p}
Suppose that $X$ is quasisimple with $X/Z(X)\cong \Alt(2p)$ where $p\geq 5$ is a prime, and let $S\in\syl_p(X)$. Let $V$ be a faithful $\FF_pX$-module and assume that $x\in S^\#$ is $k$-active on $V$ where $k\leq 2$. Then every non-trivial irreducible composition factor of $V$ is described by Lemma \ref{Alt2pModBound}. Moreover, either
\begin{enumerate}
    \item $[V, X]/C_{[V, X]}(X)$ is irreducible; or
    \item $X\cong \Alt(2p)$, $x$ is a $p$-cycle and $V$ has two non-trivial composition factors, both of which are natural modules for $X$.
\end{enumerate}
\end{proposition}
\begin{proof}
By Proposition \ref{generation} we have that $\Alt(2p)$ is generated by at most three conjugates of $x$. We see that $\dim_{\FF_p} C_V(x)$ coincides with the number of Jordan blocks that $x$ has when acting on $V$. Since each non-trivial Jordan block has dimension at most $p$, we have $\dim_{\FF_p} C_V(x) \geq \dim_{\FF_p} V -2(p-1)$. Then as $C_V(X)=\cap_{g\in X} C_V(x^g)$ where $\{x^g \mid g\in X\}$ generates $X$, and as we may choose $|\{x^g \mid g\in X\}|=3$ we deduce that $\dim_{\FF_p}V/C_V(X)\leq 6(p-1)$. Hence, the dimension of every composition factor of $V$ is bounded by $6(p-1)$ and so each factor is described in Lemma \ref{Alt2pModBound}. Moreover, case (ii) of Lemma \ref{Alt2pModBound} does not occur.

Suppose that $[V, X]/C_{[V, X]}(X)$ is not irreducible. If $p=5$ and $X\cong 2\cdot \Alt(10)$, then $x$ already has two non-trivial Jordan blocks of size $4$ on one of the composition factors and we obtain a contradiction. If $X\cong \Alt(2p)$ then every non-trivial composition factor of $V$ is a natural module. We calculate that whenever $x$ is not a $p$-cycle, $x$ has two Jordan blocks on the natural module, of size $p$ and $p-2$, and we obtain a contradiction. Finally, if $x$ is a $p$-cycle then $x$ has a unique non-trivial Jordan block of size $p$ on the natural module, and it follows that $V$ has exactly two non-trivial composition factors.
\end{proof}

\begin{lemma}\label[lemma]{AltCohomology}
Suppose that $X\cong \Alt(2p)$, $5\leq p\ne r\in \Pi(\Alt(2p))$ and $V$ is an $\FF_rX$-module with $\dim_{\FF_r} V\leq 6p$, $C_V(X)=0$ and $[V, X]$ irreducible. Then $H^2(X, V)$ is trivial.
\end{lemma}
\begin{proof}
An appeal to \cite[Theorem 7]{James} and \cite{ModAt} reveals that $[V, X]$ is a natural $\FF_rX$-module of dimension $2p-1$ if $r$ is odd, or dimension $2p-2$ if $r=2$. Then $[V, X]$ is self-dual and has trivial $1$-cohomology when $r\ne 2$ by \cite[Lemma 1]{premet}. If $r=2$ and $p\geq 5$ then setting $W$ to be the $\FF_2X$-permutation module, we have that $[V, X]\cong [W/C_W(X), X]$ and by \cite[Lemma 1]{premet}, we see that both $H^1(X, W/C_W(X))$ and $H^0(X,W/C_W(X))$ are trivial so that $H^1(X, [V, X])\cong H^0(X, \FF_2)\cong \FF_2$. For this we calculate in the long exact sequence in cohomology associated to the extension \[0\rightarrow [W/C_W(X), X]\rightarrow W/C_W(X)\rightarrow \FF_2\rightarrow 0\] as in \cite[p.584]{premet}. We conclude that either $V=[V, X]$ is irreducible; or that $V\cong W/C_W(X)$. Then \cite[Theorem 1, Corollary 1]{premet} completes the proof of the lemma.
\end{proof}

For the following lemma, we employ {\sc Magma} \cite{Magma} for various representation theoretic and cohomological calculations.

\begin{lemma}\label[lemma]{alt6case}
Suppose that $X=O^{3'}(X)$, $X/O_2(X)\cong \Alt(6)$ and $R:=O_2(X)$ is an elementary abelian $2$-group. Then either:
\begin{enumerate}
    \item there is $K\le X$ with $K/Z(K)\cong \Alt(6)$ and $|Z(K)|\leq 2$; or
    \item $X$ contains a subgroup $L$ such that $X=LR$ and recognizing $L\cap R$ as an $\FF_2\Alt(6)$-module we have that $L\cap R$ is indecomposable with socle series having terms with dimensions $4$ and $1$; or $4$, $4$ and $1$.
\end{enumerate}
In particular, if (i) does not hold then the minimal dimension of a faithful $\FF_3X$-module is $30$ and for $P$ any maximal subgroup of $L$ containing $L\cap R$ with $P/L\cap R\cong \Alt(5)$, we have that $P$ contains no subgroup of index $2$.
\end{lemma}
\begin{proof}
We observe that $\Alt(6)$ has two permutation representations of degree $6$, and we label the associated $\FF_2$-permutation modules by $V_1$ and $V_2$. We set $W_i=V_i/C_{V_i}(\Alt(6))$ so that $\dim_{\FF_2} W_i=5$. Furthermore, $U_i:=[W_i, \Alt(6)]$ is an irreducible $\FF_2\Alt(6)$-module of dimension $4$. Indeed, by \cite{ModAt}, $U_1$ and $U_2$ are the unique non-trivial irreducible $\FF_2\Alt(6)$-modules of dimension strictly less than $16$, and are both self-dual. We calculate that $U_i$ has $1$-dimensional $1$-cohomology and that $W_i$ is the unique $5$-dimensional $\FF_2\Alt(6)$-module with socle $U_i$. Moreover, $V_i$ is the unique $6$-dimensional module with $1$-dimensional socle and $V_i/C_{V_i}(\Alt(6))\cong W_i$.

Let $X$ be a counterexample to the lemma with $|X|$ minimal. Since (i) does not hold, $X$ is not quasisimple and so $[R, X]\ne 1$. Set $Q\normaleq X$ such that $Q\le R$, $[R/Q, X]\ne 1$ and $Q$ has maximal order with respect to this. Note that $Q$ need not be unique. Since the $3'$-part of the Schur multiplier of $\Alt(6)$ has order $2$, and $X=O^{3'}(X)$, we see that $R/Q\cong U_i$ or $R/Q\cong W_i$ for some $i\in\{1,2\}$.

Let $Y<X$ with $X=YR$. If $Y$ satisfies part (i) of the lemma, then so too does $X$, a contradiction as $X$ is a counterexample. If $Y$ satisfies part (ii) of the lemma then there is $L\le Y\le X$ with $X=LR$, $Y=L(R\cap Y)$ and $L\cap R$ an $\FF_2\Alt(6)$-module with a prescribed description. Moreover, since the minimal dimension of a faithful $\FF_3Y$-module is $30$, the same holds for $X$. Finally, $P$ is a subgroup of $L$ and so is a subgroup of $X$ with the required properties, against $X$ being a counterexample. Hence, we may assume that no such $Y$ exists.

We observe that the $16$-dimensional $\FF_2\Alt(6)$-module is projective, and is the unique irreducible module of dimension at least $16$. Suppose that some $X$-composition factor contained in $R$ has the structure of a $16$-dimensional $\FF_2\Alt(6)$-module. Since the $16$-dimensional module is projective, we may arrange that there is $Z<R$ with $Z\norm X$ and $R/Z$ the $16$-dimensional module. Furthermore, $H^2(\Alt(6), R/Z)=0$ and so there is $Y$ with $R\cap Y=Z$ and $X=YR$, a contradiction. Hence, each $X$-composition factor contained in $R$ is either trivial, or isomorphic to $U_i$ for $i\in\{1,2\}$.

Suppose first that $R/Q\cong U_i$. We calculate in {\sc Magma} \cite{Magma} that $H^2(\Alt(6), U_i)$ is trivial for $i\in \{1,2\}$ so that there is $Q\le Y<X$ with $Y/Q\cong \Alt(6)$, a contradiction. Hence, $R/Q\cong W_i$. Since the $3'$-part of the Schur multiplier of $\Alt(6)$ has order $2$, we deduce that $Q\le [X, R]$. If there is $Q\le L<X$ with $L/Q\cong \Alt(6)$ then $X/[X, R]\cong 2\times \Alt(6)$ and so $X\ne O^{3'}(X)$, a contradiction. Thus, $X/Q$ corresponds to a non-trivial element of $H^2(\Alt(6), W_i)$. We calculate in {\sc Magma} \cite{Magma} that $H^2(\Alt(6), W_i)=\FF_2$ and that all candidates for $X/Q$ have no subgroup isomorphic to $\SL_2(9)$.

Suppose that $Q=1$. Taking $L=X$, outcome (ii) is satisfied with $|R|=2^5$. Moreover, we note that $X/[R, X]\cong 2\cdot\Alt(6)\cong \SL_2(9)$ and as $\SL_2(9)$ contains no subgroup isomorphic to $\Alt(5)$, for any maximal subgroup $P$ of $X$ containing $R$ with $P/R\cong \Alt(5)$ we see that $P=[P,P][X, R]$. Since the minimal degree of a faithful $\FF_2$-representation of $\Alt(5)$ is $4$, $P$ acts irreducibly on $[X, R]$ and so $P=[P,P]$. Finally, we verify in {\sc Magma} \cite{Magma} that the minimal dimension of a faithful $\FF_3X$-module is $30$. But then $X$ is not a counterexample, a contradiction.

Hence, $Q\ne 1$. Let $C\norm X$ with $C<Q$ and $C$ maximal subject to this. If $[X, Q]\le C$ then as $Q\le [X, R]$ we deduce that $R/C \cong V_i$ as an $\FF_2\Alt(6)$-module. Then $X/C$ has no subgroup isomorphic to $\Alt(6)$. However, we calculate in {\sc Magma} \cite{Magma} that $X$ has a subgroup $L$ with $L/C$ isomorphic to $\SL_2(9)$, a contradiction. Therefore, $Q/C\cong U_j$ for $j\in\{1,2\}$.

By maximality of $Q$, we must have that $R/C$ is a non-split extension of $W_i$ by $U_j$. We calculate in {\sc Magma} \cite{Magma} that $\mathrm{Ext}_{\FF_2\Alt(6)}(W_i, U_i)$ is trivial for $i\in\{1,2\}$. We set $T_i$ to be the unique non-split extension of $W_i$ by $U_{3-i}$ so that $R/C$ is isomorphic to $T_i$. Suppose that $C=1$. Taking $L=X$, outcome (ii) is satisfied with $|R|=2^9$. As above, we observe that for any maximal subgroup $P$ of $X$ containing $R$ with $P/R\cong \Alt(5)$ we have that $P=[P,P][X, R]$. Since the minimal degree of a faithful $\FF_2$-representation of $\Alt(5)$ is $4$, $P$ acts irreducibly on $[X, R]/Q$ and on $Q$ and so $P=[P,P]$. Finally, we verify in {\sc Magma} \cite{Magma} that the minimal dimension of a faithful $\FF_3D$-module is $30$, where $D$ is some maximal subgroup of $X$ containing $R$ with $D/R\cong \Alt(5)$. Hence, the minimal dimension of a faithful $\FF_3X$-module is at least $30$. But then $X$ is not a counterexample, a contradiction.

Hence, there is $D<C$ with $D\norm X$. Choose $D$ maximally with respect to this. Then as an $\FF_2\Alt(6)$-module, $C/D$ is either a trivial module or is isomorphic to $U_k$ for some $k\in\{1,2\}$. Again, by maximality of $Q$ and Schur multiplier considerations, we conclude that $R/D$ corresponds to a non-split extension of $T_i$ by a trivial module, or a non-split extension of $T_i$ by a $U_k$.  We calculate in {\sc Magma} \cite{Magma} that $\mathrm{Ext}_{\FF_2\Alt(6)}(T_i, U_k)$ is trivial for all $i,k\in\{1,2\}$. Thus, $C/D$ is a trivial module. Moreover, we may assume that $C/D=\mathrm{Soc}(R/D)$ as an $\FF_2\Alt(6)$-module. We calculate in {\sc Magma} that there is a unique extension of $T_i$ by the trivial module with $1$-dimensional socle, which we label $A_i$. We have that $X/D$ does not split over $R/D$. We verify that $H^2(\Alt(6), A_i)$ is $1$-dimensional and construct the unique non-split extension of $\Alt(6)$ by $A_i$. Finally, we calculate that such a group contains a subgroup isomorphic to $\SL_2(9)$, a final contradiction. This completes the proof.
\end{proof}

\section{Bounds on minimal dimensions of representations and the proof of Theorem \ref{lem:with one boundintro}}\label[section]{BoundSec}

In this section, we use the results of the previous section alongside known information about representations of groups with a strongly $p$-embedded subgroup to prove Theorem \ref{lem:with one boundintro}. This section is technical, and the proofs of the results within are long. Beyond this section, the remaining results only make use of Theorem \ref{lem:with one boundintro}, and so do not utilize the supporting lemmas and propositions from this section. As a consequence, the reader may skip this section in a first pass provided they are willing to take Theorem \ref{lem:with one boundintro} as a black box.

\begin{lemma}\label[lemma]{lem:perm degree}
Assume that $p$ is a prime and that the finite group $X=O^{p'}(X)$ acts transitively on the set $\Omega$ of size $k$. Let $R$ be a proper normal $p'$-subgroup of $X$ and write $\wt X:=X/R$. Then $k=1$ or $k \geq \pdeg(X/R)$.
\end{lemma}
\begin{proof}
Write $\Omega=\{\omega_1, \dots, \omega_k\}$ and, for $1\leq i\leq k$ set $X_i=\Stab_X(\omega_i)$. Define $K_0=X$, and, for $1\leq j\leq k$, put $K_j=\bigcap_{i=1}^j X_i$. Then $K_k$ is the kernel of the action of $X$ on $\Omega$. Assume that $X=K_kR$. Then, as $R$ is a $p'$-group, $K_k$ contains a Sylow $p$-subgroup $T$ of $X$. Hence $X=O^{p'}(X)=\gen{T^X}\le K_k$, which yields $X=K_k$ and $k=1$. So we may assume that $K_kR<X$. Choose $j$ maximal such that $K_jR=X$. Then $j<k$, and \[k\geq |\omega_{j+1}\cdot K_j|= |K_j :K_j\cap X_{j+1}|=|K_j:K_{j+1}|\geq |K_jR:K_{j+1}R|= |X:K_{j+1}R|\geq \pdeg(\wt X),\]
as desired.
\end{proof}

Recall from Notation \ref{SNotation} that if $X\in\mathcal{S}$, we set $T\in\syl_p(X)$, $m:=m_p(X)$, $R:=O_{p'}(X)$ and $\wt X:=X/R$.

\begin{lemma}\label[lemma]{NoComps}
Suppose that $X\in \mathcal{S}$ and that $V$ is a faithful, irreducible $\FF_pX$-module with $\dim_{\FF_p}V \leq 4m(p-1)$. Then either
\begin{enumerate}
    \item $O^{p'}(O^p(X))$ is quasisimple;
    \item $E(X)=E(R)=1$ and $F^*(X)=F(X)$; or
    \item $\wt X\cong \Alt(2p)$, $X$ acts transitively by conjugation on the set of  its components and $E(X)=E(R)=\prod_{i=1}^{2p} K_i$ where:
    \begin{enumerate}
        \item $p\equiv 1, 9 \mod 10$ and $K_1\cong \SL_2(5)$;
        \item $p\equiv 1,2,4 \mod 7$ and $K_1\cong\PSL_2(7)$; or
        \item $p\equiv 1,4 \mod 15$ and $K_1\cong 3\cdot \Alt(6)$.
    \end{enumerate}
\end{enumerate}
In particular, if (i) or (iii) holds then $X$ appears in Theorem \ref{lem:with one boundintro}.
\end{lemma}
\begin{proof}
Assume that $X$ is a counterexample to the lemma of minimal order. If $K$ is a component of $X$ which is not contained in $R$, then as $O^p(X)R/R$ is simple, we conclude that $O^p(X)\le KR$. Hence $\langle K^X\rangle$ is a central product of components of $X$, all divisible by $p$, which is normal in $X$. Since $O^p(X)R/R$ is simple, it follows that $K$ itself is normal in $X$. Hence, $[K, R]\le O_{p'}(K)\le Z(K)$ and as $[K, K]=K$, we deduce that $[K, R]=1$. As $O^p(X)\le KR$, we conclude that $O^{p'}(O^p(X))=K$ is quasisimple, a contradiction as $X$ is a counterexample. Thus, we may assume that $E(X)\le R$ for the remainder of the proof. In particular, $E(X)=E(R)$ and by the Feit--Thompson theorem, we see that $p$ is odd.

Suppose that $E(X)=E(R)\ne 1$ and let $\Omega$ be the set of components of $X$. Then $X$ acts on $\Omega$ by conjugation. Assume that $E=K_1\dots K_k$ is a product of an $X$-orbit of components in $\Omega$, with $k \ne 1$. Then $E\norm X$. As $V$ is irreducible and faithful, $C_V(E)=0$. If $V|_E$ is not a single homogeneous component then $4m(p-1)\geq \dim_{\FF_p} V \geq \pdeg(\wt X)f$ by Lemma \ref{lem:perm degree} applied to the action of $X$ on the set of Wedderburn components of $V|_E$, where $f$ is the dimension of a Wedderburn component. Then Proposition \ref{cor:data} yields $f\leq 3$ with $f\geq 2$ only if $\wt X\cong \Alt(2p)$. Since $\GL_1(p)$ is solvable, we deduce that $\wt X\cong \Alt(2p)$, $f\geq 2$ and $K_1$ is isomorphic to a quasisimple subgroup of $\GL_3(p)$ of $p'$-order. Furthermore, we must have that $\Omega=\{K_1, \dots, K_{2p}\}$. Finally, \cite[Tables 8.1--8.4]{BHRD} implies that outcome (iii) holds for $X$. Letting $A$ be the subgroup generated by pure diagonal elements in $E(X)$, we see that $C_X(A)\cong \Alt(2p)$ and $C_X(A)$ is a complement to $R$ in $X$. Then $X$ arises as case \ref{Alt2pCaseintronew} of Theorem \ref{lem:with one boundintro}, contradicting our assumption that $X$ is a counterexample.

Hence, $V|_E$ is a single Wedderburn component, and so there is $U$ an irreducible and faithful $E$-submodule of $V$. Then $U$ arises as a tensor product of irreducible $\FF_pK_i$-modules and we conclude that $\dim_{\FF_p} U\geq 2^k$. As $k \geq mp$ by Table \ref{tab:data} and Lemma \ref{lem:perm degree}, we deduce that $2^{mp} \leq 2^k\leq 4m(p-1)$, another contradiction as $m \geq 2$. It follows that every component of $X$ is normal in $X$. Let $K$ be such a component. Then as $X/C_X(K)K$ is soluble by the Schreier property of simple groups and $K \le R$, $X=C_X(K)R$. But then $C_X(K)\ge O^{p'}(X)=X$ which is absurd. Hence $E(X)=E(R)=1$ and $F^*(X)=F(X)$, as claimed.
\end{proof}

For the remainder of this section, we assume the following hypothesis:

\begin{hypothesis}\label{hyp: modulehyp}
$X\in \mathcal{S}$; $V$ is a faithful, irreducible $\FF_pX$-module with $\dim_{\FF_p}V \leq 4m(p-1)$; $(X, V)$ does not appear as an outcome of Theorem \ref{lem:with one boundintro} and $|X|$ is minimal among all such pairs.
\end{hypothesis}

\begin{lemma}\label{ell-lemma}
We have that $E(R)=E(X)=1$, $F^*(X)=F(X)$ and there is $\ell\in \Pi(R)$ such that $[T, O_\ell(X)]\ne 1$.
\end{lemma}
\begin{proof}
By Lemma \ref{NoComps}, and as $X$ is a counterexample to Theorem \ref{lem:with one boundintro}, we see that $E(X)=E(R)=1$ and $F^*(X)=F(X)$. If $[T, \prod\limits_{\ell \in \Pi(R)} O_\ell(X)]=1$ then $T\le C_X(F(X))=Z(X)$, a contradiction. Hence, there is $\ell\in \Pi(R)$ such that $[T, O_\ell(X)]\ne 1$.
\end{proof}

\begin{notation}
Assuming Hypothesis \ref{hyp: modulehyp}, and in light of Lemma \ref{ell-lemma}, we set $\ell\in \Pi(R)$ maximally such that $[T, O_\ell(X)]\ne 1$. Then set $M$ to be a normal subgroup of $X$ contained in $O_\ell(X)$ such that $[T, M]\ne 1$, and $M$ is chosen minimally subject to these constraints.
\end{notation}

\begin{proposition}\label{prop:Mdich}
Either $M$ is an elementary abelian $\ell$-group; or $M$ is an $\ell$-group such that every proper characteristic subgroup of $M$ is cyclic and central in $X$.
\end{proposition}
\begin{proof}
Suppose that $M$ is not elementary abelian, and let $1\ne K<M$ with $K$ normal in $X$. Then by the minimal choice of $M$, $[X, K]=[T, K]^X=1$ so that $K \le Z(X)$. In particular, as $V$ is irreducible we see that $K$ is cyclic by Schur's lemma. Since characteristic subgroups of $M$ are normal in $X$ and $M$ is not elementary abelian, $M$ is an $\ell$-group with every proper characteristic subgroup cyclic and central in $X$.
\end{proof}

The remainder of the section separates according to the structure of the normal $\ell$-subgroup $M$. First we treat the case where $M$ is not elementary abelian; this leads to extraspecial $\ell$-groups and embeddings in small symplectic or orthogonal groups. We then show that this case produces only the exceptional examples already listed in Theorem~\ref{lem:with one boundintro}. Then we examine the case where $M$ is elementary abelian. Here, Clifford theory reduces the problem to transitive permutation actions of $\wt X$, and the difficult cases are controlled by the alternating group calculations from Section \ref{sec: Alt2p}.

\begin{hypothesis}\label{hyp: modulehyp1}
Hypothesis \ref{hyp: modulehyp} holds, $M$ is an $\ell$-group which is not elementary abelian, and every proper characteristic subgroup of $M$ is cyclic and central in $X$.
\end{hypothesis}

\begin{lemma}\label{lem:SympDet}
$M = Q\circ C$ with $Q$ an extraspecial $\ell$-group of order $\ell^{1+w}$, $|C|\leq 4$ and either:
\begin{enumerate}
    \item $(p,m, \ell)=(5,2,2)$, $8\leq w\leq 10$ and $\wt X\cong \Alt(10)$;
    \item $(p,m, \ell)=(3,2,2)$, $w\leq 8$ and $X/R$ is isomorphic to one of $\Alt(6)$, $\PSU_3(3)$, $\Ree(3)$ or $\PSL_3(4)$; or
    \item $p=2$, $\wt X\cong \PSL_2(2^m)$, $M=Q$ and $(p,m,\ell)\in\{(2,2,3), (2,2,5), (2,2,7), (2,3,3)\}$.
\end{enumerate}
\end{lemma}
\begin{proof}
We have that $M$ is an $\ell$-group which is not elementary abelian, and every proper characteristic subgroup of $M$ is cyclic and central in $X$. Furthermore, $X$ acts irreducibly on $M/Z(M)$ and $C_X(M)\le R$.

By a theorem of P. Hall \cite[Theorem 5.4.9]{GOR}, $M = Q\circ C$ with $Q$ an extraspecial $\ell$-group. Furthermore, $C$ is either cyclic; or $\ell = 2$, $|C| \geq 16$ and $C$ is dihedral, semidihedral, or quaternion. But if $C$ is non-abelian and of order at least $16$ then $Z_2(M)$ is characteristic, proper, and non-central in $M$, and so non-central in $X$. Since $M$ is chosen minimally, this is a contradiction. Set $t=(2, \ell)$. If $C$ is cyclic of order greater than or equal to $\ell^{1+t}$ then $\Omega_{t}(M)$ is a characteristic, proper subgroup of $M$ which is not central in $M$, and we obtain a contradiction as before. Hence if $\ell$ is odd then $M=Q$ is extraspecial and if $\ell=2$ then $C$ is cyclic of order at most $4$. We fix the notation $M=QC$.

\begin{claim}\label{clm:RG5}
$|Q|\geq \ell^{1+d}$ where $d=\rdim_{\ell}(\wt X)$.
\end{claim}

Set $d:=\rdim_{\ell}(\wt X)\geq \rdim_{p'}(\wt X)$ and note $|Q|=\ell^{1+w}$ where $|M/C|=\ell^w$ and $M=QC$. Suppose first that $X$ has no proper subgroup $L$ satisfying $X=LR$. In particular, every maximal subgroup of $X$ contains $R$ and so $R\le \Phi(X)$ is nilpotent. Indeed, $[M,R] < M$ and the minimal choice of $M$ yields $[M,R]\le C\le Z(X)$. Since $[X,M]=M$, again by the minimality of $M$, we now have $|M/Z(M)|\geq \ell^d$. This proves \ref{clm:RG5}.

Hence, to prove \ref{clm:RG5}, we may assume that there is $L<X$ such that $X=LR$. We choose $L$ minimal with respect to this. Then we have that $X=O^{p'}(L)R$ and so we may arrange that $L=O^{p'}(L)=O^{p'}(O^p(L))T$. Set $F:=O^{p'}(LM)$.

Assume first that $F<X$ and observe that $[M, T]\le F\cap M$. If $F\cap M\le Z(M)\le Z(X)$ then $[M, T]=[M, T, T]=1$ by coprime action, a contradiction. Hence, there is $W$ an $F$-composition factor of $V$ with $[M, T]\not\le C_F(W)$. Set $\ov F:=F/C_F(W)$. Then by minimality of $X$, $(\ov F, W)$ appears as an outcome in Theorem \ref{lem:with one boundintro}. Since $[M, T]\not\le C_F(W)$, $O^{p'}(O^p(\wt F))$ is not quasisimple. In all other cases, since $[M, T]\not\le C_F(W)$ and $M\cap F\norm F$, we see that $\ov{M\cap F}$ covers a non-central $F$-chief factor. Indeed, this chief factor arises as an irreducible $\FF_\ell$-module for $F/(F\cap R)\cong \wt X$ and \ref{clm:RG5} holds.

Thus, we may assume that $X=O^{p'}(LM)$. Then $L\cap M\le Z(M)$ else $M=\langle (L\cap M)^X\rangle=\langle (L\cap M)^L\rangle\le L$ and $X=LM\le L$, against $L<X$. Assume first that $L$ is quasisimple. Then $C_L(M/Z(M))\le Z(L)$ and $L$ acts faithfully on $M/Z(M)$. By definition, $|Q/Z(Q)|=|M/Z(M)|\geq \ell^d$, and \ref{clm:RG5} holds. Hence, $L$ is not quasisimple and there is $W$ an $L$-composition factor of $V$ such that $L/C_L(W)$ is not quasisimple. As $X$ is a minimal counterexample to Theorem \ref{lem:with one boundintro}, $L/C_L(W)$ is determined by Theorem \ref{lem:with one boundintro}. In particular, $m=2$.

Observe by \cite[Theorem 5.5.5]{GOR} that \[\ell^{\frac{w}{2}}\leq \dim_{\FF_p} V \leq 8(p-1).\] Furthermore, if $w=2$ then $X/C_X(M/Z(M))$ embeds in $\SL_2(\ell)$, where $C_X(M/Z(M))\le R$. Comparing with the maximal subgroups of $\SL_2(\ell)$ as provided by \cite[Theorem 6.5.1]{GLS3}, we have a contradiction in this case. Hence, $w\geq 4$. If $p=3$ then $8(p-1)=16$ and we ascertain that $\ell=2$. If $p=5$ then $8(p-1)=32$ so that $\ell=2$ or $\ell=3$ and $w\leq 6$.

Set $K:=L/C_L(W)$. If Theorem \ref{lem:with one boundintro} \ref{5Caseintro} or \ref{3Case2intro} hold for $K$, then $H^2(K, O_2(K)/Z(K))$ is trivial by Lemma \ref{AltCohomology} and a calculation as in Lemma \ref{alt6case}. By the minimality of $L$, this is a contradiction. Case \ref{Alt2pCaseintronew} of Theorem \ref{lem:with one boundintro} cannot hold by the minimality of $L$.

If one of Theorem \ref{lem:with one boundintro} \ref{3Caseintro}-\ref{permcase2intro} holds, then as $\dim_{\FF_p} V-\dim_{\FF_p} W<4(p-1)$ we deduce that $L/C_L(U)$ is quasisimple for every $L$-composition factor not equal to $W$. A consideration of the possible Schur multiplier of $\wt L$ then implies that $O_{p'}(L)$ is a $2$-group. If $p=3$ then $\ell=2$ and by minimality of $M$, $[M, O_{p'}(L)]\le [M,R]\le Z(M)$ and so $M/Z(M)$ is a faithful $\FF_2$-module for $\wt L=\wt X$. By definition, $|Q/Z(Q)|=|M/Z(M)|\geq \ell^d$, as desired. Hence, we have that $p=5$, $\wt X\cong \PSL_2(25)$, $\ell=3$ and $w\leq 6$. Since a Sylow $5$-subgroup of $\Sp_6(3)$ has order $5$, there is $t\in T^\#$ with $[t, M]\le Z(M)$. Then $[X, M]=[t, M]^X\le Z(M)$, a contradiction. Thus, \ref{clm:RG5} holds.

It follows from \ref{clm:RG5} that $Q$ has width at least $d/2$. As $Q$ acts faithfully on $V$, we now obtain \[\ell^{\frac{d}{2}}\leq \dim_{\FF_p} V \leq 4m(p-1)\] from \cite[Theorem 5.5.5]{GOR}. Appealing to Table \ref{tab:data}, we see that one of the promised outcomes holds.
\end{proof}

\begin{lemma}\label{lem:NoFix}
$C_{M/Z(M)}(X)=1$.
\end{lemma}
\begin{proof}
Since $X=O^{\ell}(X)$ for each choice of $X$ provided by Lemma \ref{lem:SympDet}, if $C_{M/Z(M)}(X)\ne 1$ then $X$ centralizes the preimage in $M$ of $C_{M/Z(M)}(X)$. Since $M\le X$, this leads to a contradiction.
\end{proof}

\begin{lemma}\label{lem:FirstSymplecticReduction}
If Hypothesis \ref{hyp: modulehyp1} holds then $p\leq 3$. Moreover, if $p=3$ and $|Z(M)|\geq 4$ then $|Z(M)|=4$ and $M\cong  4\circ 2^{1+6}$.
\end{lemma}
\begin{proof}
By Lemma \ref{lem:SympDet}, we have that $M = Q\circ C$ with $Q$ an extraspecial $\ell$-group of order $\ell^{1+w}$ and $|C|\leq 4$. Furthermore, $p\leq 5$ and $\ell=2$. We remark that by the maximal choice of $\ell$, we must have that $F(X)=O_2(X)Z(X)$.

Suppose first that $p=3$ so that $\wt X$ is isomorphic to one of $\Alt(6)$, $\PSU_3(3)$, $\Ree(3)$ or $\PSL_3(4)$ by Lemma \ref{lem:SympDet}. Suppose that $|Z(M)|\geq 4$. Let $a$ be such that $\End_{\FF_3X}(V)=\FF_{3^a}$. Then $\dim_{\FF_{3^a}} V \geq 2^{\frac{w}{2}}$ so that $3^a=9$, $|Z(M)|=4$ and $\frac{w}{2}\leq 3$. Comparing with Table \ref{tab:data} we see that if $w\ne 6$, then $\wt X\cong \Alt(6)$ and $w=4$. If $w=4$ then as $\Alt(6)\cong \Sp_4(2)'$ is a maximal subgroup of $\Sp_4(2)$, we deduce that $M\cong 4\circ 2^{1+4}$ and $C_X(M/Z(M))=MC_R(M)=R$.

Suppose now that $p=5$ so that $|M/Z(M)|\in \{2^8, 2^{10}\}$ and $\wt X\cong \Alt(10)$. We record by \cite{ModAt} that $\Alt(10)$ has a unique, non-trivial irreducible $\FF_2$-module whose dimension is at most $10$: the natural module. If $|M/Z(M)|=2^{10}$ then, analyzing the maximal subgroups of $\Sp_{10}(2)$ as in \cite[Tables 8.64, 8.65]{BHRD}, we have that $X/C_X(M/Z(M))$ either embeds in $2^9.\Sp_8(2)$; $\Sp_8(2)\times \Sym(3)$; $\mathrm{O}_{10}^+(2)$; or $\mathrm{O}_{10}^-(2)$. In the first two cases, $X$ must centralize a non-trivial subspace of $M/Z(M)$ and we have a contradiction by Lemma \ref{lem:NoFix}. In the latter cases, we deduce that $X$ further sits in a subgroup isomorphic to $\Sp_8(2)$ or $\Alt(12)$ and in this case, we calculate that $R=MC_X(M)$. Then we calculate that the $1$-cohomology of an irreducible $\FF_2\Alt(10)$-module has dimension $1$ from which we conclude that $C_{M/Z(M)}(X)\ne 1$, against Lemma \ref{lem:NoFix}. Hence, we have that $|M/Z(M)|=2^8$ and analyzing the maximal subgroups of $\mathrm{SO}_8^{\pm}(2)$ and of $\Sp_8(2)$ as provided by \cite[Tables 8.48, 8.50, 8.52]{BHRD}, we deduce $X/C_X(M/Z(M))$ does not embed in $\mathrm{SO}_8^{\pm}(2)$. Then $X/C_X(M/Z(M))$ embeds in the maximal subgroup of $\Sp_8(2)$ isomorphic to $\Sym(10)$. In particular, $M\cong 4\circ 2^{1+8}$ and $R=MC_R(M)$.

Aiming for a contradiction, we let $p$ be either $3$ or $5$, and if $p=3$ then we further assume that $|Z(M)|=4$ and $w=4$. Then the irreducible $\FF_pM$-modules are either $1$-dimensional, or the unique irreducible module of dimension $4(p-1)$. Since $M$ is non-abelian, $V|_M$ is completely reducible and $V$ is irreducible, we conclude that $V|_M=V_1\oplus V_2$ where $\dim_{\FF_p} V_1=4(p-1)$ and $\dim_{\FF_p} V_2\in \{0, 4(p-1)\}$.

Suppose that $\dim_{\FF_p} V_2=0$ so that $V|_M$ is irreducible. If $p=3$ then we may view $V$ as $4$-dimensional over $\FF_9$ and appealing to \cite[Tables 8.8, 8.10]{BHRD}, we see that $N_{\SL_{4}(9)}(M)\cong 4\circ 2^{1+4}.\Sp_4(2)\le \SU_4(3).2\le \SL_4(9)$. Hence, $C_X(M)=Z(M)$, $M=R$ and $X\cong 4\circ 2^{1+4}.\Alt(6)$ which is included as Theorem \ref{lem:with one boundintro}\ref{3Case2intro}. Hence, $X$ does not satisfy Hypothesis \ref{hyp: modulehyp1}, a contradiction. If $p=5$ then appealing to \cite[Table 3.5.A]{KleidmanLiebeck}, we see that $N_{\SL_{16}(5)}(M)\cong 4\circ 2^{1+8}.\Sp_8(2)$. Hence, $C_X(M)=Z(M)$, $M=R$ and $X\cong 4\circ 2^{1+8}.\Alt(10)$ which is included as Theorem \ref{lem:with one boundintro}\ref{5Caseintro}. Hence, $X$ does not satisfy Hypothesis \ref{hyp: modulehyp1}, another contradiction.

Thus, $V|_M=V_1\oplus V_2$ where $V_1\cong V_2$. If $p=5$ then we appeal to \cite[Theorem 3.5.6]{GOR} to see that the number of proper non-trivial submodules of $V|_M$ is $|\End_{\FF_5M} V_1|+1=6$. Let $K$ be the kernel of the action of $X$ on the proper non-trivial submodules of $V|_M$. If $K\not\le R$ then $X=KR$ so that $K$ contains a Sylow $5$-subgroup of $X$. Since $K\norm X=O^{p'}(X)$, we see that $X=K\le N_X(V_1)$, impossible as $V$ is irreducible. Hence, $K\le R$ and $X/K$ embeds as a subgroup of $\Sym(6)$. Since $|\Alt(10)|>|\Sym(6)|$, this is a clear contradiction.

Hence, we have that $V|_M=V_1\oplus V_2$ where $V_1\cong V_2$, $p=3$ and $M\cong 4\circ 2^{1+4}$. Appealing to \cite[Theorem 3.5.6]{GOR}, we see that the number of proper non-trivial submodules of $V|_M$ is $|\End_{\FF_3M}(V_1)|+1=10$. As before, we let $K$ be the kernel of the action of $X$ on the proper non-trivial submodules of $V|_M$, and observe that $K\le R$. Then, as $X=O^{3'}(X)$, $X/K$ embeds as a subgroup of $\Sym(10)$. Surveying the subgroups of $\Sym(10)$, using {\sc Magma} \cite{Magma}, we ascertain that $K=R$. In particular, $C_X(M)\le N_X(V_1)$.

Write $\bar{N_X(V_1)}=N_X(V_1)/C_X(V_1)$. Then $\bar{C_X(M)}$ centralizes $\bar{M}$. As calculated above, $\bar{M}$ is self-centralizing in $N_{\SL_{4}(9)}(\bar{M})$. Furthermore, one can calculate that $C_{\GL_4(9)}(\bar{M})$ is cyclic of order $8$. In particular, $C_X(M)/C_{C_X(M)}(V_1)$ is cyclic of order at most $8$. Similarly, $C_X(M)/C_{C_X(M)}(V_2)$ is cyclic of order at most $8$ and as $V=V_1\oplus V_2$ is a faithful module, we conclude that $|C_X(M)/\Phi(C_X(M))|\leq 4$. Since $X=O^{3'}(X)=O^2(X)$, we deduce that $X$ centralizes $C_X(M)$ and as $R=MC_X(M)$, we deduce that $X$ centralizes $R/M$. A consideration of the Schur multiplier of $\Alt(6)$, using that $X=O^{3'}(X)$ and considering the cohomology of $\Alt(6)$ in its action on the relevant submodules and factors of $R/\Omega_1(Z(M))$, as calculated in Lemma \ref{alt6case}, reveals that $C_X(M)=Z(M)$, $M=R$ and $X\cong 4\circ 2^{1+4}.\Alt(6)$ which is included as Theorem \ref{lem:with one boundintro}\ref{3Case2intro}. Hence, $X$ does not satisfy Hypothesis \ref{hyp: modulehyp1}.
\end{proof}

\begin{lemma}\label{lem:p=2Symp}
If Hypothesis \ref{hyp: modulehyp1} holds then $p=2$.
\end{lemma}
\begin{proof}
By Lemma \ref{lem:FirstSymplecticReduction}, to prove the result we may assume throughout that $p=3$. We remark that by the maximal choice of $\ell$, we must have that $F(X)=O_2(X)Z(X)$.

Suppose first that $|Z(M)|\geq 4$ so that $|Z(M)|=4$ and $M\cong  4\circ 2^{1+6}$ by Lemma \ref{lem:FirstSymplecticReduction}. In particular, $C_X(M)\le R$ and $X/C_X(M)$ embeds as a subgroup of $\Sp_6(2)$. Furthermore, it follows from \cite[Theorem 5.5.5]{GOR} that $\dim_{\FF_9} V=8$ and $M$ acts irreducibly on $V$. Appealing to \cite[Table 8.44]{BHRD}, we see that $N_{\SL_{8}(9)}(M)\cong 4\circ 2^{1+6}.\Sp_6(2)$ is maximal in $\SL_8(9)$. Hence, $C_X(M)\le Z(M)$ so that $F(X)=M\cong 4\circ 2^{1+6}$.

By \cite{ModAt} any faithful, irreducible $\FF_2\Alt(6)$-module of dimension at most $6$ has dimension $4$. Since the $2$-part of the Schur multiplier of $\Alt(6)$ has order $2$, if $\wt X\cong \Alt(6)$ then $C_{M/Z(M)}(X)\ne 1$, against Lemma \ref{lem:NoFix}. If $\wt X\cong \PSL_3(4)$ then a comparison of orders reveals that $X/Z(M)$ lies inside a maximal subgroup of $\Sp_6(2)$ isomorphic to $\Sym(8)$. One can easily verify that no such subgroup of $\Sym(8)$ exists.

Hence, $\wt X$ is isomorphic to one of $\PSU_3(3)$ or $\Ree(3)$, which are included in Theorem \ref{lem:with one boundintro}\ref{3Caseintro}, and so $X$ does not satisfy Hypothesis \ref{hyp: modulehyp1}.

Thus, we continue under the restriction that $|Z(M)|=2$ and $M\cong 2^{1+w}_{\pm}$ where $w\leq 8$. Analyzing the maximal subgroups of $\mathrm{O}_{w}^\epsilon(2)$ as gleaned from \cite{BHRD} when $w\leq 6$ yields that either $\wt X\cong \Alt(6)$ and $w=6$; or $w=8$. In the former case, we have that $X/C_X(M/Z(M))\cong \Alt(6)$. By \cite{ModAt}, $\Alt(6)$ has two non-trivial irreducible $\FF_2$-modules whose dimension is at most $6$. Both are $4$-dimensional. As calculated for Lemma \ref{alt6case}, these modules are self-dual and have $1$-dimensional $1$-cohomology so that $C_{M/Z(M)}(X)\ne 1$, against Lemma \ref{lem:NoFix}.

Hence, $w=8$ so that $\dim_{\FF_3} V=16$ and $M$ acts irreducibly on $V$. We calculate that no subgroup of $\mathrm{O}_{8}^\pm(2)$ has quotient $\PSL_3(4)$, and so we have that $\wt X$ is isomorphic to one of $\Alt(6)$, $\PSU_3(3)$ or $\Ree(3)$. If $\wt X\cong \PSU_3(3)$ we then calculate that $R=C_X(M/Z(M))$. We observe by \cite{ModAt} that $\PSU_3(3)$ has a unique non-trivial irreducible $\FF_2$-module whose dimension is at most $8$, namely its irreducible module of dimension $6$. We calculate in {\sc Magma} \cite{Magma} that this module has $1$-dimensional $1$-cohomology, so that $C_{M/Z(M)}(X)\ne 1$ against Lemma \ref{lem:NoFix}.

Next, appealing to \cite[Table 3.5.C, Table 3.5.E]{KleidmanLiebeck}, we see that we have embeddings of maximal subgroups $2^{1+8}_+.\mathrm{O}_8^+(2)\le \mathrm{O}_{16}^+(3)$ and $2^{1+8}_-.\mathrm{O}_8^-(2)\le \Sp_{16}(3)$. From this, we deduce that $N_{\SL_{16}(3)}(M)\cong 2^{1+8}_{\pm}.\mathrm{O}_8^{\pm}(2)$, $C_X(M)=Z(M)$ and $F(X)=M$. If $M\cong 2^{1+8}_-$ then we associate $M/Z(M)$ with the natural module for $\mathrm{O}_8^-(2)$. We calculate that the only maximal subgroup of $\Omega_8^-(2)$ which fixes no non-trivial subspace of $M/Z(M)$ and has order divisible by $|\Alt(6)|$ or $|\Ree(3)|$ is $\Omega_2^-(2)\times \Omega_6^+(2)\cong 3\times \Alt(8)$. We calculate in {\sc Magma} \cite{Magma} that there are no suitable subgroups with quotient isomorphic to $\Ree(3)$, and that any $\Alt(6)$ subgroup of $\Omega_2^-(2)\times \Omega_6^+(2)$ fixes a non-trivial subspace of $V$, against Lemma \ref{lem:NoFix}.

Hence, we have that $M\cong 2^{1+8}_+$ and we now associate $M/Z(M)$ with the natural module for $\mathrm{O}_8^+(2)$. We remark by \cite[Table 8.50]{BHRD} that $\Sp_6(2)$ is an irreducible maximal subgroup of $\mathrm{O}_8^+(2)$. Then, appealing to \cite[Table 8.28]{BHRD}, $\Ree(3)\cong \PSL_2(8):3\cong \Sp_2(8):3\le \Sp_6(2)$ and we calculate that $\Ree(3)$ is an irreducible subgroup of $\mathrm{O}_8^+(2)$. We deduce in this case that $X\cong 2^{1+8}_+.\Ree(3)$, which is another example in Theorem \ref{lem:with one boundintro}\ref{3Caseintro}. Now, $\Alt(6)\cong \Sp_4(2)'\le \Sp_4(2)\times \Sp_2(2)$ and $\Sp_4(2)\times \Sp_2(2)$ is maximal in $\Sp_6(2)$. We calculate that the restriction of $M/Z(M)$ to this $\Alt(6)$ is a direct sum of two isomorphic $4$-dimensional $\FF_2\Alt(6)$-modules. This gives rise to the group $X\cong 2^{1+8}_+.\Alt(6)$, as in Theorem \ref{lem:with one boundintro}\ref{3Caseintro}. We verify that there are no other suitable candidates in $\mathrm{O}_8^+(2)$, completing the proof.
\end{proof}

\begin{proposition}\label[proposition]{lem:with one bound2}
There are no pairs $(X, V)$ satisfying Hypothesis \ref{hyp: modulehyp1}.
\end{proposition}
\begin{proof}
Since $p=2$ by Lemma \ref{lem:p=2Symp}, we either have that $m=2$ and $\ell\leq 7$; or $(p,m,\ell)=(2,3,3)$. Observe that $\Out(M)\cong \Sp_d(\ell).(\ell-1)$ when $Q\cong \ell^{1+d}_{+}$ and $\Out(M)\cong \Sp_{d-2}(\ell).(\ell-1)$ when $Q\cong \ell^{1+d}_{-}$. Assume $(p,m,\ell)=(2,3,3)$ so that $\frac{d}{2}\leq 2$. Since $7$ divides the order of $\PSL_2(8)$ but does not divide the order of $\Sp_4(3)$ nor of $\Sp_2(3)$, this case does not arise. Hence, $m=2$ so that $\frac{d}{2}\leq 2$ if $\ell=3$ and $\frac{d}{2}=1$ if $\ell\in \{5,7\}$. Since $\PSL_2(4)\cong \Alt(5)$ and $\SL_2(\ell)$ has no subgroup isomorphic to $\Alt(5)$ when $\ell\leq 7$, we must have that $(p,m,\ell)=(2, 2,3)$ and $Q\cong 3^{1+4}_+$. In this case, we verify that the unique faithful irreducible $\FF_2Q$-module has dimension $18>4m(p-1)$ and so this case does not arise. This completes the proof.
\end{proof}

\begin{remark}\label[remark]{rmk: Examples}
That Theorem \ref{lem:with one boundintro} \ref{5Caseintro}, \ref{3Case2intro} and \ref{3Caseintro} arise as genuine examples follows in each case from their embeddings in the relevant symplectic or orthogonal groups as detailed in the above proof, except when $X\cong 4\circ 2^{1+4}.\Alt(6)$ acting irreducibly on a $16$-dimensional $\FF_3$-module. We have included a group satisfying these latter constraints in the ancillary code.
\end{remark}

\begin{hypothesis}\label{hyp: modulehyp2}
Hypothesis \ref{hyp: modulehyp} holds and $M$ is an elementary abelian $\ell$-group.
\end{hypothesis}

\begin{lemma}\label{lem:permClifford}
We have that $V|_M=W_1\oplus \dots \oplus W_b$, $X$ permutes the $b$ components of $V|_M$ transitively, $N_{X}(W_i)R<X$ and $b \geq \pdeg (X/R)$.
\end{lemma}
\begin{proof}
Since $X=O^{\ell}(X)$ and $[T, M]\ne 1$, we deduce that $M/(M\cap Z(X))$ is non-cyclic. By minimality of $M$, we conclude that no maximal subgroup of $M$ is normalized by $X$. Then $V=\bigoplus\limits_{|M:B|=\ell} C_V(B)$ and as $V$ is irreducible, we conclude that $X$ permutes the summands of $V$ transitively. Writing $W_i=C_V(B)$ for some appropriate $B$ and applying Lemma \ref{lem:perm degree}, using that $V$ is a faithful module and $|M|>\ell$, the result holds.
\end{proof}

\begin{proposition}\label{prop:M11andRee}
If $(X, V)$ satisfies Hypothesis \ref{hyp: modulehyp2} then $\wt X\cong \Alt(2p)$.
\end{proposition}
\begin{proof}
Aiming for a contradiction, we assume throughout that $\wt X\not\cong \Alt(2p)$. Set $L:=O^{p'}(O^p(X))$ and $K$ to be the kernel of the permutation action of $R$ on the summands of $V|_M$. Note that $L=X$ unless $\wt X$ is isomorphic to $\Ree(3)$ or $\Sz(32):5$. By Lemma \ref{lem:permClifford}, we have that $\pdeg(\wt L)\leq \pdeg(\wt X)\leq \dim_{\FF_p}V \leq 4m(p-1)$. By Proposition \ref{cor:data} we have that $\pdeg(\wt X) \geq p^m+1$ or $\wt L\cong \PSL_2(8)$ and $\pdeg(\wt L)=\pdeg(\wt X)=9$.

If $\wt L\not\cong \PSL_2(8)$ we have that
\[4m(p-1)\geq p^m+1\geq p^m-1=(p-1)(p^{m-1}+\dots + p+1)\]
so that $4m \geq (p^{m-1}+\dots p+1)$ which implies that $p=2$ and $m\leq 4$; or that $m=2$ and $3\leq p\leq 7$. Furthermore, if $m=2$ and $p=7$ then $4m(p-1)=48<50=p^m+1$. Thus, if $p$ is odd then we deduce that $m=2$ and $p\leq 5$.

If $p=2$ then as $\pdeg(\wt X)\leq \dim_{\FF_2} V \leq 4m$, appealing to Table \ref{tab:data} we see that $\wt X\cong \PSL_2(2^m)$ with $2\leq m\leq 3$. Since $\dim_{\FF_2} V\leq 4m$ and $\pdeg(\PSL_2(2^m))>2m$ we deduce that $V|_{M}$ is a sum of $1$-dimensional modules over $\FF_2$, so a sum of trivial modules for $M$. But then $M$ acts trivially on $V$, a contradiction.

If $p=5$ then as $\pdeg(\wt L)\leq \dim_{\FF_5} V \leq 32$ and appealing to Table \ref{tab:data} we see that $\wt X=\wt L$ is isomorphic to $\PSL_2(25)$. In particular, as $\pdeg(\wt L)=26>16$, we deduce that $R=K$ and $V|_M$ is a sum of $1$-dimensional $\FF_5$-modules. It follows that $R$ is an abelian $2$-group of exponent at most $4$. Then as $V$ is irreducible and $X$ acts transitively on the summands of $V|_R$, we have that $V|_R$ is a direct sum of $26$ $1$-dimensional $\FF_5$-modules. This example is included as Theorem \ref{lem:with one boundintro}\ref{permcase1intro}, and so $X$ does not satisfy Hypothesis \ref{hyp: modulehyp2}.

If $p=3$ then as $\pdeg(\wt L)\leq \dim_{\FF_3} V \leq 16$ and as $\wt L\not\cong \Alt(6)\cong \PSL_2(9)$, appealing to Table \ref{tab:data} we see that $\wt L$ is isomorphic to $\PSL_2(8)$ or $\mathrm{M}_{11}$. In particular, as $\pdeg(\wt L)>8$, we deduce that $R=K$ and $V|_M$ is a sum of $1$-dimensional $\FF_3$-modules. It follows that $R$ is an elementary abelian $2$-group. Finally, if $\wt X\cong \Ree(3)$ then as $V$ is irreducible and $X$ acts transitively on the summands of $V|_R$, we have that $V|_R$ is a direct sum of $9$ $1$-dimensional modules. If $\wt X\cong \mathrm{M}_{11}$ then as $V$ is irreducible and $X$ acts transitively on the summands of $V|_R$, we have that $V|_R$ is a direct sum of either $11$ or $12$ $1$-dimensional modules. These examples are included in Theorem \ref{lem:with one boundintro}\ref{permcase2intro}, and so $X$ does not satisfy Hypothesis \ref{hyp: modulehyp2}.

Thus, if $X$ satisfies Hypothesis \ref{hyp: modulehyp2} then $\wt X\cong \Alt(2p)$, as desired.
\end{proof}

\begin{lemma}\label{lem:Rinertial}
If $(X, V)$ satisfies Hypothesis \ref{hyp: modulehyp2} then $R=\bigcap\limits_{i\in \{1,\dots, b\}} N_X(W_i)$.
\end{lemma}
\begin{proof}
By Proposition \ref{prop:M11andRee}, we have that $\wt X\cong \Alt(2p)$. Throughout, write $V|_M=W_1\oplus \dots \oplus W_b$ and set $K=\bigcap\limits_{i\in \{1,\dots, b\}} N_X(W_i)$ to be the kernel the of permutation action of $X$ on the summands of $V|_M$. Then $K$ is a normal subgroup of $X$ contained in $R$. Aiming for a contradiction, assume throughout that $K<R$. Applying Lemma \ref{lem:perm degree}, we have that $N_X(W_i)R<X$ for each $i$. Since $\dim_{\FF_p} V\leq 8(p-1)$ we must have that $[X :N_X(W_i)R]=2p$, $\dim_{\FF_p} W_i=1$ and $[R: N_R(W_i)]\in\{2,3\}$ for each $i\in\{1,\dots, b\}$. We may then write $V|_R=V_1\oplus \dots \oplus V_r$ where each $V_i$ is a submodule generated by a distinct orbit of $R$ in its action on the set $\{W_i\}$. Hence, $\dim_{\FF_p} V_i\leq 3$. Since $R\norm X$, $X$ respects the orbits of $R$ and as $X$ acts transitively on $\{W_i\}$ it follows that $X$ also acts transitively on the set $\{V_i\}$. Indeed, $V|_R=V_1\oplus \dots \oplus V_{2p}$. We observe $N_R(W_i)$ fixes each $W_j$ in the $R$-orbit of $W_i$. Indeed, we must have that $R/K$ is an elementary abelian $a$-group and $|R/K|\leq a^{2p}$ where $a=[R:N_R(W_i)]$.

Suppose that $p=3$. Since $a\in\{2,3\}$ and $R$ is a $3'$-group, we must have that $a=2$ and $b=12$. Furthermore, as $\dim_{\FF_3} W_i=1$, we have that $K$ is a $2$-group. Hence, $R$ is a $2$-group and so is nilpotent. Set $B_1$ a maximal subgroup of $M$ such that $W_1=C_V(B_1)$. Since $Z(X)$ is contained in a maximal subgroup of $M$, we arrange that $Z(X)\le B_1$. Then $B_1=C_M(W_1)$. Since $K<R$ there are $W_2$ and $B_2\le M$ with $W_1^r=W_2$ and $B_2=C_M(W_2)=B_1^r$. Since $M$ was chosen minimally and $R$ is nilpotent, $[M, R]\le Z(X)$ and $B_2=B_1^r\le B_1Z(X)=B_1$. Hence, $B_1=B_2$ and $W_1=C_V(B_1)=C_V(B_2)=W_2$. But then $R=K$, a contradiction. Hence, $p\geq 5$.

It remains to show that there is $L\le X$ such that $X=LR$ and $L$ is quasisimple. Set $L\le X$ minimal such that $X=LR$ and $O^{p'}(L)=L$. Set $U:=(R/K)/C_{(R/K)}(X)$ and recognize $U$ as an $\FF_a\Alt(2p)$-module. Then $U$ satisfies the hypothesis of Lemma \ref{AltCohomology} and we conclude that $(LK\cap R)/K\le C_{R/K}(X)$. By Schur multiplier considerations, we deduce that $|C_{R/K}(X)\cap LK/K|\leq 2$. If $|C_{R/K}(X)\cap LK/K|=2$ then $LK/K\cong 2\cdot\Alt(2p)$ and $LK/K$ acts on the set $\{W_i\}$. By Lemma \ref{permalt(2p)}, and as $2\cdot \Alt(2p)$ contains no subgroup isomorphic to $\Alt(2p-1)$, we conclude that the minimal faithful permutation representation of $2\cdot\Alt(2p)$ has degree strictly larger than $8(p-1)$, and so we have a contradiction. Hence, $L\cap R\le K$.

Now, the orbits of $O^{p'}(LK)$ on the set $\{W_i\}$ all have length $2p$. Write $\hat{V}$ for the submodule generated by one such orbit, so $\hat{V}$ is an irreducible $\FF_pO^{p'}(LK)$-module. Since $O^{p'}(LK)\le LK<X$ and $|X|$ was chosen minimally such that it does not appear in Theorem \ref{lem:with one boundintro}, we see that the pair $(O^{p'}(LK)/C_{O^{p'}(LK)}(\hat{V}), \hat{V})$ appears in Theorem \ref{lem:with one boundintro}. Since $\dim_{\FF_p} \hat{V}=2p$, there is $\hat{L}\le O^{p'}(LK)$ such that $\hat{L}/C_{\hat{L}}(\hat{V})$ is quasisimple. By minimality, $L=\hat{L}$. Repeating the process for all $L$-orbits on $\{W_i\}$ shows that $L$ is quasisimple. But then $(X, V)$ satisfies Theorem \ref{lem:with one boundintro}\ref{Alt2pCaseintronew} and so $X$ does not satisfy Hypothesis \ref{hyp: modulehyp2}.
\end{proof}

We now show there are no pairs $(X, V)$ satisfying Hypothesis \ref{hyp: modulehyp}. Consequently, Theorem \ref{lem:with one boundintro} holds.

\begin{proof}[Proof of Theorem~\ref{lem:with one boundintro}]
By Proposition \ref{prop:Mdich} and Proposition \ref{lem:with one bound2}, we have that $M$ is elementary abelian and so if Hypothesis \ref{hyp: modulehyp} holds then Hypothesis \ref{hyp: modulehyp2} holds. By Proposition \ref{prop:M11andRee} and Lemma \ref{lem:Rinertial}, writing $V|_M=W_1\oplus \dots \oplus W_b$, we have that $\dim_{\FF_p} W_i\leq 3$, $R=\bigcap\limits_{i\in \{1,\dots, b\}} N_X(W_i)$ and $X$ acts transitively on the set $\{W_i\}$.

Since $X$ acts transitively on the $b$ components of $V|_M$ and $\dim_{\FF_p} V\leq 8(p-1)$, applying Lemma \ref{permalt(2p)} we deduce that $b=2p$; or we calculate that $b=10$ or $15$ when $p=3$. Setting $V_i:=W_i$, to obtain a final contradiction, it remains to show that there is $L\le X$ with $X=LR$ and $L$ quasisimple, for then $(X, V)$ satisfies Theorem \ref{lem:with one boundintro}\ref{Alt2pCaseintronew} and so $X$ does not satisfy Hypothesis \ref{hyp: modulehyp2}. If $b>2p$ then $p=3$ and  applying Lemma \ref{alt6case}, as $\dim_{\FF_3} V \leq 16$, we see that there is $L\le X$ with $X=LR$ and $L$ quasisimple.

Hence, $b=2p$. Assume that there is $Y$ a proper subgroup of $X$ with $X=YR$. Set $\hat{R}:=R\cap O^{p'}(Y)$. If $V|_{O^{p'}(Y)}$ is irreducible, then the pair $(O^{p'}(Y), V)$ satisfies the hypothesis of Theorem \ref{lem:with one boundintro}. By the minimality of $X$, we deduce that $O^{p'}(Y)$ appears in Theorem \ref{lem:with one boundintro} and $O^{p'}(Y)/\hat{R}\cong \Alt(2p)$. In cases Theorem \ref{lem:with one boundintro}\ref{5Caseintro},\ref{3Case2intro} we verify that $H^2(\Alt(2p), \hat{R}/Z(\hat{R}))=0$ using Lemma \ref{AltCohomology} when $p=5$ and {\sc Magma} \cite{Magma} when $p=3$. Hence, there is $L\le O^{p'}(Y)$ with $L$ quasisimple and $O^{p'}(Y)=L\hat{R}$. But then $L\le X$ and $X=LR$ and we have uncovered the desired candidate for $L$. If $V|_{O^{p'}(Y)}$ is reducible then for each $O^{p'}(Y)$-composition factor $W$ of $V$, we see that $O^{p'}(Y)/C_{O^{p'}(Y)}(W)$ appears in Theorem \ref{lem:with one boundintro}. Letting $L$ be a subgroup of $O^{p'}(Y)$ chosen minimally such that $O^{p'}(Y)=L\hat{R}$, we see that $L/C_L(W)$ is quasisimple for each $W$. Since $V$ is a faithful module, we conclude that $L$ itself is quasisimple. Then $L\le X$ and $X=LR$ and again we have uncovered the desired candidate for $L$. Hence, no such $Y$ exists.

By the Frattini argument, we have that $X=N_X(D)R$ for each $D\in\syl_t(R)$ and each $t\in \Pi(R)$. Then, by the previous paragraph, we have that $X=N_X(D)$ for all such $D$, and so $R=F(X)$ is nilpotent. Then for each $X$-composition factor $U$ contained in $R$, $U$ is an $\FF_t\Alt(2p)$-module for some prime $t\in \Pi(R)$. Since $X$ acts transitively on the set $\{V_i\}$, it follows that every $X$-composition factor inside of $R$ is isomorphic as an $\FF_t\Alt(2p)$-module to an irreducible factor of the $\FF_t\Alt(2p)$ permutation module of dimension $2p-c$ where $c=(2,t)$. If $p\geq 5$ then an appeal to Lemma \ref{AltCohomology} and a consideration of the Schur multiplier of $\Alt(2p)$ provides a subgroup $L<X$ with $X=LR$, a contradiction.

Hence, $p=3$. Note that as $\dim_{\FF_3} W_i\leq 3$, it follows that $\Pi(R)\subseteq \Pi(\GL_3(3))\setminus \{3\}=\{2,13\}$. We observe that the $3'$ part of the Schur multiplier of $\Alt(6)$ is a $2$-group, and that the natural $\FF_{13}\Alt(6)$ permutation module has trivial $2$-cohomology, and so if $O_{13}(X)\ne 1$ then as $O^{3'}(X)=X$, we uncover $Y<X$ with $X=YR$, a contradiction. Hence, $R$ is a $2$-group. We examine the Sylow $2$-subgroups of $\GL_3(3)$, observing that they all lie in a maximal subgroup of shape $\GL_2(3)\times 2$. In particular, a Sylow $2$-subgroup of $\GL_3(3)$ acts reducibly on a $3$-dimensional module. Since $V$ is irreducible, and $X$ acts transitively on the components of $V|_R$, we conclude that $\dim_{\FF_3} W_i\leq 2$ and $\dim_{\FF_3} V \leq 12$. Since the sectional rank of a Sylow $2$-subgroup of $\GL_2(3)$ is $2$, we deduce that $|R/\Phi(R)|\leq 2^{12}$. Hence, $X/\Phi(R)$ satisfies the hypothesis of Lemma \ref{alt6case}. Indeed, $X/\Phi(R)$ satisfies case (ii) of Lemma \ref{alt6case}, from which we deduce that $X/\Phi(R)$ corresponds to $L$ in Lemma \ref{alt6case} and $N_X(W_1)/\Phi(R)$ satisfies the role of $P$. Thus, $N_X(W_1)/\Phi(R)$ contains no subgroup of index $2$. Note that $\GL_2(3)$ is solvable and so contains no subgroups with quotient $\Alt(5)$. We deduce that $N_X(W_i)=C_X(V_i)R$. Since $N_X(W_i)$ has no subgroups of index $2$, we conclude that $N_X(V_i)=C_X(V_i)\Phi(R)\le C_X(V_i)\Phi(N_X(V_i))$ so that $N_X(V_i)=C_X(V_i)$. But then $R\le C_X(V_i)$ for all $i$, impossible since $1\ne R\norm X$ and $V$ is faithful. Hence, no $X$ satisfies Hypothesis \ref{hyp: modulehyp2} and Theorem \ref{lem:with one boundintro} holds.
\end{proof}

\begin{remark}
We detail some instances where outcomes Theorem \ref{lem:with one boundintro} \ref{permcase1intro}-\ref{Alt2pCaseintronew} occur. We begin with \ref{Alt2pCaseintronew}. In a maximal subgroup of $\Sym(2p^2)$ of shape $\Sym(p)\wr \Sym(2p)$, we identify the group $H:=A^{2p}:\Sym(2p)$, where $A$ is a Sylow $p$-normalizer in $\Sym(p)$. Then $O_p(H)$ is elementary abelian of order $p^{2p}$ and admits a faithful action from $H/O_p(H)\cong (p-1)^{2p}:\Sym(2p)$. Hence, $O^{p'}(H/O_p(H))$ is not quasisimple, has quotient isomorphic to $\Alt(2p)$ and acts faithfully on $O_p(H)$, which we may regard as an $\FF_p$-module of dimension $2p$.

For \ref{permcase1intro} and \ref{permcase2intro}, let $n$ be the minimal faithful permutation degree of $\wt X$ so that $n\leq 4m(p-1)$. Consider the group $\Sp_{2n}(r)$ where $|r-1|_p=p$, which has Weyl group $C_2^n: \Sym(n)$. We get a group $C_2^n: H$ where $H\cong \wt X$ and this group acts on a maximal torus of rank $n$ in $\Sp_{2n}(r)$. Hence, we have a faithful action on an elementary $p$-group of order $p^n$ and so $C_2^n:H$ has a faithful $\FF_p$-module of dimension $n$.
\end{remark}

\section{Reductions for Theorem \ref{ModResult}}\label[section]{ReductionSec}

In this section, we reduce the proof of Theorem \ref{ModResult} to quasisimple groups of Lie type in characteristic $p$. This is the point where we begin to fully utilize the restrictions on the Jordan form of $p$-elements on the relevant modules. Throughout this section, we use the notation established in Notation \ref{SNotation}.

\begin{lemma}\label[lemma]{dimbound}
Suppose that $G$ is a group and $V$ is a faithful $\FF_pG$-module. Let $x \in G$ have order $p$ and assume that $x$ is $k$-active on $V$. Then $\dim_{\FF_p}C_V(x) \geq \dim_{\FF_p}V -k(p-1)$. In particular, if $G$ is generated by $s$ conjugates of $x$, then $\dim_{\FF_p }V/C_V(G) \leq sk(p-1)$.
\end{lemma}
\begin{proof}
We see that $\dim_{\FF_p}C_V(x)$ coincides with the number of Jordan blocks that $x$ has when acting on $V$. Since each non-trivial Jordan block has dimension at most $p$, we have $\dim_{\FF_p} C_V(x) \geq \dim_{\FF_p} V -(p-1)k$. If $G$ is generated by a set $\mathcal{Y}$ of $s$ conjugates of $x$, then as $C_V(G)=\cap_{y\in \mathcal{Y}} C_V(y)$, we deduce that $\dim_{\FF_p}V/C_V(G)\leq sk(p-1)$.
\end{proof}

\begin{lemma}\label[lemma]{blocks lemma}
Suppose that $X\in \mathcal{S}$ and let $x\in X$ have order $p$. Let $V$ be a faithful $\FF_pX$-module and assume that $x$ is $k$-active on $V$ for some $k\leq m$. Then either $\wt X \cong \Alt(2p)$, or $O^{p'}(O^p(X))$ is quasisimple.
\end{lemma}
\begin{proof}
By Proposition \ref{generation} we know that $O^p(\wt X)$ is contained in a subgroup generated by the images of at most $3$ conjugates of $x$. Assume throughout that $O^{p'}(O^p(X))$ is not quasisimple.

Suppose that $x\in O^p(X)$ and let $K\le X$ be such that $K$ is generated by at most $3$ conjugates of $x$ and $O^p(\wt X)=\wt K$. By Lemma \ref{dimbound}, we have that $\dim_{\FF_p} V/C_V(K) \leq 3m(p-1)<4m(p-1)$. We apply Theorem \ref{lem:with one boundintro} to see that either $K=O^{p'}(O^p(K))$ is quasisimple or $\wt K\cong \Alt(2p)$. To prove the result in this case we assume, aiming for contradiction, that $O^{p'}(O^p(X))$ is not quasisimple, $K$ is a proper quasisimple subgroup of $O^{p'}(O^p(X))$ and $\wt X$ is not isomorphic to an alternating group. Let $L=\langle K, x^g\rangle$ for some $g\in X$ such that $K<L$, so that $L=O^{p'}(L)\le O^{p'}(O^p(X))$. By assumption, we have that $O^{p'}(O^p(L))$ is not quasisimple. Still, we have that $\dim_{\FF_p} V/C_V(L) \leq 4m(p-1)$ by Lemma \ref{dimbound} and so Theorem \ref{lem:with one boundintro} applies to $L/C_L(W)$ where $W$ is a non-trivial $L$-composition factor of $V$.

It follows that $L$ satisfies case \ref{3Caseintro}, \ref{permcase1intro} or \ref{permcase2intro} of Theorem \ref{lem:with one boundintro}. By Proposition \ref{generation}, unless $\wt X\cong \PSU_3(3)$ and $x\in Z(T)$, we get that $K$ is generated by two conjugates of $x$ and $\dim_{\FF_p} V/C_V(K)\leq 2m(p-1)$ by Lemma \ref{dimbound}. Then $\dim_{\FF_p} V/C_V(L)\leq 3m(p-1)$ and Theorem \ref{lem:with one boundintro} gives a contradiction. Hence, $\wt X\cong \PSU_3(3)$, $x\in Z(T)$, $K\cong \PSU_3(3)$ and $\dim_{\FF_3} V/C_V(K)\leq 12$.

Appealing to \cite{ModAt}, we see that $\PSU_3(3)$ has three non-trivial irreducible $\FF_3$-representations of dimension at most $12$: the natural module of dimension $6$; any non-trivial composition factor of $(M\otimes M^*)|_{\FF_3}$, where $M$ is the natural $\FF_9\SU_3(3)$-module, of dimension $7$; and the symmetric square of the natural module of dimension $12$. We verify in {\sc Magma} \cite{Magma} that $x$ has at least three non-trivial Jordan blocks on the $7$-dimensional and $12$-dimensional modules. These assertions are established independently in Section \ref{SU3Section}.

We further calculate that in any non-split extension of the $6$-dimensional by itself, $x$ has $6$ Jordan blocks of size $2$; and we see that if $V/C_V(K)$ is a direct sum of two $6$-dimensional modules then $x$ has $4$ Jordan blocks of size $2$ and $4$ trivial Jordan blocks. Hence, $V/C_V(K)$ has exactly one non-trivial composition factor, and this arises as a single $6$-dimensional module. Finally, once again appealing to {\sc Magma} \cite{Magma}, we see that there is no non-trivial extension of the $6$-dimensional module by a trivial module and we conclude that $x$ has only $2$ non-trivial Jordan blocks in its action on $V/C_V(K)$. Since $x\in Z(T)$, we see that $|V/C_V(x)|=3^2$. But $L$ is generated by $4$ conjugates of $x$ and we see that $|V/C_V(L)|\leq 3^8$, a contradiction by Theorem \ref{lem:with one boundintro}.

Hence, $x\not\in O^p(X)$ so that $\wt X$ is isomorphic to $\Ree(3)$ or $\Sz(32):5$. Then by Proposition \ref{generation} we see that $\wt X$ is generated by two conjugates of $x$ and we set $K=\langle x, x^g\rangle$ for $g\in X$ with $\wt K\cong \wt X$. Then Lemma \ref{dimbound} implies that $\dim_{\FF_p} V/C_V(K) \leq 2m(p-1)$ and Theorem \ref{lem:with one boundintro} implies that $O^{p'}(O^p(K))$ is quasisimple. Moreover, for any $h\in X$ we have that $\dim_{\FF_p} V/C_V(\langle K, x^h\rangle)\leq 3m(p-1)$ so that $O^{p'}(O^p(\langle K, x^h\rangle))$ is quasisimple and we conclude that $O^{p'}(O^p(K))=O^{p'}(O^p(X))$ is quasisimple, the desired contradiction. This completes the proof.
\end{proof}

\begin{proposition}\label[proposition]{blocks lemma2}
Suppose that $X\in\mathcal{S}$ and let $x\in X$ have order $p$. Let $V$ be a faithful $\FF_pX$-module and assume that $x$ is $k$-active on $V$ for some $k\leq m$. Then either
\begin{enumerate}
\item $p$ is arbitrary, $m=n>1$ and $X/Z(X)\cong\PSL_2(p^n)$;
\item $p$ is odd, $m=2n$ and $X/Z(X)\cong \PSU_3(p^n)$;
\item $p=3$, $m=2$, $O^{3'}(O^3(X))\cong \PSL_2(8)$ and $\wt X\cong \Ree(3)$;
\item $p=3$, $m=2n$, $n>1$ and $X\cong \Ree(3^n)$;
\item\label{M11} $p=3$, $n=m=2$ and $X/Z(X)\cong \mathrm{M}_{11}$;
\item\label{L34} $p=3$, $n=m=2$ and $X/Z(X)\cong \PSL_3(4)$; or
\item\label{A2pcase} $p$ is odd, $n=m=2$ and $X/O_{p'}(X)\cong \Alt(2p)$.
\end{enumerate}
\end{proposition}
\begin{proof}
By Lemma \ref{blocks lemma}, we may assume that $O^{p'}(O^p(X))$ is quasisimple else \ref{A2pcase} is satisfied.

Suppose that $X/Z(X) \cong \PSU_3(2^n)$. Then $m=n$ and $x$ has at most $n$ non-trivial Jordan blocks on $V$ over $\FF_2$. By Proposition \ref{generation} and Lemma \ref{dimbound}, we have that $\dim_{\FF_2} V/C_V(X)\leq 3n$, against Table \ref{tab:data}. Suppose that $X=O^{2'}(O^2(X))\cong \Sz(2^n)$. Then $m=n$ and $x$ has at most $n$ non-trivial Jordan blocks on $V$ over $\FF_2$. By Proposition \ref{generation} and Lemma \ref{dimbound}, we have that $\dim_{\FF_2} V/C_V(X)\leq 3n$, against Table \ref{tab:data}.

Hence, to complete the proof, we may assume that $m=2$. Applying Proposition \ref{generation} and Lemma \ref{dimbound}, we deduce that $\dim_{\FF_p} V/C_V(O^{p'}(O^p(X)))\leq 6(p-1)$. Appealing to Table \ref{tab:data}, it remains to examine the case $X/Z(X)\cong \mathrm{McL}$ when $p=5$. But in this case, by Proposition \ref{generation} we have that $X=\langle x, x^g\rangle$ for some $g\in X$ so that Lemma \ref{dimbound} gives $\dim_{\FF_p} V/C_V(X)\leq 4(p-1)=16<21$, against Table \ref{tab:data}.
\end{proof}

We quickly handle cases \ref{L34} and \ref{M11}, which give \ref{M11Main}, \ref{L34Main1} and \ref{L34Main2} of Theorem \ref{ModResult}.

\begin{proposition}\label[proposition]{M11Prop}
Suppose that $X/Z(X)\cong \mathrm{M}_{11}$ and $V$ is a faithful $\FF_3X$-module. Assume that $x\in X$ has order $3$ and assume that $x$ is $k$-active on $V$ for some $k\leq 2$. Then $X\cong \mathrm{M}_{11}$ and $W:=[V, X]/C_{[V, X]}(X)$ is either the code or cocode module of dimension $5$.
\end{proposition}
\begin{proof}
Since the Schur multiplier of $\mathrm{M}_{11}$ is trivial, we have that $X\cong \mathrm{M}_{11}$. Since $X$ is generated by two conjugate $3$-elements by Proposition \ref{generation}, each of which has at most two non-trivial Jordan blocks, and as $3$-elements act cubically on $V$, we have that $|V/C_V(X)|\leq 3^8$. Then \cite{ModAt} implies that $V$ contains exactly one non-trivial composition factor, so that $W:=[V, X]/C_{[V, X]}(X)$ is irreducible, and $W$ is either the code module or the cocode module.
\end{proof}

\begin{proposition}\label[proposition]{PSL34Prop}
Suppose that $X/Z(X)\cong \PSL_3(4)$ and $V$ is a faithful $\FF_3X$-module. Assume that $x\in X$ has order $3$ and assume that $x$ is $k$-active on $V$ for some $k\leq 2$. Then either
\begin{enumerate}
    \item $X\cong 2\cdot\PSL_3(4)$, $V=[V, X]\oplus C_V(X)$ and $\dim_{\FF_3} [V, X]=6$; or
    \item $X\cong 4\cdot\PSL_3(4)$, $V=[V, X]\oplus C_V(X)$ and $\dim_{\FF_3} [V, X]=8$.
\end{enumerate}
\end{proposition}
\begin{proof}
Since $X$ is generated by two conjugate $3$-elements by Proposition \ref{generation}, each of which has at most two non-trivial Jordan blocks, and as elements of order $3$ act cubically on $V$, we have that $|V/C_V(X)|\leq 3^8$. Then \cite{ModAt} implies that $V$ contains exactly one non-trivial composition factor and either $X\cong 2\cdot \PSL_3(4)$ and $[V, X]/C_{[V, X]}(X)$ is $6$-dimensional; or $X\cong 4b.\PSL_3(4)$ and $[V, X]/C_{[V, X]}(X)$ is $8$-dimensional (where $4b\cdot \PSL_3(4)$ is a notation following the ATLAS convention). In either case, $Z(X)$ is non-trivial and $C_V(X)=C_V(Z(X))$. Thus, by coprime action, we have $V=[V, X]\oplus C_V(X)$ and the result holds.
\end{proof}

We now show that if case \ref{A2pcase} of Proposition \ref{blocks lemma2} holds, then $X$ satisfies \ref{Alt2pCaseintronew} of Theorem \ref{ModResult}.

\begin{proposition}\label[proposition]{Alt2pp>5New}
Suppose that $X/O_{p'}(X)\cong \Alt(2p)$ where $p\geq 3$ is a prime. Let $V$ be a faithful $\FF_pX$-module and $x\in X$ have order $p$. Assume $x$ is $k$-active on $V$ for some $k\leq m$, and $x \in S\in \Syl_p(X)$. Then there is $S\le L\le X$ such that either $L\cong \Alt(2p)$; $p=3$ and $L\cong \SL_2(9)$; or $p=5$ and $L\cong 2\cdot\Alt(10)$.

If $p\geq 5$ then every non-trivial composition factor of $V|_L$ is described by Lemma \ref{Alt2pModBound}. Furthermore, unless $L\cong \Alt(2p)$ and $x$ is a $p$-cycle, we have that $[V,L]/C_{[V,L]}(L)$ is irreducible.
\end{proposition}
\begin{proof}
Note that if the first part of the proposition is satisfied, then the second part of the proposition holds by Proposition \ref{CompFactorsAlt2p}. Let $(X, V)$ be a counterexample to the first part of the proposition with $|X|+|V|$ minimal. By minimality, we immediately see that $C_V(X)=0$ and $V=[V, X]$. By Proposition \ref{generation}, there is $g, h\in X$ such that $X=\langle x, x^g, x^h\rangle O_{p'}(X)$ and by minimality, we see that $X=\langle x, x^g, x^h\rangle$, and $X$ is not quasisimple. Since $x$ has at most two non-trivial Jordan blocks in its action on $V$, we have that $\dim_{\FF_p} V\leq 6(p-1)<6p-2$ by Lemma \ref{dimbound}.

If $V$ is reducible, then for $W$ a non-trivial $\FF_pX$-submodule of $V$, we see that $X/C_X(W)$ is quasisimple by minimality of $X$. Moreover, $X/C_X(V/W)$ is either trivial or also quasisimple by minimality and we conclude that $X$ is quasisimple, a contradiction. Hence, $V$ is irreducible. Applying Theorem \ref{lem:with one boundintro} and using that $X$ is a minimal counterexample, we either have that $X$ is isomorphic to $4\circ 2^{1+4}.\Alt(6)$ or $4\circ 2^{1+8}. \Alt(10)$. We calculate in {\sc Magma} \cite{Magma} when $p=3$, and appeal to \cite[Corollary 1]{premet} when $p=5$, to see that $O_2(X)/Z(O_2(X))$ has trivial $2$-cohomology as an $\FF_2\Alt(2p)$-module and so there is $L\le X$ with $LZ(O_2(X))/Z(O_2(X))\cong \Alt(2p)$, a contradiction as $X$ is a minimal counterexample.
\end{proof}

\begin{remark}
When $p=3$, the composition factors of $V|_L$ will be described in Proposition \ref{SL2qmods}.
\end{remark}

\begin{example}
Let $X\cong 4\circ 2^{1+4}.\Alt(6)$ and $V$ be a faithful $8$-dimensional $\FF_3X$-module. Taking $L\le X$ with $L\cong \SL_2(9)$ we see that $V|_L$ is indecomposable with two composition factors. Regarding each of these factors as $\FF_3\SL_2(9)$-modules, they are both isomorphic to natural modules for $\SL_2(9)$. One can see an action of $\SL_2(9)$ of this sort in a parabolic subgroup of $\mathrm{G}_2(9)$. In particular, there are two classes of $3$-elements in $L$: one class which acts on $V$ with four $\FF_3$ Jordan blocks of size two, and another class which acts on $V$  with two Jordan blocks of size three and two trivial Jordan blocks over $\FF_3$.
\end{example}

To complete the proof of Theorem \ref{ModResult}, it remains to describe the $k$-active faithful $\FF_pX$-modules, where $k\leq m_p(X)$ and $X/Z(X)$ is isomorphic to $\PSL_2(p^n)$, $\PSU_3(p^n)$ or $\Ree(3^n)$. We calculate these modules in the remaining sections of this work.

\section{The $\SL_2(p^n)$ case}\label[section]{subsection SL2modules}

In this section, we describe the faithful $\FF_p\SL_2(p^n)$-modules $V$ and the non-trivial $p$-elements $x\in \SL_2(p^n)$ for which $x$ is $k$-active on $V$ where $k\leq n$.

Letting $X=\SL_2(p^n)$, we have that $\FF_{p^n}$ is a splitting field for $X$ by \cite[Proposition 5.4.4]{KleidmanLiebeck}. Following \cite{poly}, we write $V_i(p^n)$ for the $\FF_{p^n}X$-modules of homogeneous polynomials in two commuting variables of degree $i$ with the \emph{natural} action of $\SL_2(p^n)$, so that $\dim_{\FF_{p^n}} V_i(p^n) = i+1$. These are the ($p$-restricted) \emph{basic modules} for $\FF_{p^n}X$. The Steinberg Tensor Product Theorem gives that each irreducible $\FF_{p^n}X$-module arises as a tensor product of Galois twists of these modules. In what follows (in this section and the following section) it is important to realize that for $V$ a $\FF_{p^n}X$-module and $\sigma$ an element of $\Aut(\FF_{p^n})$, we have that $V\cong V^\sigma$ as $\FF_{p}X$-modules.

\begin{lemma}\label[lemma]{basic blocks}
Suppose that $1 \leq i \leq p-1$, $X=\SL_2(p^n)$ and $H\le X$ is a cyclic group of order $p$. Then ${V_i(p^n)}|_H= nJ_{i+1}$ as an $\FF_pH$-module. In particular, $V_i(p^n)$ is $m_p(X)$-active when regarded as an $\FF_{p}X$-module.
\end{lemma}
\begin{proof}
Write elements of $V_i(p^n)$ as homogeneous polynomials of degree $i$ in two commuting variables $y$ and $z$. Without loss of generality, we may assume that for $h$ a generator of $H$, we have $h= \left(\begin{smallmatrix} 1&0\\1&1\end{smallmatrix}\right)$. We may arrange that $z^i.h=(y+z)^i$ and $y^i.h = y^i$. We calculate, by computing the commutators of the basis vectors, that
\[[V_i(p^n),h]=\langle y^{i-j}z^j\mid 0\leq j \leq i-1\rangle_{\FF_{p^n}}\] which has dimension $i$. Thus $V_i(p^n)$ is indecomposable as an $\FF_{p^n}H$-module and restricting to $\FF_p$ proves the claim.
\end{proof}

\begin{remark}
Another proof of Lemma \ref{basic blocks} comes from recognizing that $V_i(p^n)=\Sym^i(V_1(p^n))$, $V_1(p^n)|_H=J_2(p^n)$ and applying the methods in Section \ref{JordanSec}.
\end{remark}

\begin{proposition}\label[proposition]{SL2qmods}
Suppose that $p$ is a prime, $X=\SL_2(p^n)$, $S\in\syl_p(X)$ and $W$ is a non-trivial irreducible $\FF_pX$-module. Let $\sigma$ be a generator of $\Aut(\FF_{p^n})$. Let $x \in S$ have order $p$, and assume that $x$ is $k$-active on $W$ for some $k\leq n$. Then, setting $H=\gen{x}$, one of the following holds:
\begin{enumerate}
\item\label{basic} $\End_{\FF_pX}(W) =\FF_{p^n}$, $W\cong V_i(p^n)|_{\FF_p}$ for some $1\leq i \leq p-1$ considered as an $(i+1)n$-dimensional $\FF_pX$-module and $W|_H= nJ_{i+1}$.
\item\label{1times1} $p$ is odd, $\End_{\FF_pX}(W) =\FF_{p^n}$, $W= (V_1(p^n)\otimes V_1(p^n)^{\sigma^j})|_{\FF_p}$ for some $j\in\{1,\dots, n-1\}$ with $j \ne \frac{n}{2}$, considered as a $4n$-dimensional $\FF_pX$-module and $W|_H= nJ_1\oplus nJ_3$.
\item $n$ is even, $\End_{\FF_pX}(W) =\FF_{p^{\frac{n}{2}}}$, and either
    \begin{enumerate}
    \item\label{orthogonal} $W$ is an irreducible summand, of dimension $2n$ over $\FF_p$, of $(V_1(p^n)\otimes V_1(p^n)^{\sigma^{\frac{n}{2}}})|_{\FF_p}$ and \[W|_H=
        \begin{cases}
            nJ_2 & p=2 \\
            \frac{n}{2}J_1\oplus \frac{n}{2}J_3 & p\geq 3
        \end{cases}; \,\,\text{or}\]
    \item\label{2times2}  $p \geq 5$, $W$ is an irreducible summand, of dimension $\frac{9n}{2}$ over $\FF_p$, of $(V_2(p^n)\otimes V_2(p^n)^{\sigma^{n/2}})|_{\FF_p}$ and
            \[W|_H= \frac{n}{2}J_1\oplus\frac{n}{2}J_3\oplus\frac{n}{2}J_5.\]
    \end{enumerate}
\item\label{triality} $p$ is odd, $n$ is divisible by $3$, $\End_{\FF_pX}(W) =\FF_{p^{n/3}}$, $W$ is any irreducible summand of $(V_1(p^n)\otimes V_1(p^n)^\tau\otimes V_1(p^n)^{\tau^2})|_{\FF_p}$, where $\tau = \sigma^{n/3}$, considered as an $\frac{8n}{3}$-dimensional $\FF_pX$-module and \[W|_H=
        \begin{cases}
            \frac{n}{3}J_2\oplus \frac{2n}{3}J_3 & p=3\\
            \frac{2n}{3}J_2\oplus \frac{n}{3}J_4 & p \geq 5
            \end{cases}.\]
\item \label{4times1} $p\geq 5$, $n$ is divisible by $4$, $\End_{\FF_pX}(W) =\FF_{p^{n/4}}$, $W$ is any irreducible summand of $(V_1(p^n)\otimes V_1(p^n)^\tau\otimes V_1(p^n)^{\tau^2}\otimes V_1(p^n)^{\tau^3})|_{\FF_p}$, where $\tau = \sigma^{n/4}$, considered as a $4n$-dimensional $\FF_pX$-module and \[W|_H= \frac{n}{2}J_1\oplus \frac{3n}{4}J_3\oplus \frac{n}{4}J_5.\]
\end{enumerate}
\end{proposition}
\begin{proof}
Let $\mathbb{K}=\FF_{p^n}$ and write $H=\langle x\rangle$ which is cyclic of order $p$. We have that $\mathbb{K}$ is a splitting field for $\SL_2(p^n)$. Hence every irreducible module can be written over this field. Let $W$ be an irreducible $\FF_pX$-module, set $\mathbb{L}= \End_{\FF_pX}(W)$ and assume that $\mathbb{L}=\FF_{p^\ell}$. Then we may regard $W$ as an $\mathbb{L}X$-module and it has $\mathbb{L}$-dimension $d:=\dim_{\FF_p} W/\ell$. Now $\bar{W}=W\otimes_{\mathbb L}\mathbb K$ is an irreducible $\mathbb{K}X$-module of $\mathbb{K}$-dimension $d$.

Write $\bar{W}= \bigoplus_{i=1}^k n_iJ_i(\mathbb{K})$ as a $\mathbb{K}H$-module. When $W$ is considered as an $\mathbb{L}H$-module it decomposes as $W|_H= \bigoplus_{i=1}^k n_iJ_i(\mathbb{L})$. Then as each $J_i(\mathbb L)$ when considered as a $\FF_pH$-module breaks as a direct sum of $\ell$ copies of $J_i$, if $W$ has at most $n$ non-trivial Jordan blocks as an $\FF_pH$-module, we conclude that $\bar{W}$ has at most $n/\ell $ non-trivial Jordan blocks over $\mathbb{K}$.

By the Steinberg Tensor Product Theorem \cite[Corollary 2.8.6]{GLS3} we may write
\begin{eqnarray}
\bar{W} & = & W_0\otimes W_1^\sigma \otimes \dots \otimes W_{n-1}^{\sigma^{n-1}}\label{STP1}
\end{eqnarray}
where $W_0, \dots, W_{n-1}$ are basic, not necessarily distinct, $\mathbb{K}X$-modules. By \cite[Remark 2.8.8]{GLS3}, $\bar{W}$ can be written over the subfield $\mathbb{L}$ if and only if for $\sigma^{\ell} \in \Gal(\mathbb{K}/\mathbb{L})\le \Gal(\mathbb{K}/\FF_p)=\Aut(\mathbb{K})$, we have that $\bar{W} = \bar{W}^{\sigma^{\ell}}$. Since the field of definition of the basic modules is $\mathbb{K}$, this means that there are at least $n/\ell$ non-trivial factors in the tensor decomposition of $\bar{W}$. Furthermore, each non-trivial factor in the tensor product expansion of $\bar{W}$ contains an $\mathbb{K}H$-submodule isomorphic to $J_2(\mathbb{K})$. Applying Lemma \ref{tensor powers}, since $\bar{W}$ has at most $n/\ell$ non-trivial Jordan blocks for $H$ over $\mathbb{K}$, we see that $n/\ell \leq 4$.

Suppose first that $\ell = n$. Then $\bar{W}$ has a unique non-trivial Jordan block. Comparing with Lemma \ref{tensor powers} we see that in the decomposition described in (\ref{STP1}), there are at most two non-trivial factors. If this decomposition has only one non-trivial factor, then $W$ is a basic module and \ref{basic} holds. So assume that $\bar{W}$ contains two non-trivial tensor factors, written $\bar{W}=U_1\otimes U_2$. As an $\mathbb{K}H$-module, upon considering $U_1\otimes D$ where $D$ is a $\mathbb{K}H$-submodule of $U_2$ isomorphic to $J_2(\mathbb{K})$, applying Lemma \ref{tensor powers} we see that $p$ is odd and $U_1|_H=J_2(\mathbb{K})$. By symmetry, we see that both $U_1$ and $U_2$ are $2$-dimensional and there are no other factors so that $\bar{W}=V_1(p^n)^{\sigma^j}\otimes V_1(p^n)^{\sigma^k}$ where ${j-k}\ne \frac{n}{2}$ modulo $n$, and \ref{1times1} holds.

Suppose that $n/\ell =2$. Hence, $\bar{W}$ has at most two non-trivial Jordan blocks. Considering the tensor product of $\mathbb{K}H$-modules each of which is isomorphic to $J_2(\mathbb{K})$ and applying Lemma \ref{tensor powers}, we see that in the decomposition of $\bar{W}$ given in (\ref{STP1}) there are at most two non-trivial factors. Since $n/\ell =2$ we must have that $\bar{W}=V_i(p^n)^{\sigma^j}\otimes V_i(p^n)^{\sigma^{j+n/2}}$ for some $j\in \{0,\dots, n-1\}$. Applying Lemma \ref{234 tensor}\ref{4tensor}, we deduce that $i\in \{1,2\}$. Moreover, by Lemma \ref{234 tensor}\ref{3tensor} we see that if $p=3$ then $i=1$. This gives (iii)(a) and (iii)(b).

Suppose that $n/\ell =3$. Hence, $\bar{W}$ has at most three non-trivial Jordan blocks. Considering the tensor product of $\mathbb{K}H$-modules each of which is isomorphic to $J_2(\mathbb{K})$ and applying Lemma \ref{tensor powers}, we see that in the decomposition of $\bar{W}$ given in (\ref{STP1}) there are at most three non-trivial factors. Since $n/\ell =3$ we must have that $\bar{W}=V_i(p^n)^{\sigma^j}\otimes V_i(p^n)^{\sigma^{j+n/3}}\otimes V_i(p^n)^{\sigma^{j+2n/3}}$ for some $j\in \{0,\dots, n-1\}$ and applying \ref{3tensor} and \ref{4tensor} of Lemma \ref{234 tensor}, we must have that $i=1$, $p$ is odd, and \ref{triality} holds.

Finally, suppose that $n/\ell=4$. Hence, $\bar{W}$ has at most four non-trivial Jordan blocks. Considering the tensor product of $\mathbb{K}H$-modules each of which is isomorphic to $J_2(\mathbb{K})$ and applying Lemma \ref{tensor powers}, we see that in the decomposition of $\bar{W}$ (\ref{STP1}) there are at most four non-trivial factors. Since $n/\ell =4$ we must have that $\bar{W}=V_i(p^n)^{\sigma^j}\otimes V_i(p^n)^{\sigma^{j+n/4}}\otimes V_i(p^n)^{\sigma^{j+n/2}}\otimes V_i(p^n)^{\sigma^{j+3n/4}}$ for some $j\in \{0,\dots, n-1\}$. Applying Lemma \ref{tensor powers} we see that $p\geq 5$. If $i\geq 2$, then $\bar{W}|_H$ has a submodule isomorphic to \[J_3(\mathbb{K})\otimes J_3(\mathbb{K})\otimes J_3(\mathbb{K})\otimes J_3(\mathbb{K})=(J_1(\mathbb{K})\oplus J_3(\mathbb{K})\oplus J_5(\mathbb{K}))\otimes (J_1(\mathbb{K})\oplus J_3(\mathbb{K})\oplus J_5(\mathbb{K}))\] by Lemma \ref{234 tensor}\ref{3tensor}, which clearly has more than four non-trivial Jordan blocks. Hence, $i=1$ and \ref{4times1} holds.
\end{proof}

Proposition \ref{SL2qmods} only really describes the structure of simple modules for $\SL_2(q)$ on which an element of order $p$ has few Jordan blocks. For applications in future work, we investigate the structure of certain indecomposable $\SL_2(q)$-modules with these restrictions. We set $\Lambda(q)$ to be the $\FF_q\SL_2(q)$-module which is dual to $V_p(q)$.

\begin{proposition}\label{extension}
Let $p$ be an odd prime, $q=p^n>p$, $X\cong \SL_2(q)$, $1\ne x\in T\in \syl_p(X)$, and let $Y$ be an indecomposable $\FF_pX$-module. Assume that
\begin{enumerate}
    \item [(a)] $x$ is $k$-active on $Y$ for some $k\leq n$;
    \item [(b)] $W:=\mathrm{Soc}(Y)\cong V_i(q)|_{\FF_p}$, where $1\leq i\leq p-1$;
    \item [(c)] $Y/W$ is a non-trivial irreducible $\FF_pX$-module; and
    \item [(d)] if $p=3$ then $|C_Y(T)|=q^2$ and $T$ does not act quadratically on $Y$.
\end{enumerate}
Then either
\begin{enumerate}
    \item $Y\cong V_p(q)|_{\FF_p}$;
    \item $Y\cong \Lambda(q)|_{\FF_p}$; or
    \item $p\geq 5$, $W\cong V_{p-3}(q)|_{\FF_p}$ and either
    \begin{enumerate}
        \item $n=2$ and $Y/W$ is isomorphic to a natural $\Omega_4^-(p)$-module; or
        \item $n>2$ and $Y/W\cong (V_1(q)^{\sigma^k}\otimes V_1(q)^{\sigma^{k+1}})|_{\FF_p}$ where $0\leq k\leq n-2$.
    \end{enumerate}
\end{enumerate}
\end{proposition}
\begin{proof}
First, we observe that as $x$ has at most $n$ non-trivial Jordan blocks in its action on $Y$, $x$ has at most $n$ non-trivial Jordan blocks in its action on $Y/W$. In particular, $Y/W$ is determined by Proposition \ref{SL2qmods}. Suppose that $Y/W\cong V_j(q)|_{\FF_p}$ for some $1\leq j\leq p-1$. Then $|C_{Y/W}(x)|=q$ and $|Y|=q^{i+j+2}$. Since $C_W(x)\le [W, x]$ we deduce that $Y$ has at most $n$ trivial Jordan blocks under the action of $x$. Since $Y$ has at most $n$ non-trivial Jordan blocks, it follows that $|Y|\leq q^{p+1}$. Applying \cite[Proposition 5.7]{poly}, we conclude that $Y\cong V_p(q)|_{\FF_p}$ or $Y\cong \Lambda(q)|_{\FF_p}$, as desired. Aiming for a contradiction, we suppose for the remainder of the proof that $Y/W\not\cong V_j(q)|_{\FF_p}$ for any $1\leq j\leq p-1$.

Suppose that $p=3$. If $W\cong V_2(q)|_{\FF_3}$ then as $x$ has at most $n$ non-trivial Jordan blocks in its action on $Y$, we deduce that $Y=WC_Y(x)$ and $X$ acts trivially on $Y/W$, a contradiction. Hence, $W\cong V_1(q)|_{\FF_3}$ and so $Z(X)$ acts non-trivially on $W$. By coprime action, using that $Y$ is indecomposable, we see that $Z(X)$ acts non-trivially on $Y/W$ and comparing with Proposition \ref{SL2qmods}, we conclude that $n$ is divisible by $3$ and $Y/W$ is a triality module. By Proposition \ref{SL2qmods}, we have that $|C_{Y}(x)/W|\leq |C_{Y/W}(x)|=q$. Since $C_W(x)\le [W, x]$ we conclude that $Y$ has at most $n$ trivial Jordan blocks under the action of $x$. But then $p^{4n}=q^4\geq |Y|=|Y/W||W|=p^{\frac{8n}{3}}p^{2n}>p^{4n}$, a contradiction. Hence, we may assume that $p\geq 5$.

Form $\mathcal{Y}:=Y\otimes_{\FF_p} \FF_q$ and $\mathcal{W}:=W\otimes_{\FF_p} \FF_q$. Note that $\FF_q$ is a splitting field for $X$ by \cite[Proposition 5.4.4]{KleidmanLiebeck}. If $\mathcal{Y}$ is completely reducible then $Y$ is isomorphic to an $\FF_p$-submodule of $\mathcal{Y}|_{\FF_p}$ and hence is also completely reducible, a contradiction. Hence, within the composition series of $\mathcal{Y}$ there is $\mathcal{Q}$ and $\mathcal{U}$ such that both $\mathcal{Q}/\mathcal{U}$ and $\mathcal{U}$ are irreducible, and $\mathcal{Q}$ is indecomposable. It follows that $\mathcal{U}|_{\FF_p}$ is a direct sum of modules isomorphic to $W$ and $(\mathcal{Q}/\mathcal{U})|_{\FF_p}$ is a direct sum of modules isomorphic to $Y/W$.

We adopt the conventions of \cite{andersenjorgsenlandrock}, noting that the existence of $\mathcal{Q}$ gives a non-trivial element of $\mathrm{Ext}_{\FF_qX}(\mathcal{Q}/\mathcal{U}, \mathcal{U})$. We set $\mathcal{Q}/\mathcal{U}=L(\lambda)$ and $\mathcal{U}=L(\mu)$. Then $\mu=ip^j$ for some $0\leq j\leq n-1$, and $\lambda=\sum\limits_{0\leq \ell\leq n-1} \lambda_\ell p^\ell$ where $\lambda_\ell\leq 1$ and between two and four $\lambda_\ell$s are non-zero. We apply \cite[Corollary 4.5]{andersenjorgsenlandrock}, and set $k$ as in their result.

If $j\ne k$ then $\mu_k=0$ and so $\lambda_k=p-2>1$, a contradiction. Hence, $j=k$ so that $\lambda_k=p-i-2$ and $\lambda_{k+1}=\pm 1$. Since $\lambda_k\leq 1$, the only possibility is that $i=p-3$ and $\mathcal{Q}/\mathcal{U}\cong V_1(q)^{\sigma^k}\otimes V_1(q)^{\sigma^{k+1}}$ for $\sigma$ a generator of $\Aut(\FF_q)$. Then either $q=p^2$ and $Y/W$ is isomorphic to a natural $\Omega_4^-(p)$-module; or $q>p^2$ and $Y/W\cong (V_1(q)^{\sigma^k}\otimes V_1(q)^{\sigma^{k+1}})|_{\FF_p}$, as desired.
\end{proof}

\section{The $\SU_3(p^n)$ case}\label[section]{SU3Section}

In this section, we describe the faithful $\FF_p\SU_3(p^n)$-modules $V$ and the non-trivial $p$-elements $x\in \SU_3(p^n)$ for which $x$ is $k$-active on $V$, where $k\leq 2n$. In addition, we also describe the faithful $\FF_{p^n}\SL_3(p^n)$-modules for which some non-trivial $p$-element of $\SL_3(p^n)$ has at most two non-trivial Jordan blocks.

Letting $X=\SU_3(p^n)$, we have that $\FF_{p^{2n}}$ is a splitting field for $X$ by \cite[Proposition 5.4.4]{KleidmanLiebeck}. The \emph{natural module} for $\SU_3(p^n)$ is the restriction to $\FF_p\SU_3(p^n)$ of any $3$-dimensional $\FF_{p^{2n}}\SL_3(p^{2n})$-module.

We refer to \cite{Zavarn} for a description of basic $p$-restricted $\FF_{p^n}\SL_3(p^n)$-modules. We adopt the notation that $W_{0,0}$ is the trivial module and whenever $\{a,b\}\ne \{0,0\}$ we have that $W_{a,b}$ is the kernel of the canonical map from $\Sym^a(M)\otimes \Sym^b(M^*)$ to $\Sym^{a-1}(M)\otimes \Sym^{b-1}(M^*)$, where $M$ is the natural $\FF_{p^n}\SL_3(p^n)$-module. Then we write $\bar{W}_{a,b}$ for the unique largest irreducible composition factor of $W_{a,b}$ and note by \cite[p. 265]{Zavarn} that $\{\bar{W}_{a,b}\mid 0 \le a,b \le p-1\}$ form the complete set of $p$-restricted basic modules for $\FF_{p^n}\SL_3(p^n)$. In particular, $W_{1,0}=\bar{W}_{1,0}$ is the natural $\FF_{p^n}\SL_3(p^n)$-module.

\begin{proposition}\label[proposition]{Wisirreducible}
Suppose that $\dim_{\FF_{p^n}} W_{a,b}\leq 6p-6$ for some $0\leq a,b\leq p-1$. Then either $W_{a,b}$ is irreducible; or
\begin{enumerate}
    \item $p=7$ and $\{a,b\}\in \{\{1,5\}, \{2,4\}\}$;
    \item $p=5$ and $\{a,b\}\in \{\{1,3\}, \{2,2\}\}$; or
    \item $p=3$ and $\{a,b\}=\{1,1\}$.
\end{enumerate}
Furthermore, if $\dim_{\FF_{p^n}} W_{a,b}\leq 4p-4$ and $W_{a,b}$ is reducible, then $p=3$ and $\{a,b\}=\{1,1\}$.
\end{proposition}
\begin{proof}
We observe that $\dim_{\FF_{p^n}} W_{a,b}=(a+1)(b+1)\frac{a+b+2}{2}$. If $a=0$ or $b=0$ then $W_{a,b}$ is (dual to) a symmetric power of the natural $\FF_{p^n}\SL_3(p^n)$ and is irreducible. Moreover, if either $a+b\leq p-2$, $a=p-1$ or $b=p-1$ then $W_{a,b}$ is irreducible by \cite[Lemma 7(a)]{Zavarn}. Otherwise, by \cite[Lemma 7(b)]{Zavarn}, we have that

\begin{align*}
\dim_{\FF_{p^n}} \bar{W}_{a,b}&=(a+1)(b+1)\frac{a+b+2}{2}-((p-1)-b)((p-1)-a)\frac{2p-(a+b+2)}{2} \\
&=(2ab+(a+b)+1 + (p-1)^2 -(a+b)(p-1))\frac{a+b+2}{2} - p((p-1)-a)((p-1)-b) \\
&\geq p(ab+\frac{p^2}{2}-p+1 -\frac{a+b}{2}(p-1)) - p((p-1)-a)((p-1)-b)\\
&=p(p-\frac{p^2}{2} + \frac{p(a+b)}{2})\geq p(p-\frac{p^2}{2} + \frac{p(p-1)}{2})
=\frac{p^2}{2}.
\end{align*}
Hence, we require $\frac{p^2}{2}\leq 6p-6$ and deduce  $p\leq 7$. Since $a,b<p-1$ and $(a+1)(b+1)\frac{a+b+2}{2}=\dim_{\FF_{p^n}} W_{a,b}\leq 6p-6$, we calculate that we have the promised possibilities, alongside the case $p=3$ and $\{a, b\}=\{1,2\}$. However, in this case we see that $\bar{W}_{a,b}=W_{a,b}$ is irreducible. If $\frac{p^2}{2}\leq 4p-4$ and $a,b<p-1\leq a+b$ then $p=3$ and $\{a,b\}=\{1,1\}$.
\end{proof}

Parts of the following two results may also be found in \cite[Proposition 5.1]{MinActive}.

\begin{proposition}\label[proposition]{SL3Split}
Suppose that $X=\SL_3(p^n)$, $S\in \syl_p(X)$ and $Y$ is a non-trivial irreducible $\FF_{p^n}X$-module.  Let $M$ denote a natural $\FF_{p^n}X$-module, $\mathcal R_1$ be the root elements in $S$ and $\mathcal R_2$ be the regular unipotent elements in $S$, $x \in S$ have order $p$, $H=\langle x\rangle$ and assume that $Y|_H$ has at most two non-trivial Jordan blocks. Then either
\begin{enumerate}
    \item $x\in \mathcal R_1$, $Y$ is a $3$-dimensional $\FF_{p^{n}}\SL_3(p^{n})$-module and $Y|_H=J_1(p^n)\oplus J_2(p^n)$;
    \item $p$ is odd, $x\in \mathcal R_1$, $Y=\Sym^2(M)$  and $Y|_H=J_1(p^n)\oplus J_2(p^n)\oplus J_3(p^n)$;
    \item $p$ is odd, $x\in \mathcal R_2$, $Y=\Sym^a(M)$, $1\leq a\leq \text{min}(4, \frac{p+1}{2})$ and $Y|_H$ is as in Lemma \ref{3SymPower1};
    \item $p\geq 5$, $x\in \mathcal R_2$, $Y=M\otimes M^{\sigma^i}$ or $M\otimes (M^*)^{\sigma^i}$, $\sigma$ is a generator of $\Aut(\FF_{p^n})$ and $0<i\leq n-1$, and $Y|_H= J_1(p^n)\oplus J_3(p^n)\oplus J_5(p^n)$; or
    \item $p$ is odd, $x \in \mathcal R_2 $, $Y$ is the  non-trivial irreducible composition factor of $M\otimes M^*$ and
    \[Y|_H=\begin{cases}
       J_1(p^n)\oplus 2J_3(p^n),& p=3\\
J_3(p^n)\oplus J_5(p^n), & p \geq 5
    \end{cases}.\]
\end{enumerate}
\end{proposition}
\begin{proof}
Let $Y$ be a counterexample to the proposition of minimal dimension. By the Steinberg Tensor Product Theorem, we may write
\[Y=W_0\otimes W_1^\sigma \otimes \dots \otimes W_{n-1}^{\sigma^{n-1}}\] where $W_i$ are all $p$-restricted basic modules. Up to a Galois twist, we may arrange that $W_0$ is a non-trivial $\FF_{p^n}X$-module. Suppose that there are at least two non-trivial tensorands and choose $j$ such that $W_j$ is also non-trivial. In particular, both $W_0$ and $W_j$ have an $H$-submodule isomorphic to $J_2(p^n)$. Then Lemma \ref{tensor powers} yields that $x$ has a unique non-trivial Jordan block on both $W_0$ and $W_j$. Observe that $\dim_{\FF_{p^n}} W_0\geq 3\leq \dim_{\FF_{p^n}} W_j$. Then we calculate using Lemma \ref{234 tensor} and linearity of the tensor product that $W_0|_H\cong W_j|_H\cong J_3(p^n)$. By Lemma \ref{234 tensor} \ref{3tensor}, we see that $p\geq 5$. Then by the Steinberg Tensor Product Theorem, we have that $Y\cong M\otimes M^{\sigma^j}$, $M\otimes (M^*)^{\sigma^j}$, $M^*\otimes M^{\sigma^j}$ or $M^*\otimes (M^*)^{\sigma^j}$, where $M$ is the natural $\FF_{p^n}\SL_3(p^n)$-module and $0<j\leq n-1$. But then (iv) holds, against $Y$ being a minimal counterexample. Hence we may assume, again up to a Galois twist, that $Y=W_0$ is $p$-restricted and so $Y\cong \bar{W}_{a,b}$ for some $0\leq a,b\leq p-1$. Set $M$ to be the natural $\FF_{p^n}\SL_3(p^n)$-module.

Suppose that $p=2$. Then $W_{a,b}=\bar{W}_{a,b}$ is irreducible and every involution of $X$ is conjugate to a root element and so has Jordan form $J_2(p^n)\oplus J_1(p^n)$ on the natural module. Moreover, either (i) holds or $a=b=1$. As $Y$ is a counterexample, we must have that $a=b=1$. Then Theorem \ref{Renaud} reveals that $(M\otimes M^*)|_H=4J_2(p^n)\oplus J_1(p^n)$ and as $W_{1,1}$ has codimension $1$ in $M\otimes M^*$, $Y|_H$ has at least three non-trivial Jordan blocks, a contradiction.

For the remainder of the proof, we may assume that $p$ is odd. Suppose first that $x\in \mathcal R_1$, so that $x$ has Jordan form $J_2(p^n)\oplus J_1(p^n)$ on $M$. Then by Lemma \ref{SL3Generation} and Lemma \ref{dimbound} we see that $\dim_{\FF_{p^n}} Y\leq 6p-6$. If $W_{a,b}$ is not irreducible then $\{a,b\}$ is one of the exceptions listed in Proposition \ref{Wisirreducible}. It is clear that the number of Jordan blocks of $x$ on $\bar{W}_{a,b}$ is the same regardless of $n$, and depends only on $p$, $a$ and $b$. We calculate in {\sc Magma} \cite{Magma} that none of the exceptions have the required number of Jordan blocks.

Therefore, $Y\cong W_{a,b}$ is irreducible. If $a=b=1$ then as $Y$ is irreducible we see that $p\geq 5$ and $Y$ is the unique non-trivial irreducible composition factor of $M\otimes M^*$. By Lemma \ref{234 tensor}\ref{2tensor}, we see that $(M\otimes M^*)|_H$ has three non-trivial Jordan blocks. Then \cite[Theorem 6.1]{Mikko} yields that $Y|_H$ also has three non-trivial Jordan blocks, a contradiction. If $a,b\geq 1$ and $a+b>2$ then we have a contradiction by Proposition \ref{2+1SymPower2}. Hence, one of $a$ or $b$ is zero so that $Y\cong \Sym^j(M)$ or $Y\cong \Sym^j(M^*)$ for $1\leq j\leq p-1$. But then Lemma \ref{2+1SymPower1} implies that $j\leq 2$ and (i) or (ii) holds.

Suppose that $x\in \mathcal R_2$  so that $x$ has Jordan form $J_3(p^n)$ on $M$. Then by Lemma \ref{SL3Generation}, $X$ is generated by two conjugates of $x$ and since $Y|_H$ has at most two non-trivial Jordan blocks, it follows that $\dim_{\FF_{p^n}} Y\leq 4p-4$. By Proposition \ref{Wisirreducible}, we have that $W_{a,b}$ is irreducible unless $\{a,b\}=\{1,1\}$ and $p=3$. In this latter case, we obtain (v). Hence, we may assume that $Y\cong W_{a,b}$ is irreducible, for $a,b\leq p-1$. If $a,b\geq 1$ then we have that $a=b=1$ by Lemma \ref{3SymPower4}, which gives (v) when $p\geq 5$. Hence, one of $a$ or $b$ is zero so that $Y\cong \Sym^j(M)$ or $Y\cong \Sym^j(M^*)$ for $1\leq j\leq p-1$. Then (iii) holds applying Lemma \ref{3SymPower1}.
\end{proof}

\begin{proposition}\label[proposition]{SU3Split}
Suppose that $X=\SU_3(p^n)$, $S\in \syl_p(X)$ and $W$ is an irreducible $\FF_{p^{2n}}X$-module. Let $M$ be a natural $\FF_{p^{2n}}\SL_3(p^{2n})$-module,
$x \in S$ have order $p$, $H=\langle x\rangle$ and assume that $W|_H$ has a unique non-trivial Jordan block. Then either
\begin{enumerate}
    \item $x\in Z(S)$, $W$ is the restriction to $\FF_{p^{2n}}X$ of $M$ and $W|_{H}=J_1(p^{2n})\oplus J_2(p^{2n})$;
    \item $p\geq 3$, $x\not\in Z(S)$, $W$ is the restriction to $\FF_{p^{2n}}X$ of  $M$ and $W|_{H}=J_3(p^{2n})$; or
    \item $p\geq 5$, $x\not\in Z(S)$, $W$ is the restriction to $\FF_{p^{2n}}X$ of $\Sym^2(M)$, and $W|_{H}=J_1(p^{2n})\oplus J_5(p^{2n})$.
\end{enumerate}
\end{proposition}
\begin{proof}
By \cite[Theorem 5.4.1]{KleidmanLiebeck}, we have that $W$ is the restriction of an irreducible $\FF_{p^{2n}}\SL_3(p^{2n})$-module to $\FF_{p^{2n}}\SU_3(p^n)$. Then if $x\in Z(S)$, we have that $x$ is conjugate to a root element in $\SL_3(p^{2n})$, while $x$ is regular unipotent otherwise. We apply Proposition \ref{SL3Split}, observing by Lemma \ref{3SymPower1} that if $x$ is regular unipotent and has a unique non-trivial Jordan block on $\Sym^a(M)$ then $a\leq 2$.
\end{proof}

\begin{lemma}\label[lemma]{fielddefn}
Suppose that $X=\SU_3(p^n)$, $S\in \syl_p(X)$ and $W$ is an irreducible $\FF_pX$-module. Let $x \in S$ have order $p$ and assume that $x$ is $k$-active on $W$ for some $k\leq 2n$. Then either:
\begin{enumerate}
    \item \label{qsplit} $\mathrm{End}_{\FF_pX}(W)\cong \FF_{p^n}$ and considering $W$ as an $\FF_{p^n}H$-module, $W$ has at most two non-trivial Jordan blocks; or
    \item\label{natsplit} $\mathrm{End}_{\FF_pX}(W)\cong \FF_{p^{2n}}$ and considering $W$ as an $\FF_{p^{2n}}H$-module, $W$ has a unique non-trivial Jordan block.
\end{enumerate}
\end{lemma}
\begin{proof}
Let $\mathbb{K}=\FF_{p^{2n}}$. Then $\mathbb{K}$ is a splitting field for $X$ and every irreducible module can be written over this field. Set $\mathbb{L}= \End_{\FF_pX}(W)$ and assume that $\mathbb{L}=\FF_{p^\ell}$. Then we may regard $W$ as an $\mathbb{L}X$-module and it has $\mathbb{L}$-dimension $d:=\dim_{\FF_p} W/\ell$. We set $\bar{W}=W\otimes_{\mathbb L}\mathbb K$, an irreducible $\mathbb{K}X$-module of $\mathbb{K}$-dimension $d$.

Suppose that $\bar{W}= \bigoplus_{i=1}^k n_iJ_i(\mathbb{K})$ as a $\mathbb{K}H$-module. Then, when $W$ is considered as an $\mathbb{L}H$-module it decomposes as $W|_H= \bigoplus_{i=1}^k n_iJ_i(\mathbb{L})$. Then as each $J_i(\mathbb L)$ when considered as a $\FF_pH$-module breaks as a direct sum of $\ell$ copies of $J_i$, if $W$ has at most $2n$ non-trivial Jordan blocks as an $\FF_pH$-module, we conclude that $\bar{W}$ has at most $2n/\ell$ non-trivial Jordan blocks as an $\mathbb{K}H$-module. In particular, if $\ell=n$ then \ref{qsplit} holds and if $\ell=2n$ then \ref{natsplit} holds. Hence, we may assume for the remainder of the proof that $\ell\not\in \{n, 2n\}$. 

Suppose first that $\ell\ne n$, $\ell$ divides $n$ and write $n=\ell a$ where $a>1$. Then $\bar{W}$ has at most $2a$ non-trivial Jordan blocks and by \cite[Proposition 5.4.6(ii)(a)]{KleidmanLiebeck} we may write
\begin{eqnarray}
\bar{W} & = & W_0\otimes W_0^{(\ell)} \otimes \dots \otimes W_0^{((a-1)\ell)}\label{ldivn}
\end{eqnarray}
where $W_0$ is an irreducible $\mathbb{K}X$-module with $W_0\cong W_0^{\tau}$, where $\tau$ is the restriction of the graph automorphism of $\SL_3(p^{2n})$ to $X$. We observe that $W_0$ is not a $3$-dimensional $\FF_{p^{2n}}X$-module by \cite[Proposition 5.4.8]{KleidmanLiebeck}.

If $W_0$ has at least two non-trivial Jordan blocks upon restriction to $H$, then we see by Lemma \ref{234 tensor}\ref{2tensor} that $\bar{W}$ has at least $2^a$ non-trivial Jordan blocks and we conclude that $p$ is odd and $a=2$. Indeed, we infer that $W_0|_H=J_2(\mathbb{K})\oplus J_2(\mathbb{K})$ and $\dim_{\mathbb{K}} W_0=4$. Since $W_0$ is the restriction to $X$ of an irreducible $\mathbb{K}\SL_3(p^{2n})$-module, we have a contradiction. Hence, $W_0$ has a unique non-trivial Jordan block upon restriction to $H$ and we conclude by Proposition \ref{SU3Split} that $p \geq 5$, $\dim_{\mathbb{K}} W_0=6$ and $W_0|_H=J_5(\mathbb{K})\oplus J_1(\mathbb{K})$. Since $J_5(\KK)\otimes J_5(\KK)$ certainly contains a non-trivial block, applying the binomial theorem the only possibility is that $a=2$. But then \[\bar{W}|_H=(J_5(\mathbb{K})\otimes J_5(\mathbb{K}))\oplus J_5(\mathbb{K})\oplus J_5(\mathbb{K})\oplus J_1(\mathbb{K})\] and Lemma \ref{234 tensor}\ref{4tensor} yields a contradiction.

Hence, we assume that $\ell$ does not divide $n$ and write $2n=\ell a$ where $a>1$ is odd. Then $\bar{W}$ has at most $a$ non-trivial Jordan blocks and by \cite[Proposition 5.4.6(ii)(b)]{KleidmanLiebeck} we may write
\begin{eqnarray}
\bar{W} & = & W_0\otimes (W_0^\tau)^{\left(\frac{\ell}{2}\right)} \otimes W_0^{(\ell)}\otimes (W_0^\tau)^{\left(\frac{3\ell}{2}\right)}\otimes \dots  \otimes (W_0)^{((a-1)\ell)}\otimes (W_0^\tau)^{\left(2n-\frac{\ell}{2}\right)}\label{ldivn2}
\end{eqnarray}
where $W_0$ is an irreducible $\mathbb{K}X$-module.

There are $2a$ twists of $W_0$ in (\ref{ldivn2}). Lemma \ref{tensor powers} then immediately reveals that $a=1$, another contradiction as $a>1$. Thus, we have shown that $\ell\in \{n, 2n\}$, which completes the proof.
\end{proof}

\begin{proposition}\label[proposition]{SU3mods}
Suppose that $X/Z(X)\cong \PSU_3(p^n)$, $S\in \syl_p(X)$ and $V$ is a faithful $\FF_pX$-module. Let $M$ be a natural $\FF_{p^{2n}}\SU_3(p^n)$-module, $m:=m_p(X)$, $x \in S$ have order $p$ and $H=\langle x\rangle$. If $x$ is $k$-active on $V$ for some $k\leq m$, then $p$ is odd, $W:=[V,X]/C_{[V, X]}(X)$ is irreducible and one of the following holds:
\begin{enumerate}
   \item $x\in Z(S)$, $\dim_{\FF_p}W=6n$, $W=M|_{\FF_p}$ is a natural $\SU_3(p^n)$-module and
    \[W|_H= 2nJ_1\oplus 2nJ_2;\]
    \item $x\not\in Z(S)$, $\dim_{\FF_p}W=6n$, $W=M|_{\FF_p}$ is a natural $\SU_3(p^n)$-module and
    \[W|_H=2nJ_3;\]
    \item $p=3$, $x\not\in Z(S)$, $\dim_{\FF_p}W=7n$, $W$ is any non-trivial, irreducible composition factor of $(M\otimes M^*)|_{\FF_3}$, and
    \[W|_H=nJ_1\oplus 2nJ_3;\]
    \item $p\geq 5$, $x\not\in Z(S)$, $\dim_{\FF_p}W=8n$, $W$ is any non-trivial, irreducible composition factor of $(M\otimes M^*)|_{\FF_p}$, and
    \[W|_H=nJ_3\oplus nJ_5;\,\,\text{or}\]
    \item $p\geq 5$, $x\not\in Z(S)$, $\dim_{\FF_p}W=12n$, $W=\Sym^2(M)|_{\FF_p}$, and
    \[W|_H=2nJ_1\oplus 2nJ_5.\]
\end{enumerate}
\end{proposition}
\begin{proof}
By Proposition \ref{blocks lemma2}, we have that $p$ is odd. We will show first that $W$ is irreducible. We may as well assume that $W$ has exactly two non-trivial composition factors, $U_1$ and $U_2$, both of which are non-trivial and, by Lemma \ref{fielddefn} and Proposition \ref{SU3Split}, belong to the list appearing in the statement. Clearly, we may assume that $W$ is indecomposable and without loss of generality, we suppose that $U_1$ is a non-trivial $\FF_pX$-submodule of $W$ with the non-trivial composition factor in $W/U_1$ isomorphic to $U_2$. We observe that, as $x$ has at most $2n$ non-trivial Jordan blocks on $W$, we have that $\dim_{\FF_p} W\leq 2n(p-1)+\dim_{\FF_p} C_W(x)$. Moreover, $\dim_{\FF_p} C_W(x)\leq \dim_{\FF_p} C_{U_1}(x)+\dim_{\FF_p}C_{U_2}(x)$ so that \[\dim_{\FF_p} W\leq 2n(p-1)+\dim_{\FF_p} C_{U_1}(x)+\dim_{\FF_p} C_{U_2}(x).\]

If $p=3$ then $\dim_{\FF_3} W\geq 12n$ and since $x$ has at most $2n$ non-trivial Jordan blocks, we deduce that $\dim_{\FF_3} C_W(x)\geq \dim_{\FF_3} W-4n\geq 8n$. But $\dim_{\FF_3} C_{U_1}(x)+\dim_{\FF_3} C_{U_2}(x)\leq 4n+4n=8n$, and we conclude that both $U_1$ and $W/U_1\cong U_2$ are natural $\SU_3(3^n)$-modules, $\dim_{\FF_3} W=12n$, $x\in Z(S)$ and $W|_{\langle x\rangle}=2nJ_3\oplus 6nJ_1$. In particular, $\dim_{\FF_3} C_W(x)=8n$ and $x\in Z(S)$.

Now, $\dim_{\FF_3} [W, x]=4n$ and $\dim_{\FF_3} [U_1, x]=2n$. Hence, $[W/U_1, x]=[W,x]U_1/U_1=C_{W/U_1}(S)$ has $\FF_3$-dimension $2n$. Since $C_{U_1}(S)=[U_1, x]\le [W, x]$, we have $[W,x]\cap U_1=[U_1,x]=C_{U_1}(S)$. Since $x\in Z(S)$, we have $[W, x]$ is $S$-invariant. Thus, $[W, x, S]\le [W,x]\cap C_{U_1}(S)=[U_1, x]$ and so $[W, x, S, S]=1$. The Three Subgroups Lemma reveals that $[S, S, [W, x]]=1$ and as $x\in [S, S]=Z(S)$, we conclude that $[W, x, x]=1$, a contradiction as $x$ has Jordan block of length $3$ on $W$.

Hence, if $W$ is not irreducible, then $p\geq 5$. Suppose first that $x\in Z(S)$ so that both $U_1$ and $U_2$ are natural $\SU_3(p^n)$-modules, and set $D:=\langle x, x^g\rangle\cong \SL_2(p^n)$ for some $g\in X$. Then $U_i|_D$ is a direct sum of two natural $\FF_p\SL_2(p^n)$-modules and $2n$ trivial modules. Hence, by coprime action applied using the center of $\SL_2(p^n)$, we see that $[W, D]$ contains four composition factors, all of which are natural $\SL_2(p^n)$-modules. Since each composition factor has $n$ non-trivial Jordan blocks, and $W$ has at most $2n$ non-trivial Jordan blocks, we must witness some non-split extension $T$ of a natural $\SL_2(p^n)$-module by a natural $\SL_2(p^n)$-module. We apply Proposition \ref{extension}, using that $p>3$, to see that $n=1$. We then appeal to the result preceding \cite[Corollary 4.5]{andersenjorgsenlandrock}, using that $p>3$, to see that the only possibility is that $p=5$. We calculate in {\sc Magma} \cite{Magma} that there are no indecomposable $\FF_5\SU_3(5)$-modules with the required composition factors.

Suppose now that $x\not\in Z(S)$ and set $D:=\langle x, x^g\rangle\cong \PSL_2(p^n)$ for some $g\in X$. Since the maximum size of a non-trivial Jordan block in $U_i$ is $5$, we deduce that every composition factor of $W|_D$ is either trivial, or isomorphic to $V_2(p^n)|_{\FF_p}$ or $V_4(p^n)|_{\FF_p}$. Assume that $n>1$. To understand the possible indecomposable submodules of $W|_D$ with no trivial composition factors, we appeal to Proposition \ref{extension}. Otherwise, to understand $W|_D$ we appeal to \cite[Corollary 4.5]{andersenjorgsenlandrock}. We deduce that $W|_D$ is completely reducible. It follows that $U_i|_D$ is completely reducible, both $U_1|_H$ and $U_2|_H$ contain at least $n$ non-trivial Jordan blocks and $W|_H$ has at most $2n$ non-trivial Jordan blocks. Hence, $U_i|_H$ has exactly $n$ non-trivial Jordan blocks, a contradiction as $U_i$ is described by the proposition. Thus $n=1$. If $W$ is completely reducible, then argue as in the $n>1$ case. To understand the possible indecomposable submodules of $W|_D$, we appeal to the result preceding \cite[Corollary 4.5]{andersenjorgsenlandrock} to deduce that $p\leq 11$. We then appeal to {\sc Magma} \cite{Magma} to show that there are no indecomposable $\FF_p\SU_3(p)$-modules with the required composition factors. Hence, $W$ is irreducible.

By Proposition \ref{SU3Split} and Lemma \ref{fielddefn}, we may assume that $\mathrm{End}_{\FF_pX}(W)\cong \FF_{p^n}$ and $x$ has at most two non-trivial Jordan blocks on $W$, viewed as an $\FF_{p^n}H$-module. Set $\bar{W}=W\otimes_{\FF_{p^n}} \FF_{p^{2n}}$ so that $\bar{W}$ is an irreducible $\FF_{p^{2n}}\SU_3(p^n)$-module on which $x$ has at most two non-trivial Jordan blocks. Applying \cite[Theorem 5.4.1]{KleidmanLiebeck} and \cite[Proposition 5.4.6(ii)(a)]{KleidmanLiebeck}, and using similar arguments as in Proposition \ref{SU3Split} and Lemma \ref{fielddefn}, we see that $\bar{W}\cong U\cong U^\tau$, where $\tau$ is the restriction of the graph automorphism of $\SL_3(p^{2n})$ to $\SU_3(p^n)$ and $U$ is the restriction to $\FF_{p^{2n}}\SU_3(p^n)$ of an irreducible $\FF_{p^{2n}}\SL_3(p^{2n})$-module.

From now on, we view $\bar{W}$ as an irreducible $\FF_{p^{2n}}\SL_3(p^{2n})$-module which is invariant under the graph automorphism of $\SL_3(p^{2n})$. We survey the modules from Proposition \ref{SL3Split} which are invariant under the graph automorphism of $\SL_3(p^{2n})$. In particular, for $M$ the natural $\FF_{p^{2n}}\SL_3(p^{2n})$-module we observe that there is $P$, a parabolic subgroup of $\SL_3(p^{2n})$, with the property that $\dim_{\FF_{p^{2n}}} C_{M|_{O_p(P)}}(O_p(P))=1$ so that $\dim_{\FF_{p^{2n}}} C_{\Sym^a(M)|_{O_p(P)}}(O_p(P))=1$ for all $a\geq 1$. On the other hand, \[1<\dim_{\FF_{p^{2n}}} C_{M^*|_{O_p(P)}}(O_p(P))\leq \dim_{\FF_{p^{2n}}} C_{\Sym^a(M^*)|_{O_p(P)}}(O_p(P)).\] Hence, $\Sym^a(M)\not\cong \Sym^a(M^*)$ for any relevant $a\geq 1$ so that $\Sym^a(M)$ is never invariant under the graph automorphism of $\SL_3(p^{2n})$. This gives the modules provided in the proposition.
\end{proof}

\section{The $\Ree(3^n)$ case}\label{sec: Ree}
In this section, we describe the faithful $\FF_3\Ree(3^n)$-modules $V$ and the non-trivial $3$-elements $x\in \Ree(3^n)$ for which $x$ is $k$-active on $V$ for $k\leq m_3(\Ree(3^n))=2n$. We note that $n$ is necessarily of the form $2a+1$ for some $a\in \N\cup \{0\}$ throughout this section.

Letting $X=\Ree(3^n)$, we have that $\FF_{3^n}$ is a splitting field for $X$ by \cite[Proposition 5.4.4]{KleidmanLiebeck}, and $X$ has three $3$-restricted basic modules over $\FF_{3^n}$: the trivial module, a $7$-dimensional module and a $27$-dimensional module \cite[Corollary 2.8.6]{GLS3}. The \emph{natural module} for $X$ is the restriction to $\FF_3X$ of any irreducible $7$-dimensional $\FF_{3^n}X$-module.

We recall that if $n=1$ then a Sylow $3$-subgroup of $X$ is extraspecial of order $27$ and exponent $9$. If $n>1$ then for $S\in\syl_3(X)$ we have that $m_3(S)=2n$, $|S|=3^{3n}$, $Z(S)=\mho^1(S)$ has order $3^n$, and $\Omega_1(S)=\Phi(S)=[S, S]$ has order $3^{2n}$.

\begin{lemma}\label[lemma]{RegRee}
Suppose that $X=\Ree(3^n)$, $S\in\syl_3(X)$ and let $K$ be a complement to $S$ in $N_X(S)$ so that $|K|=3^n-1$. Then $K$ acts regularly on the non-trivial elements of both $S/\Omega_1(S)$ and $Z(S)$.
\end{lemma}
\begin{proof}
If $n=1$, then this is easy to calculate. Let $n>1$ so that $\Omega_1(S)=\Phi(S)$ and every element in $S\setminus \Omega_1(S)$ has order $9$ and cubes to an element of $Z(S)$. Indeed, representatives from distinct cosets in $S/\Omega_1(S)$ cube to distinct elements of $Z(S)$. Now, $K$ acts faithfully on $S/\Phi(S)$, and respects the cubing map, and so acts faithfully on $Z(S)$. It follows that $K$ acts regularly on $S/\Omega_1(S)$ and $Z(S)$.
\end{proof}

\begin{lemma}\label[lemma]{JBRee}
Suppose that $D=\Ree(3)$, $S\in\syl_3(D)$ and $M$ is an irreducible $7$-dimensional $\FF_3D$-module. Then
\begin{enumerate}
    \item if $x\in Z(S)$, we have that $M|_{\langle x\rangle}=2J_2\oplus J_3$;
    \item if $x\in \Omega_1(S)\setminus Z(S)$, we have that $M|_{\langle x\rangle}=J_1\oplus 2J_3$; and
    \item if $x\in S$ has order $9$, then $M|_{\langle x\rangle}$ is indecomposable.
\end{enumerate}
\end{lemma}
\begin{proof}
This may be calculated directly using {\sc Magma} \cite{Magma}.
\end{proof}

\begin{lemma}\label[lemma]{27Proj}
Suppose that $D=\Ree(3)$, $S\in\syl_3(D)$ and $M$ is an irreducible $27$-dimensional $\FF_3D$-module. Then $M$ is projective and
\begin{enumerate}
    \item if $x\in \Omega_1(S)$, we have that $M|_{\langle x\rangle}=9J_3$; and
    \item if $x\in S$ has order $9$, then $M|_{\langle x\rangle}=3J_9$ where $J_9$ is the unique indecomposable $\FF_3C_9$-module of dimension $9$.
\end{enumerate}
\end{lemma}
\begin{proof}
This may be calculated directly using {\sc Magma} \cite{Magma}.
\end{proof}

\begin{lemma}\label[lemma]{NoExt}
Suppose that $D=\Ree(3)$, $S\in\syl_3(D)$ and $U$ is a faithful indecomposable $\FF_3D$-module. If $\dim_{\FF_3} U \geq 8$ and $x\in S$ has order $9$ then $x$ has a Jordan block of size at least $8$ on $U$.
\end{lemma}
\begin{proof}
By Lemma \ref{JBRee} and Lemma \ref{27Proj}, we may assume that $U$ is not irreducible. Let $T$ be a proper $\FF_3D$-submodule of $U$ and for the sake of simplicity, assume that both $U/T$ and $T$ are irreducible. The general case easily follows from this. Then either $T$ is trivial, $U/T$ is trivial, or $T$ and $U/T$ are both isomorphic to the natural $7$-dimensional module. We verify using {\sc Magma} \cite{Magma} that in all cases, $x$ has a Jordan block of size at least $8$ on $U$.
\end{proof}

\begin{lemma}\label[lemma]{ReeDecomposition}
Suppose that $D<X$ such that $D\cong \Ree(3)$ and $X=\Ree(3^n)$ for $n>1$. Set $M$ to be the natural module for $X$. Then $M|_D$ is a direct sum of $n$ natural modules for $D$.
\end{lemma}
\begin{proof}
Let $S\in\syl_3(X)$ and choose $x\in S\setminus \Omega_1(S)$ such that $x\in D$. We observe that $M$ arises as the restriction to $\FF_3$ of a $7$-dimensional $\FF_{3^n}X$-module and so $x$ has no Jordan blocks of size greater than $7$ on $M$. Applying Lemma \ref{27Proj} we see every composition factor of $M|_D$ is either trivial, or a natural module for $D$. Then Lemma \ref{NoExt} implies that $M|_D$ is a direct sum of natural and trivial modules for $D$.

Since $M|_D$ is a direct sum of natural and trivial $D$-modules, we may write
\[M=[M,D]\oplus C_M(D).\] Then $C_M(x)=C_{[M, D]}(x)\oplus C_M(D)$ where $\dim_{\FF_3}([M, D])=7a$ for some $a\in \N$. Then $\dim_{\FF_3} C_{[M, D]}(x)=a$ by Lemma \ref{JBRee}. Aiming for a contradiction, assume that $C_M(D)\ne 0$ so that $a<n$. Note that $\dim_{\FF_3} C_M(S)=n$ so that $C_M(S)\cap C_M(D)\ne 1$. However, $X=\langle S, D\rangle$ and since $C_M(X)=0$, we have the desired contradiction. Hence, $C_M(D)=0$ and $M|_D$ is a direct sum of $n$ natural modules for $D$.
\end{proof}

\begin{proposition}\label[proposition]{Ree3Case}
Suppose that $\wt X\cong \Ree(3)$ and $V$ is a faithful $\FF_3X$-module. Let $S\in\syl_3(X)$ and $x \in S$ have order $3$. Assume that $x$ is $k$-active on $V$ where $k\leq 2$. Then $X\cong \Ree(3)$, $W:=[V, X]/C_{[V, X]}(X)$ is the natural $7$-dimensional module and $x\in \Omega_1(S)\setminus Z(S)$.
\end{proposition}
\begin{proof}
By Proposition \ref{blocks lemma2}, we have that $L:=O^{3'}(O^3(X))\cong \PSL_2(8)$. Whether $x\in L$ or not, we compute that there is $g\in X$ such that $L\le P:=\langle x, x^g\rangle$ and since $x$ has at most $2$ non-trivial Jordan blocks on $V$, and acts cubically on $V$, we deduce that $|V/C_V(P)|\leq 3^8$. Indeed, $V$ contains a unique non-trivial composition factor upon restriction to $L$ so that $W_L:=[V, L]/C_{[V,L]}(L)$ is irreducible. Since $V$ is $X$-invariant, and $L\normaleq X$, we observe that $X/C_X(W_L)$ embeds in $N_{\GL_7(3)}(\PSL_2(8))\cong 2\times \Ree(3)$. Hence, $X/C_X(W_L)\cong \Ree(3)$, $W_L$ is a natural module and $C_X(W_L)$ is a $3'$-group.

By Lemma \ref{JBRee}, we see that $x\in \Omega_1(S)\setminus Z(S)$. Indeed, since $x$ has only $2$ non-trivial Jordan blocks on $V$, we must have that $[V, x]\le [V, L]$ and $[x, C_V(L)]=0$. Since $X=\langle x^X\rangle$, we ascertain that $[V, X]\le [V, L]+C_V(L)$ and $C_V(L)=C_V(X)$, so that the result holds.
\end{proof}

\begin{proposition}\label[proposition]{prop:ReeMods}
Suppose that $\wt X\cong \Ree(3^n)$, $n \geq 1$ and $V$ is a faithful $\FF_3X$-module. Let $S\in\syl_3(X)$ and $x \in S$ have order $3$. Assume that $x$ is $k$-active on $V$ where $k\leq 2n$. Then $X\cong \Ree(3^n)$, $W:=[V,X]/C_{[V, X]}(X)$ is the natural $7n$-dimensional module and $x\in \Omega_1(S)\setminus Z(S)$.
\end{proposition}
\begin{proof}
By Proposition \ref{Ree3Case} we may assume that $n>1$. Then Proposition \ref{blocks lemma2} yields that $X\cong \Ree(3^n)$ is simple. Since $x$ has at most $2n$ non-trivial Jordan blocks on $V$, we see that $\dim_{\FF_3} V/C_V(x)\leq 4n$. Then Proposition \ref{generation} yields that three conjugates of $x$ generate $X$, from which we deduce that $\dim_{\FF_3} V/C_V(X)\leq 12n$. Then Lemma \ref{ReeBound} implies that $W$ is irreducible.

Write $\mathbb{L}=\End_{\FF_3X} W$ and regard $W$ as an irreducible $\mathbb{L}X$-module. We recognize that $\FF_{3^n}$ is a splitting field for $X$ and we assume that $\mathbb{L}=\FF_{3^\ell}$ where $n=\ell a$. Set $\bar{W}:=W\otimes_{\mathbb{L}} \FF_{3^n}$ so that $\bar{W}$ is an irreducible $\FF_{3^n}X$-module with $\dim_{\FF_{3^n}} \bar{W}=\dim_{\mathbb{L}} W=\dim_{\FF_3} W/\ell$.

By the Steinberg Tensor Product Theorem we can write

\begin{eqnarray}
\bar{W} = W_0\otimes W_1^\sigma \otimes \dots \otimes W_{n-1}^{\sigma^{n-1}}\label{STP1Ree}
\end{eqnarray}

where $W_0, \dots, W_{n-1}$ are $3$-restricted basic $\FF_{3^n}X$-modules, and $\sigma$ is a generator of $\Aut(\FF_{3^n})$. Then there are $a\geq 1$ non-trivial tensor factors in this decomposition and we deduce that $7^a\leq \dim_{\FF_{3^n}} \bar{W}$. Then $7^a\leq \dim_{\FF_{3^n}} \bar{W}\leq 12n/\ell=12a$ so that $a=1$. Thus, $\End_{\FF_3X} W=\FF_{3^n}$, $\dim_{\FF_{3^n}} W\leq 12$ and we deduce that $W$ is a natural $7n$-dimensional module for $X$ when viewed as an $\FF_3X$-module.

Assume that $x\in Z(S)^\#$ is $k$-active on $V$, where $k\leq 2n$. Since $N_X(S)$ acts transitively on $Z(S)^\#$ by Lemma \ref{RegRee}, there is $D\cong \Ree(3)$ with $x\in D$. Then Lemma \ref{ReeDecomposition} implies that $W|_D$ is a direct sum of $n$ natural modules for $D$. Applying Lemma \ref{JBRee}, we see that $x$ has at least $3n$ non-trivial Jordan blocks in its action on $V$, a contradiction. Hence, if $x$ has at most $2n$ non-trivial Jordan blocks in its action on $V$, we must have that $x\in \Omega_1(S)\setminus Z(S)$.
\end{proof}

\appendix

\section{Tensor products of Jordan blocks}\label[appendix]{JordanSec}

In our treatment of groups of Lie type in characteristic $p$, motivated by Steinberg's Tensor Product Theorem, we require results about the Jordan form of a $p$-element acting on a module which arises as a tensor product of smaller indecomposable modules.

For this reason, we include this appendix which concerns only the representation theory of a cyclic group of order $p$, which may be of broad interest for those working in representation theory. While nothing here is new, we hope the results contained within provide a clear account of the Jordan block structure of tensor products of modules for a cyclic group of order $p$.  

We recall the notation for Jordan blocks from the Introduction.

\begin{notation}
Suppose that $p$ is a prime, $\mathbb{K}$ a field of characteristic $p$ and $H$ is a cyclic group of order $p$. The indecomposable $\mathbb{K}H$-module of dimension $n$, where $1\leq n \leq p$, is denoted by $J_n(\mathbb{K})$ and we represent a direct sum of $m\geq 0$ copies of $J_n(\mathbb{K})$ by $mJ_n(\mathbb{K})$. Furthermore, if $|\mathbb{K}|= p^a$ for some $a\in \N$, we will adapt our notation and write $J_n(p^a)$ for $J_n(\mathbb{K})$. If $|\mathbb{K}|=p$, we shall write $J_n:=J_n(p)$. Then for $V$ a finite dimensional $\FF_p H$-module, we have $V= \bigoplus\limits_{i=1}^p n_iJ_i$. Typically, we omit the terms with $n_i=0$.
\end{notation}

We will make liberal use of \cite[Theorem 1]{Renaud} throughout and so we record it here for convenience.

\begin{theorem}\label[theorem]{Renaud}
Let $\KK$ be a field of characteristic $p$. For $r\leq s\leq p$,
\[J_r(\KK)\otimes J_s(\KK)=(r-c)J_p(\KK) \oplus \bigoplus\limits_{i=1}^c J_{s-r+2i-1}(\KK)\] where
\[c=\begin{cases} r\quad\quad\quad\text{ if }r+s\leq p\\ p-s\,\,\quad\text{ if }r+s\geq p. \end{cases}\]
\end{theorem}

The following lemma documents some explicit cases of Theorem \ref{Renaud} which we will make frequent use of.

\begin{lemma}\label[lemma]{234 tensor}
Suppose that $\mathbb{K}$ is a field of characteristic $p$ and $1\leq k \leq p$.
\begin{enumerate}
\item\label{2tensor} For $p$ an arbitrary prime, we have \[J_k(\mathbb{K})\otimes J_2(\mathbb{K})=
\begin{cases}
J_2(\mathbb{K}) & k=1\\
J_{k-1}(\mathbb{K})\oplus J_{k+1}(\mathbb{K}) & 2\leq k <p \\
J_p(\mathbb{K})\oplus J_p(\mathbb{K}) & k= p
\end{cases}.\]
\item\label{3tensor} If $p\geq 3$ then
\[J_3(\mathbb{K})\otimes J_3(\mathbb{K})=
\begin{cases}
3J_3(\mathbb{K}) & p=3\\
J_1(\mathbb{K})\oplus J_3(\mathbb{K})\oplus J_5(\mathbb{K}) & p \geq 5
\end{cases}.\]
\item\label{4tensor} If $p\geq 5$, then
\[J_4(\mathbb{K})\otimes J_4(\mathbb{K})=
\begin{cases}
J_1(\mathbb{K})\oplus 3J_5(\mathbb{K}) & p=5\\
J_1(\mathbb{K})\oplus J_3(\mathbb{K})\oplus J_5(\mathbb{K})\oplus J_7(\mathbb{K}) & p \geq 7
\end{cases}.\]
\end{enumerate}
\end{lemma}

\begin{lemma}\label[lemma]{tensor powers}
Suppose that $H$ is a cyclic group of order $p$, $V$ is a $\mathbb{K}H$-module and that $e$ is a natural number with $e \geq 2$. Assume that $V$ is isomorphic to a tensor product of $e$ copies of $J_2(\mathbb{K})$. If the Jordan block decomposition of $V$ as a $\mathbb{K}H$-module has at most $e$ non-trivial Jordan blocks, then one of the following holds:
 \begin{enumerate}
 \item\label{e=2} $e=2$ and either
 \begin{enumerate}
     \item $p=2$ and $V \cong J_2(\mathbb{K}) \oplus J_2(\mathbb{K})$; or
     \item $p \geq 3$ and $V\cong J_1(\mathbb{K}) \oplus J_3(\mathbb{K})$.
 \end{enumerate}
 \item\label{e=3} $e=3$, $p \geq 3$ and $V$ is isomorphic to
            $\begin{cases}
            2J_2(\mathbb{K}) \oplus J_4(\mathbb{K}) &\text{if}\,\, p>3\\
            J_2(\mathbb{K}) \oplus 2J_3(\mathbb{K})  &\text{if}\,\, p=3
            \end{cases}$.
 \item\label{e=4} $e=4$, $p > 3$ and $V \cong 2J_1(\mathbb{K}) \oplus 3J_3(\mathbb K) \oplus J_5(\mathbb K)$.
 \end{enumerate}
\end{lemma}
\begin{proof}
Suppose that $p=2$. Applying Lemma \ref{234 tensor}\ref{2tensor}, we see that for $e\geq 2$ a tensor product of $e$ copies of $J_2(\mathbb{K})$ splits as a sum of $2^{e-1}$ copies of $J_2(\mathbb{K})$ and so the result certainly holds in this case. So suppose that $p \geq 3$. Then Lemma \ref{234 tensor}\ref{2tensor} yields
\[J_2(\mathbb{K})\otimes J_2(\mathbb{K})=J_1(\mathbb{K})\oplus J_3(\mathbb{K}),\]
\[J_2(\mathbb{K})\otimes J_2(\mathbb{K})\otimes J_2(\mathbb{K}) =
\begin{cases}
J_2(\mathbb{K})\oplus 2J_3(\mathbb{K}) & p=3\\
2J_2(\mathbb{K})\oplus J_4(\mathbb{K}) & p>3.\\
\end{cases}\]
and
\[J_2(\mathbb{K})\otimes J_2(\mathbb{K})\otimes J_2(\mathbb{K})\otimes J_2(\mathbb{K}) =
\begin{cases}
J_1(\mathbb{K})\oplus 5J_3(\mathbb{K}) & p=3\\
2J_1(\mathbb{K})\oplus  3J_3(\mathbb{K})\oplus J_5(\mathbb{K}) & p>3.\\
\end{cases}\]
This provides the descriptions of the decomposition of $V$ given in parts \ref{e=2}, \ref{e=3} and \ref{e=4}.

Thereafter, for $e\geq 5$, Lemma \ref{234 tensor}\ref{2tensor} shows that each tensor product with $J_2(\mathbb K)$ spawns two non-trivial Jordan blocks for each $J_k(\mathbb{K})$ with $k \geq 3$ one of which is either $J_p(\mathbb{K})$ or $J_{k+1}(\mathbb{K})$, and one non-trivial Jordan block for each $J_1(\mathbb{K})$ or $J_2(\mathbb{K})$. It follows from the decompositions when $e=4$ that for $e \geq 5$, $V$ has at least $e+1$ non-trivial Jordan blocks.
\end{proof}

We now present some results on the Jordan block structure of symmetric powers of certain small modules. These are used in the determination of the Jordan block structure of $p$-elements acting on $\FF_p\SL_3(p^n)$-modules and $\FF_p\SU_3(p^n)$-modules.

\begin{lemma}\label[lemma]{2+1SymPower1}
Suppose that $H$ is a cyclic group of order $p$, $p$ odd, and $V=J_1(\KK)\oplus J_2(\KK)$ is a $\KK H$-module. Then for $0\leq a\leq p-1$ we have that
\[\Sym^a(V)=\bigoplus\limits_{i=0}^{a} J_{i+1}(\KK).\]
\end{lemma}
\begin{proof}
We have that \[\Sym^a(V)=\bigoplus\limits_{i=0}^{a} \left(\Sym^i(J_2(\KK))\otimes \Sym^{a-i}(J_1(\KK))\right) =\bigoplus\limits_{i=0}^{a} \left(J_{i+1}(\KK)\otimes J_1(\KK)\right)=\bigoplus\limits_{i=0}^{a} J_{i+1}(\KK),\] as desired.
\end{proof}

\begin{lemma}\label[lemma]{2+1SymPower2}
Suppose that $H$ is a cyclic group of order $p$, $p$ odd, and $V=J_1(\KK)\oplus J_2(\KK)\cong W$ are $\KK H$-modules. Then for $1\leq a,b\leq p-1$ we have that  $\Sym^a(V)\otimes \Sym^b(W)=(\Sym^{a-1}(V)\otimes \Sym^{b-1}(W))\oplus U$, and exactly one of the following holds:
\begin{enumerate}
    \item\label{a+bge} $a+b\geq p$ and $U$ has at least four Jordan blocks of size $p$.
    \item\label{a+ble} $a+b< p$ and $U$ has four Jordan blocks of size at least $a+b-1$.
\end{enumerate}
Moreover, in outcome~\ref{a+ble}, $\Sym^{a-1}(V)\otimes\Sym^{b-1}(W)$ has a unique Jordan block of largest size $a+b-1$. If $a,b\geq 2$, then it has
exactly two Jordan blocks of size $a+b-2$. If exactly one of $a,b$ equals $1$, then it has exactly one Jordan block of size $a+b-2$. If $a=b=1$, then it has no further Jordan blocks. In every case, all remaining Jordan blocks have size at most $a+b-3$.
\end{lemma}
\begin{proof}
We note by Lemma \ref{2+1SymPower1} that $\Sym^a(V)=\Sym^a(J_2(\KK))\oplus \Sym^{a-1}(V)$. Then \[\Sym^a(V)\otimes \Sym^b(W)=\Sym^a(J_2(\KK))\otimes \Sym^b(W)\oplus \Sym^{a-1}(V)\otimes \Sym^b(J_2(\KK))\oplus \Sym^{a-1}(V)\otimes \Sym^{b-1}(W).\]
In particular, we find a $\KK H$-submodule of \[U:=(\Sym^a(J_2(\KK))\otimes \Sym^b(W))\oplus (\Sym^{a-1}(V)\otimes \Sym^b(J_2(\KK)))\] isomorphic to \[(J_{a+1}(\KK)\otimes J_{b+1}(\KK))\oplus (J_{a+1}(\KK)\otimes J_b(\KK))\oplus (J_a(\KK)\otimes J_{b+1}(\KK)).\] If $a+b\geq p$, an application of Theorem \ref{Renaud} yields that $U$ contains at least four Jordan blocks of size $p$, which proves (i).

Assume that $a+b<p$. Then by Theorem \ref{Renaud}, both $J_{a+1}(\KK)\otimes J_b(\KK)$ and $J_a(\KK)\otimes J_{b+1}(\KK)$ have a unique largest block of size $a+b$, and $J_{a+1}(\KK)\otimes J_{b+1}(\KK)$ contains a Jordan block of size $a+b+1$ and of size $a+b-1$. Thus, we recover four Jordan blocks of size at least $a+b-1$ in $U$. Another application of Theorem \ref{Renaud}, alongside Lemma \ref{2+1SymPower1}, implies that every Jordan block of $\Sym^{a-1}(V)\otimes \Sym^{b-1}(W)$ has size smaller than $p$. Furthermore, there is a unique Jordan block of largest size found in $\Sym^{a-1}(V)\otimes \Sym^{b-1}(W)$ arising from $J_a(\KK)\otimes J_b(\KK)$ which has size $a+b-1$, and the second largest blocks arise exactly as summands of $J_{a-1}(\KK)\otimes J_{b}(\KK)$ and $J_{a}(\KK)\otimes J_{b-1}(\KK)$. The result follows.
\end{proof}

\begin{proposition}\label[proposition]{2+1SymPower4}
Suppose that $H$ is a cyclic group of order $p$, $p$ odd, and $V=J_1(\KK)\oplus J_2(\KK)\cong W$ are $\KK H$-modules. Assume that $1\leq a,b\leq p-1$. Then for $\phi: \Sym^a(V)\otimes \Sym^b(W)\to \Sym^{a-1}(V)\otimes \Sym^{b-1}(W)$ a surjective map of $\mathbb{K}H$-modules, we have that $\ker(\phi)$ has at least three non-trivial Jordan blocks, or $a=b=1$ and $\ker(\phi)$ has at least two non-trivial Jordan blocks.
\end{proposition}
\begin{proof}
Suppose first that $a=b=1$. Then by Lemma \ref{234 tensor}\ref{2tensor}, we see that $V\otimes W$ is of the form $J_3(\KK)\oplus 2J_2(\KK)\oplus 2J_1(\KK)$, while $\dim_{\KK} \mathrm{Im}(\phi)=1$. Then the result clearly holds.

Suppose now that $2<a+b$, and set $c:=a+b-1$ if $a+b<p$ and $c:=p$ otherwise. Let $T$ be a submodule of $\Sym^a(V)\otimes \Sym^b(W)$ for which every Jordan block of $T$ has size at least $c$ and $T$ is maximal subject to this. If $a+b\geq p$ and $c=p$ then by Lemma \ref{2+1SymPower2}\ref{a+bge} we have that $\dim_{\KK} [T, H; p-2]=2d+\dim_{\KK} [\phi(T), H; p-2]$ where $d\geq 4$ is the number of blocks of size at least $p$ in $U$. It follows that $\dim_{\KK} ([T, H; p-2]\cap \ker(\phi))=2d$. Then $\dim_{\KK} C_T(H)=\dim_{\KK} [T, H; p-1]=d+\dim_{\KK} [\phi(T), H; p-1]$ so that $\dim_{\KK} (\ker(\phi)\cap C_T(H))=d$ from which we conclude that $\ker(\phi)$ has at least $d\geq 4$ non-trivial Jordan blocks.

If $a+b<p$ then by Lemma \ref{2+1SymPower2}\ref{a+ble}, we have that $U$ contains a block of size $c$, two blocks of size $c+1$ and a block of size $c+2$. Furthermore, $\phi(T)$ contains a unique maximal block of size $c$. We adopt the convention that $[T, H; 0]=T$. Thus, we see that $[T, H; c-2]=J_4(\KK)\oplus 2J_3(\KK)\oplus 2J_2(\KK)\oplus \ell J_1(\KK)$, for some $\ell$, and $[\phi(T), H; c-2]=J_2(\KK)+xJ_1(\KK)$ where $x\leq 2$. We calculate from this that $\ker(\phi)$ has at least three non-trivial Jordan blocks, as desired.
\end{proof}

The next lemma is a consequence of \cite[Proposition 1.4]{Almkvist}, which more generally determines the Jordan block structure of symmetric powers of indecomposable $\KK H$-modules. Here an empty direct sum is understood to be zero.

\begin{lemma}\label[lemma]{3SymPower1}
Suppose that $H$ is a cyclic group of order $p$, $p$ is odd, $1\leq a\leq p-1$ and $V=J_3(\KK)$ is a $\mathbb{K}H$-module. Then exactly one of the following holds:
\begin{enumerate}
    \item $a<\frac{p-1}{2}$ and
\[\Sym^a(V)=\begin{cases}
            \bigoplus\limits_{i=1}^{\frac{a}{2}+1} J_{4i-3}(\KK),& a\,\,\text{even} \\
            \bigoplus\limits_{i=1}^{\frac{a+1}{2}} J_{4i-1}(\KK),& a\,\,\text{odd}
\end{cases}.\]
    \item $\frac{p-1}{2}\leq a\leq p-1$ and
\[\Sym^a(V)=\begin{cases}
            \bigoplus\limits_{i=1}^{\frac{p-1-a}{2}} J_{4i-3}(\KK)\oplus \frac{2a-p+3}{2}J_{p}(\KK),& a\,\,\text{even} \\
            \bigoplus\limits_{i=1}^{\frac{p-2-a}{2}} J_{4i-1}(\KK)\oplus \frac{2a-p+3}{2}J_{p}(\KK),& a\,\,\text{odd}
\end{cases}.\]
\end{enumerate}
\end{lemma}

\begin{lemma}\label[lemma]{NoEvenBlocks}
Suppose that $H$ is a cyclic group of order $p$, $p$ is odd, and $V=J_3(\KK)\cong W$ are $\mathbb{K}H$-modules. Then for $1\leq a,b\leq p-1$ we have that $\Sym^a(V)\otimes \Sym^b(W)$ has only Jordan blocks of odd size.
\end{lemma}
\begin{proof}
By Lemma \ref{3SymPower1} and as $p$ is odd,  every Jordan block in $\Sym^a(V)$ and $\Sym^b(W)$ is odd. Since $p$ is odd, Theorem \ref{Renaud} gives every block of $\Sym^a(V)\otimes \Sym^b(W)$ is of odd size.
\end{proof}

\begin{lemma}\label[lemma]{3SymPower3}
Suppose that $H$ is a cyclic group of order $p$, $p$ odd, and $V=J_3(\KK)\cong W$ are $\mathbb{K}H$-modules. Then for $1\leq a,b\leq p-1$ with at least one of $a$ or $b$ strictly larger than $\frac{p-1}{2}$, we have that $\Sym^a(V)\otimes \Sym^b(W)$ has at least four more Jordan blocks of size $p$ than $\Sym^{a-1}(V)\otimes \Sym^{b-1}(W)$.
\end{lemma}
\begin{proof}
Without loss of generality, assume that $b\leq a>\frac{p-1}{2}$. If $a\geq p-2$ then $\Sym^a(V)=\frac{2a-p+3}{2}J_{p}(\KK)$ while $\Sym^{a-1}(V)=cJ_1(\KK)\oplus \frac{2a-p+1}{2}J_{p}(\KK)$ where $c=0$ if $a=p-1$ and $c=1$ if $a=p-2$. Applying Theorem \ref{Renaud} and linearity, we see that $J_p(\KK)\otimes J_i(\KK)=iJ_p(\KK)$ while $J_1(\KK)\otimes J_i(\KK)=J_i(\KK)$. Using that $b\geq 1$, we have the result unless perhaps $b=1$. But then $\Sym^{a-1}(V)\otimes \Sym^{b-1}(W)$ has exactly $\frac{2a-p+1}{2}J_{p}(\KK)$ blocks of size $p$, whereas $\Sym^a(V)\otimes \Sym^b(W)$ has $3\frac{2a-p+3}{2}$ blocks of size $p$. Since $p$ is odd, we have that $3\frac{2a-p+3}{2}-\frac{2a-p+1}{2}=2\frac{2a-p+3}{2}+1\geq 4$, as $a>\frac{p-1}{2}$.

Hence, we may assume that $\frac{p-1}{2}<a\leq p-3$ so that $p\geq 7$. Then
\begin{align*}
    \Sym^a(V)&=\dots\oplus J_{2p-2a-5}(\KK)\oplus J_p(\KK)&&\oplus \frac{2a-p+1}{2}J_p(\KK)\\
    \Sym^{a-1}(V)&= \dots\oplus J_{2p-2a-7}(\KK)\oplus J_{2p-2a-3}(\KK)\hspace{-6mm}&&\oplus \frac{2a-p+1}{2}J_p(\KK)
\end{align*}
Since $J_i(\KK)\otimes J_j(\KK)$ always produces at least as many Jordan blocks of size $p$ as $J_{i-2}(\KK)\otimes J_j(\KK)$, we may focus on the difference between $J_p(\KK)$ and $J_{2p-2a-3}(\KK)$.

Now, as $\frac{p-1}{2}<a\leq p-3$, we see that $2p-2a-3\leq p-4$. It then follows from Theorem \ref{Renaud} that $\Sym^a(V)\otimes \Sym^b(W)$ has at least four more Jordan blocks of size $p$ than $\Sym^{a-1}(V)\otimes \Sym^{b-1}(W)$, unless perhaps $b=1$. In this case we have that $\Sym^{a-1}(V)\otimes \Sym^{b-1}(W)$ has exactly $\frac{2a-p+1}{2}$ blocks of size $p$, whereas $\Sym^a(V)\otimes \Sym^b(W)$ has at least $3\frac{2a-p+3}{2}$ blocks of size $p$. Then, as $a\geq \frac{p+1}{2}$, we see that $3\frac{2a-p+3}{2}-\frac{2a-p+1}{2}\geq 5$, as desired. This completes the proof.
\end{proof}

\begin{lemma}\label[lemma]{3SymPower4}
Suppose that $H$ is a cyclic group of order $p$, $p$ odd, and $V=J_3(\KK)\cong W$ are $\mathbb{K}H$-modules. Then for $1\leq a,b\leq \frac{p-1}{2}$ with $2a+2b-3\geq p$, we have that $\Sym^a(V)\otimes \Sym^b(W)$ has at least four more Jordan blocks of size $p$ than $\Sym^{a-1}(V)\otimes \Sym^{b-1}(W)$.
\end{lemma}
\begin{proof}
Since $2a+2b-3\geq p$ and $a,b\leq \frac{p-1}{2}$, we may assume that $p\geq 5$ and one of $a$ or $b$ is larger than $1$. Without loss of generality, assume that $a>1$ and write $\Sym^a(V)=J_{2a+1}(\KK)\oplus V_1$ where $V_1=J_{2a-3}(\KK)\oplus J_{2a-7}(\KK)\oplus\dots$ Applying Theorem \ref{Renaud}, we see that $V_1\otimes \Sym^b(W)$ has the same number of Jordan blocks of size $p$ as $\Sym^{a-1}(V)\otimes \Sym^{b-1}(W)$. Hence, it remains to count the number of Jordan blocks of size $p$ in $J_{2a+1}(\KK)\otimes \Sym^b(W)$. However, we obtain at least $2a+2b+2-p\geq 5$ Jordan blocks of size $p$ from the tensor $J_{2a+1}(\KK)\otimes J_{2b+1}(\KK)$, which gives the result.
\end{proof}

\begin{lemma}\label[lemma]{3SymPower5}
Suppose that $H$ is a cyclic group of order $p$, $p$ odd, and $V=J_3(\KK)\cong W$ are $\mathbb{K}H$-modules. Then for $1\leq a,b\leq \frac{p-1}{2}$ with $2a+2b-3< p$ and $(a,b)\ne (1,1)$, we have that $\Sym^{a-1}(V)\otimes \Sym^{b-1}(W)$ has a unique Jordan block of size at least $2a+2b-3$, a unique block of size $2a+2b-5$ and every other Jordan block, if any, has size at most $2a+2b-7$. Meanwhile, $\Sym^{a}(V)\otimes \Sym^{b}(W)$ has at least four Jordan blocks of size at least $2a+2b-3$.
\end{lemma}
\begin{proof}
We have that $\Sym^a(V)=J_{2a+1}(\KK)\oplus J_{2a-3}(\KK)\oplus \dots$ and $\Sym^{a-1}(V)=J_{2a-1}(\KK)\oplus J_{2a-5}(\KK)\oplus \dots$ Since $2a+2b-3< p$, an application of Theorem \ref{Renaud} implies that every Jordan block of $\Sym^{a-1}(V)\otimes \Sym^{b-1}(W)$ has size smaller than $p$. We have that $\Sym^a(V)\otimes \Sym^b(W)$ contains the tensors $J_{2a+1}(\KK)\otimes J_{2b+1}(\KK)$, $J_{2a-3}(\KK)\otimes J_{2b+1}(\KK)$ and $J_{2a+1}(\KK)\otimes J_{2b-3}(\KK)$. Since $p$ is odd and $2a+2b-3<p$, we observe that $2a+2b-1\leq p$.

We first consider the case where one of $a$ or $b$ is equal to $1$. Without loss of generality, assume that $b=1$. If $a=2$ then $\Sym^{a-1}(V)\otimes \Sym^{b-1}(W)\cong V$ and so there is no second largest Jordan block. If $a>2$ then $\Sym^{a-1}(V)\otimes \Sym^{b-1}(W)$ contains a unique largest block of size $2a-1=2a+2b-3$ and the next largest block is of size at most $2a-5=2a+2b-7$. Meanwhile, $\Sym^a(V)\otimes \Sym^b(W)$ contains the tensors $J_{2a+1}(\KK)\otimes J_3(\KK)$ and $J_{2a-3}(\KK)\otimes J_3(\KK)$ which give rise to blocks of size $2a+3=2a+2b+1$ and $2a+1=2a+2b-1$, and two blocks of size $2a-1=2a+2b-3$.

If $2a+2b-1=p$ and $1<a,b$, then applying Theorem \ref{Renaud} we obtain two blocks of size $p$ from $J_{2a+1}(\KK)\otimes J_{2b+1}(\KK)$ and at least two blocks of size $2a+2b-3$, one from $J_{2a-3}(\KK)\otimes J_{2b+1}(\KK)$ and the other from $J_{2a+1}(\KK)\otimes J_{2b-3}(\KK)$. Moreover, $\Sym^{a-1}(V)\otimes \Sym^{b-1}(W)$ contains a unique largest block of size $2a+2b-3$ arising from $J_{2a-1}(\KK)\otimes J_{2b-1}(\KK)$, and a second largest block of size $2a+2b-7$.

Finally, if $2a+2b-1<p$, then as $p$ is odd, we have that $2a+2b+1\leq p$. Furthermore, if $1<a,b$ then applying Theorem \ref{Renaud} we obtain a block of size $2a+2b+1$, a block of size $2a+2b-1$ and a block of size $2a+2b-3$ from the tensor $J_{2a+1}(\KK)\otimes J_{2b+1}(\KK)$, and further blocks of size $2a+2b-3$ from $J_{2a-3}(\KK)\otimes J_{2b+1}(\KK)$ and $J_{2a+1}(\KK)\otimes J_{2b-3}(\KK)$. Moreover, $\Sym^{a-1}(V)\otimes \Sym^{b-1}(W)$ contains a unique largest block of size $2a+2b-3$ arising from $J_{2a-1}(\KK)\otimes J_{2b-1}(\KK)$, and a second largest block of size $2a+2b-5$. Every other block has size at most $2a+2b-7$.
\end{proof}

\begin{proposition}\label[proposition]{3SymPower6}
Suppose that $H$ is a cyclic group of order $p$, $p$ odd, and  $V=J_3(\KK)\cong W$ are $\mathbb{K}H$-modules. Assume that $1\leq a,b\leq p-1$. Then for $\phi: \Sym^a(V)\otimes \Sym^b(W)\to \Sym^{a-1}(V)\otimes \Sym^{b-1}(W)$ a $\mathbb{K}H$-module epimorphism, we have that $\ker(\phi)$ has at least three non-trivial Jordan blocks, or $a=b=1$.
\end{proposition}
\begin{proof}
Assume throughout that $(a,b)\ne (1,1)$. By Lemma \ref{NoEvenBlocks}, $\Sym^{a-1}(V)\otimes \Sym^{b-1}(W)$ has no even blocks, and so no blocks of size $p-1$. If $\Sym^a(V)\otimes \Sym^b(W)$ has at least four more Jordan blocks of size $p$ than $\Sym^{a-1}(V)\otimes \Sym^{b-1}(W)$, then arguing as in Proposition \ref{2+1SymPower4}, we obtain the result.

Hence, by Lemma \ref{3SymPower3} and Lemma \ref{3SymPower4}, we may assume that $a,b\leq \frac{p-1}{2}$ and $2a+2b-3< p$. Then $\Sym^{a}(V)\otimes \Sym^{b}(W)$ has at least four Jordan blocks of size at least $c:=2a+2b-3$ while $\Sym^{a-1}(V)\otimes \Sym^{b-1}(W)$ has a unique Jordan block of size $c$, a unique block of size $c-2=2a+2b-5$ and every other blocks of size at most $c-4$. Set $T$ to be the submodule of $\Sym^a(V)\otimes \Sym^b(W)$ for which every Jordan block of $T$ has size at least $2a+2b-5$ and $T$ is maximal subject to this. Then $[T, H; c-2]$ contains a submodule of the form $4J_2(\KK)$ while $[\phi(T), H; c-2]=J_2(\KK)$. From this, it follows that $\ker(\phi)$ has at least three non-trivial Jordan blocks.
\end{proof}

\newpage
\section{The table accompanying \cref{ModResult}}\label{AppB}

\begin{table}[htbp]
\centering\small
\renewcommand{\arraystretch}{1.25}
\setlength{\tabcolsep}{5pt}
\begin{tabular}{cllccl}
\hline

\hline\hline
\multicolumn{6}{c}{$X\cong\SL_2(p^m)$ or $\PSL_2(p^m)$}\\
\hline
\hline
Case & Conditions & $W$ & $\dim_{\FF_p}W$ & $x$ & JB of $x$ on $W$\\
\hline
(i)\ref{modbasic}      & $1\le i\le p{-}1$   & $V_i(p^m)|_{\FF_p}$                       & $(i{+}1)m$ & any & $mJ_{i+1} $\\
(i)(b)                & $p$ odd, $j\ne \frac{m}{2}$      & $(V_1\otimes V_1^{\sigma^j})|_{\FF_p}$    & $4m$  & any &$mJ_1\oplus mJ_3$\\
(i)\ref{Omega4}        & $p=2$, $m$ even     & $\Omega_4^-(2^{m/2})$-module              & $2m$  & any & $mJ_2$\\
(i)\ref{Omega4}        & $p$ odd, $m$ even   & $\Omega_4^-(p^{m/2})$-module              & $2m$  & any & $\frac m2 J_1\oplus \frac m 2J_3$\\
(i)\ref{trialityintro} & $p=3$, $3\mid m$    & triality module                           & $8m/3$ & any & $\frac m3J_2\oplus \frac{2m}{3}J_3$\\
(i)\ref{trialityintro} & $p\ge5$, $3\mid m$  & triality module                           & $8m/3$ & any & $ \frac{2m}3 J_2\oplus \frac m3 J_4$\\
(i)(e)                & $p\ge5$, $4\mid m$  & quadruple twist                           & $4m$  & any & $ \frac m2J_1 \oplus \frac{3m}4J_3\oplus \frac m 4 J_5$\\
(i)(f)                & $p\ge5$, $m$ even   & $(V_2\otimes V_2^{\sigma^{m/2}})$-summand & $9m/2$ & any & $\frac m2J_1\oplus \frac m2J_3\oplus \frac m2J_5$\\
\hline\hline

\multicolumn{6}{c}{$X\cong\SU_3(p^n)$;\quad $p$ odd, $m=2n$, $M$ a natural module, $W=[V,X]/C_{[V,X]}(X)$}\\
\hline
\hline
Case & Conditions & $W$ & $\dim_{\FF_p}W$ & $x$ & JB of $x$ on $W$\\
\hline
(ii)(a) &    $p$ odd      & natural module              & $6n$  & $p$-central     & $2nJ_1 \oplus 2nJ_2$\\
        &    \quad$''$         &     \quad\quad\quad$''$              &  $''$ & not $p$-central & $2nJ_3$\\
(ii)(b) & $p=3$    & $(M\otimes M^*)$-factor     & $7n$  & not $3$-central & $nJ_1 \oplus 2nJ_3$\\
(ii)(c) & $p\ge5$  & $(M\otimes M^*)$-factor     & $8n$  & not $p$-central & $nJ_3 \oplus nJ_5$\\
(ii)(d) & $p\ge5$  & $\Sym^2(M)|_{\FF_p}$        & $12n$ & not $p$-central & $2nJ_1 \oplus 2nJ_5$\\
\hline\hline
\multicolumn{6}{c}{$X\cong\Ree(3^n)$;\quad $p=3$, $m=2n$, $W=[V,X]/C_{[V,X]}(X)$}\\
\hline
\hline
Case & Conditions & $W$ & $\dim_{\FF_p}W$ & $x$ & JB of $x$ on $W$\\
\hline
 (iii) &  & natural module & $7n$ & not $3$-central & $nJ_1 \oplus 2nJ_3$\\
\hline\hline
\multicolumn{6}{c}{$X\cong\mathrm{M}_{11}$;\quad $p=3$, $m=2$, $W=[V,X]/C_{[V,X]}(X)$}\\
\hline
\hline
Case & Conditions & $W$ & $\dim_{\FF_p}W$ & $x$ & JB of $x$ on $W$\\
\hline
\ref{M11Main} &  & code / cocode module & $5$ & any & $J_2 \oplus J_3$\\
\hline\hline
\multicolumn{6}{c}{$X/Z(X)\cong \PSL_3(4)$;\quad $p=3$, $m=2$, $W=[V,X]$}\\
\hline
\hline
Case & Conditions & $W$ & $\dim_{\FF_p}W$ & $x$ & JB of $x$ on $W$\\
\hline
\ref{L34Main1} & $X\cong2{\cdot}\PSL_3(4)$ & $[V,X]$ & $6$ & any & $2J_3 $\\
\ref{L34Main2} & $X\cong4{\cdot}\PSL_3(4)$ & $[V,X]$ & $8$ & any & $2J_1 \oplus 2J_3$\\
\hline\hline
\multicolumn{6}{c}{$X/O_{p'}(X)\cong\Alt(2p)$;\quad $p$ odd, $m=2$, $x\in S\in\syl_p(L)\subseteq\syl_p(X)$}\\
\hline
\hline
Case & Conditions & $W$ & $\dim_{\FF_p}W$ & $x$ & JB of $x$ on $W$\\
\hline
\ref{altmainintro} & $L\cong\Alt(2p)$ & heart of perm.\ module & $2p{-}2$ & $p$-cycle & $(p-2)J_1 \oplus J_p$\\
                   &       \quad\quad$''$  &        \quad\quad\quad\quad\quad$''$   & $''$  & $(p,p)$-element & $J_{p-2}\oplus J_p$\\
                   & $p=3$, $L\cong\SL_2(9)$ & \multicolumn{4}{l}{$\SL_2(p^m)$-block with $m=2$}\\
                   & $p=5$, $L\cong2{\cdot}\Alt(10)$ & spin constituent & $8$ & $5$-cycle & $J_4 \oplus J_4$\\
                   &       \quad\quad$''$  &        \quad\quad\quad\quad\quad$''$   & $''$  &$(5,5)$-element & $J_3 \oplus J_5$
                   \\
\hline
\end{tabular}
\caption{The pairs $(W,x)$ in Theorem~\ref{ModResult}, grouped by $X$. The conditions in a block header apply to every row of that block. $bJ_a $ denotes $b$ Jordan blocks of size $a$, over $\FF_p$.}\label{Thm1 Table}
\end{table}

\bibliographystyle{alpha}
\bibliography{my.books}

@book {AschbacherFG,
    AUTHOR = {Aschbacher, Michael},
     TITLE = {Finite group theory},
    SERIES = {Cambridge Studies in Advanced Mathematics},
    VOLUME = {10},
 PUBLISHER = {Cambridge University Press, Cambridge},
      YEAR = {1986},
     PAGES = {x+274},
      ISBN = {0-521-30341-9},
   MRCLASS = {20-01 (20B05 20C05 20D05)},
  MRNUMBER = {895134},
MRREVIEWER = {Stephen\ D.\ Smith},
}

@article {parkersemerarocomputing,
    	AUTHOR = {Parker, Chris and Semeraro, Jason},
	TITLE = {Algorithms for fusion systems with applications to $p$-groups of small order},
 	VOLUME = {90},
    	YEAR = {2021},
	JOURNAL = {Mathematics of Computation},
	PAGES = {2415--2461},
 	DOI = {10.1090/mcom/3634},
 	URL ={https://doi.org/10.1090/mcom/3634}
}

@book {GOR,
    AUTHOR = {Gorenstein, Daniel},
     TITLE = {Finite groups},
   EDITION = {Second},
 PUBLISHER = {Chelsea Publishing Co., New York},
      YEAR = {1980},
     PAGES = {xvii+519}
}

@article {andersenjorgsenlandrock,
    AUTHOR = {Andersen, Henning Haahr and J{\o}rgensen, Jens and Landrock,
              Peter},
     TITLE = {The projective indecomposable modules of {${\rm
              SL}(2,\,p\sp{n})$}},
   JOURNAL = {Proc. London Math. Soc. (3)},
  FJOURNAL = {Proceedings of the London Mathematical Society. Third Series},
    VOLUME = {46},
      YEAR = {1983},
    NUMBER = {1},
     PAGES = {38--52},
      ISSN = {0024-6115},
   MRCLASS = {20G05 (16A64 20C20)},
  MRNUMBER = {684821},
MRREVIEWER = {Bhama Srinivasan},
       DOI = {10.1112/plms/s3-46.1.38},
       URL = {https://doi.org/10.1112/plms/s3-46.1.38},
}

@book {GLS3,
    AUTHOR = {Gorenstein, Daniel and Lyons, Richard and Solomon, Ronald},
     TITLE = {The classification of the finite simple groups. {N}umber 3.
              {P}art {I}. {C}hapter {A}},
    SERIES = {Mathematical Surveys and Monographs},
    VOLUME = {40},
      NOTE = {Almost simple $K$-groups},
 PUBLISHER = {American Mathematical Society, Providence, RI},
      YEAR = {1998},
     PAGES = {xvi+419}
}

@article {J,
    AUTHOR = {Juh\'{a}sz, Arye},
     TITLE = {On finite groups with a {S}ylow {$p$}-subgroup of maximal
              class},
   JOURNAL = {J. Algebra},
  FJOURNAL = {Journal of Algebra},
    VOLUME = {73},
      YEAR = {1981},
    NUMBER = {1},
     PAGES = {199--237},
      ISSN = {0021-8693},
   MRCLASS = {20D20},
  MRNUMBER = {641641},
MRREVIEWER = {\`E. M. Pal\cprime chik},
       DOI = {10.1016/0021-8693(81)90355-0},
       URL = {https://doi.org/10.1016/0021-8693(81)90355-0}
}

@article {Renaud,
    AUTHOR = {Renaud, J.-C.},
     TITLE = {The decomposition of products in the modular representation
              ring of a cyclic group of prime power order},
   JOURNAL = {J. Algebra},
  FJOURNAL = {Journal of Algebra},
    VOLUME = {58},
      YEAR = {1979},
    NUMBER = {1},
     PAGES = {1--11},
      ISSN = {0021-8693},
   MRCLASS = {20C20},
  MRNUMBER = {535838},
MRREVIEWER = {Michael Klemm},
       DOI = {10.1016/0021-8693(79)90184-4},
       URL = {https://doi.org/10.1016/0021-8693(79)90184-4},
}

@book {James,
    AUTHOR = {James, Gordon D.},
     TITLE = {The representation theory of the symmetric groups},
    SERIES = {Lecture Notes in Mathematics},
    VOLUME = {682},
 PUBLISHER = {Springer, Berlin},
      YEAR = {1978},
     PAGES = {v+156},
      ISBN = {3-540-08948-9},
   MRCLASS = {20C30 (20-02)},
  MRNUMBER = {513828},
MRREVIEWER = {Dragomir \v{Z}. \Dbar okovi\'{c}}
}

@article {Magma,
    AUTHOR = {Bosma, Wieb and Cannon, John and Playoust, Catherine},
     TITLE = {The {M}agma algebra system. {I}. {T}he user language},
      NOTE = {Computational algebra and number theory (London, 1993)},
   JOURNAL = {J. Symbolic Comput.},
  FJOURNAL = {Journal of Symbolic Computation},
    VOLUME = {24},
      YEAR = {1997},
    NUMBER = {3-4},
     PAGES = {235--265},
      ISSN = {0747-7171},
   MRCLASS = {68Q40},
  MRNUMBER = {MR1484478},
       DOI = {10.1006/jsco.1996.0125},
       URL = {http://dx.doi.org/10.1006/jsco.1996.0125},
}

@book {ModAt,
    AUTHOR = {Jansen, Christoph and Lux, Klaus and Parker, Richard and
              Wilson, Robert},
     TITLE = {An atlas of {B}rauer characters},
    SERIES = {London Mathematical Society Monographs. New Series},
    VOLUME = {11},
      NOTE = {Appendix 2 by T. Breuer and S. Norton,
              Oxford Science Publications},
 PUBLISHER = {The Clarendon Press, Oxford University Press, New York},
      YEAR = {1995},
     PAGES = {xviii+327},
      ISBN = {0-19-851481-6},
   MRCLASS = {20C20 (20-00 20D06 20D08)},
  MRNUMBER = {1367961},
MRREVIEWER = {J.\ L.\ Alperin},
}

@article {nonrealizability,
    AUTHOR = {Oliver, Bob},
     TITLE = {Nonrealizability of certain representations in fusion systems},
   JOURNAL = {J. Aust. Math. Soc.},
  FJOURNAL = {Journal of the Australian Mathematical Society},
    VOLUME = {116},
      YEAR = {2024},
    NUMBER = {2},
     PAGES = {257--288},
      ISSN = {1446-7887,1446-8107},
   MRCLASS = {20D20 (20C20 20D05 20E45)},
  MRNUMBER = {4715166},
       DOI = {10.1017/S1446788723000022},
       URL = {https://doi.org/10.1017/S1446788723000022},
}

@book {KleidmanLiebeck,
    AUTHOR = {Kleidman, Peter and Liebeck, Martin},
     TITLE = {The subgroup structure of the finite classical groups},
    SERIES = {London Mathematical Society Lecture Note Series},
    VOLUME = {129},
 PUBLISHER = {Cambridge University Press, Cambridge},
      YEAR = {1990},
     PAGES = {x+303},
      ISBN = {0-521-35949-X},
   MRCLASS = {20-02 (20D06 20G40)},
  MRNUMBER = {1057341},
MRREVIEWER = {R. W. Carter},
       DOI = {10.1017/CBO9780511629235},
       URL = {https://doi.org/10.1017/CBO9780511629235},
}

@book {dixon-mortimer,
    AUTHOR = {Dixon, John D. and Mortimer, Brian},
     TITLE = {Permutation groups},
    SERIES = {Graduate Texts in Mathematics},
    VOLUME = {163},
 PUBLISHER = {Springer-Verlag, New York},
      YEAR = {1996},
     PAGES = {xii+346},
      ISBN = {0-387-94599-7},
   MRCLASS = {20B05 (20-01 20B07)},
  MRNUMBER = {1409812},
MRREVIEWER = {Martin\ W.\ Liebeck},
       DOI = {10.1007/978-1-4612-0731-3},
       URL = {https://doi.org/10.1007/978-1-4612-0731-3},
}

@book {wielandt,
    AUTHOR = {Wielandt, Helmut},
     TITLE = {Finite permutation groups.},
      NOTE = {Translated from the German by R. Bercov},
 PUBLISHER = {Academic Press, New York-London, },
      YEAR = {1964},
     PAGES = {x+114},
   MRCLASS = {20.20},
  MRNUMBER = {183775},
MRREVIEWER = {N.\ Ito},
}

@article {MinPermRank1,
    AUTHOR = {Mazurov, V. D.},
     TITLE = {Minimal permutation representations of finite simple classical
              groups. {S}pecial linear, symplectic and unitary groups},
   JOURNAL = {Algebra i Logika},
  FJOURNAL = {Sibirski\u i\ Fond Algebry i Logiki. Rossi\u iskaya Akademiya
              Nauk. Sibirskoe Otdelenie. Institut Matematiki. Algebra i
              Logika},
    VOLUME = {32},
      YEAR = {1993},
    NUMBER = {3},
     PAGES = {267--287, 343},
      ISSN = {0373-9252},
   MRCLASS = {20G05 (20D06 20G40)},
  MRNUMBER = {1286555},
MRREVIEWER = {Irina\ Suprunenko},
       DOI = {10.1007/BF02261693},
       URL = {https://doi.org/10.1007/BF02261693},
}

@article {MinPermTwisted,
    AUTHOR = {Vasilyev, A. V.},
     TITLE = {Minimal permutation representations of finite simple
              exceptional groups of twisted type},
   JOURNAL = {Algebra i Logika},
  FJOURNAL = {Sibirski\u i\ Fond Algebry i Logiki. Algebra i Logika},
    VOLUME = {37},
      YEAR = {1998},
    NUMBER = {1},
     PAGES = {17--35, 122},
      ISSN = {0373-9252},
   MRCLASS = {20G40 (20B15 20D06)},
  MRNUMBER = {1672901},
MRREVIEWER = {A.\ S.\ Kondrat\cprime ev},
       DOI = {10.1007/BF02684081},
       URL = {https://doi.org/10.1007/BF02684081},
}

@article {Jansen,
    AUTHOR = {Jansen, Christoph},
     TITLE = {The minimal degrees of faithful representations of the
              sporadic simple groups and their covering groups},
   JOURNAL = {LMS J. Comput. Math.},
  FJOURNAL = {LMS Journal of Computation and Mathematics},
    VOLUME = {8},
      YEAR = {2005},
     PAGES = {122--144},
      ISSN = {1461-1570},
   MRCLASS = {20C34},
  MRNUMBER = {2153793},
MRREVIEWER = {Robert\ A.\ Wilson},
       DOI = {10.1112/S1461157000000930},
       URL = {https://doi.org/10.1112/S1461157000000930},
}

@article {Seitz2,
    AUTHOR = {Seitz, Gary M. and Zalesskii, Alexander E.},
     TITLE = {On the minimal degrees of projective representations of the
              finite {C}hevalley groups. {II}},
   JOURNAL = {J. Algebra},
  FJOURNAL = {Journal of Algebra},
    VOLUME = {158},
      YEAR = {1993},
    NUMBER = {1},
     PAGES = {233--243},
      ISSN = {0021-8693,1090-266X},
   MRCLASS = {20C25},
  MRNUMBER = {1223676},
MRREVIEWER = {Larry\ C.\ Grove},
       DOI = {10.1006/jabr.1993.1132},
       URL = {https://doi.org/10.1006/jabr.1993.1132},
}

@article {Zsigmondy,
    AUTHOR = {Zsigmondy, K.},
     TITLE = {Zur {T}heorie der {P}otenzreste},
   JOURNAL = {Monatsh. Math. Phys.},
  FJOURNAL = {Monatshefte f\"ur Mathematik und Physik},
    VOLUME = {3},
      YEAR = {1892},
    NUMBER = {1},
     PAGES = {265--284},
      ISSN = {1812-8076},
   MRCLASS = {99-04},
  MRNUMBER = {1546236},
       DOI = {10.1007/BF01692444},
       URL = {https://doi.org/10.1007/BF01692444},
}

@article {OV,
    AUTHOR = {Oliver, Bob and Ventura, Joana},
     TITLE = {Saturated fusion systems over 2-groups},
   JOURNAL = {Trans. Amer. Math. Soc.},
  FJOURNAL = {Transactions of the American Mathematical Society},
    VOLUME = {361},
      YEAR = {2009},
    NUMBER = {12},
     PAGES = {6661--6728},
      ISSN = {0002-9947,1088-6850},
   MRCLASS = {20D20 (20D08 20D45)},
  MRNUMBER = {2538610},
MRREVIEWER = {Adam\ Glesser},
       DOI = {10.1090/S0002-9947-09-04881-8},
       URL = {https://doi.org/10.1090/S0002-9947-09-04881-8},
}

@book {vbbook,
    AUTHOR = {van Beek, Martin},
     TITLE = {Rank 2 amalgams and fusion systems},
    SERIES = {Lecture Notes in Mathematics},
    VOLUME = {2343},
 PUBLISHER = {Springer, Cham},
      YEAR = {2024},
     PAGES = {vii+203},
      ISBN = {978-3-031-54460-6; 978-3-031-54461-3},
   MRCLASS = {20-02 (05E18 20D05 20D20 20E06 20E42 55R35)},
  MRNUMBER = {4769827},
       DOI = {10.1007/978-3-031-54461-3},
       URL = {https://doi.org/10.1007/978-3-031-54461-3},
}

@article {parkerSE,
    AUTHOR = {Parker, Chris and Stroth, Gernot},
     TITLE = {Strongly {$p$}-embedded subgroups},
   JOURNAL = {Pure Appl. Math. Q.},
  FJOURNAL = {Pure and Applied Mathematics Quarterly},
    VOLUME = {7},
      YEAR = {2011},
    NUMBER = {3},
     PAGES = {797--858},
      ISSN = {1558-8599,1558-8602},
   MRCLASS = {20D30 (20D05 20E07)},
  MRNUMBER = {2848592},
MRREVIEWER = {Ben\ Fairbairn},
       DOI = {10.4310/PAMQ.2011.v7.n3.a9},
       URL = {https://doi.org/10.4310/PAMQ.2011.v7.n3.a9},
}

@incollection {Sporadics,
    AUTHOR = {Wilson, Robert A.},
     TITLE = {Maximal subgroups of sporadic groups},
 BOOKTITLE = {Finite simple groups: thirty years of the atlas and beyond},
    SERIES = {Contemp. Math.},
    VOLUME = {694},
     PAGES = {57--72},
 PUBLISHER = {Amer. Math. Soc., Providence, RI},
      YEAR = {2017},
      ISBN = {978-1-4704-3678-0; 978-1-4704-4168-5},
   MRCLASS = {20-02},
  MRNUMBER = {3682590},
       DOI = {10.1090/conm/694},
       URL = {https://doi.org/10.1090/conm/694},
}

@article {WagnerInvolution,
    AUTHOR = {Wagner, Ascher},
     TITLE = {The minimal number of involutions generating some finite
              three-dimensional groups},
   JOURNAL = {Boll. Un. Mat. Ital. A (5)},
  FJOURNAL = {Bollettino della Unione Matematica Italiana. A. Serie 5},
    VOLUME = {15},
      YEAR = {1978},
    NUMBER = {2},
     PAGES = {431--439},
      ISSN = {0392-4033},
   MRCLASS = {20G40},
  MRNUMBER = {508070},
MRREVIEWER = {Ulrich\ Dempwolff},
}

@book {BHRD,
    AUTHOR = {Bray, John N. and Holt, Derek F. and Roney-Dougal, Colva M.},
     TITLE = {The maximal subgroups of the low-dimensional finite classical
              groups},
    SERIES = {London Mathematical Society Lecture Note Series},
    VOLUME = {407},
      NOTE = {With a foreword by Martin Liebeck},
 PUBLISHER = {Cambridge University Press, Cambridge},
      YEAR = {2013},
     PAGES = {xiv+438},
      ISBN = {978-0-521-13860-4},
   MRCLASS = {20G40 (20E28)},
  MRNUMBER = {3098485},
MRREVIEWER = {Nadia\ P.\ Mazza},
       DOI = {10.1017/CBO9781139192576},
       URL = {https://doi.org/10.1017/CBO9781139192576},
}

@article {poly,
    AUTHOR = {Grazian, Valentina and Parker, Chris and Semeraro, Jason and
              van Beek, Martin},
     TITLE = {Fusion systems related to polynomial representations of
              {SL}2(q)},
   JOURNAL = {J. Lond. Math. Soc. (2)},
  FJOURNAL = {Journal of the London Mathematical Society. Second Series},
    VOLUME = {113},
      YEAR = {2026},
    NUMBER = {3},
     PAGES = {Paper No. e70481},
      ISSN = {0024-6107,1469-7750},
   MRCLASS = {20},
  MRNUMBER = {5041305},
       DOI = {10.1112/jlms.70481},
       URL = {https://doi.org/10.1112/jlms.70481},
}

@article {Ward,
    AUTHOR = {Ward, Harold N.},
     TITLE = {On {R}ee's series of simple groups},
   JOURNAL = {Trans. Amer. Math. Soc.},
  FJOURNAL = {Transactions of the American Mathematical Society},
    VOLUME = {121},
      YEAR = {1966},
     PAGES = {62--89},
      ISSN = {0002-9947,1088-6850},
   MRCLASS = {20.80 (20.29)},
  MRNUMBER = {197587},
MRREVIEWER = {Rimhak\ Ree},
       DOI = {10.2307/1994333},
       URL = {https://doi.org/10.2307/1994333},
}

@article {premet,
    AUTHOR = {Kleshchev, Alexander S. and Premet, Alexander A.},
     TITLE = {On second degree cohomology of symmetric and alternating
              groups},
   JOURNAL = {Comm. Algebra},
  FJOURNAL = {Communications in Algebra},
    VOLUME = {21},
      YEAR = {1993},
    NUMBER = {2},
     PAGES = {583--600},
      ISSN = {0092-7872,1532-4125},
   MRCLASS = {20J06 (20B30 20B35)},
  MRNUMBER = {1199691},
MRREVIEWER = {U.\ Stammbach},
       DOI = {10.1080/00927879308824581},
       URL = {https://doi.org/10.1080/00927879308824581},
}

@article {KleschevTiep,
    AUTHOR = {Kleshchev, Alexander S. and Tiep, Pham Huu},
     TITLE = {Small-dimensional projective representations of symmetric and
              alternating groups},
   JOURNAL = {Algebra Number Theory},
  FJOURNAL = {Algebra \& Number Theory},
    VOLUME = {6},
      YEAR = {2012},
    NUMBER = {8},
     PAGES = {1773--1816},
      ISSN = {1937-0652,1944-7833},
   MRCLASS = {20C25 (20C20 20C30)},
  MRNUMBER = {3033527},
MRREVIEWER = {Burkhard\ K\"ulshammer},
       DOI = {10.2140/ant.2012.6.1773},
       URL = {https://doi.org/10.2140/ant.2012.6.1773},
}

@article {BrauerNesbitt,
    AUTHOR = {Brauer, R. and Nesbitt, C.},
     TITLE = {On the modular characters of groups},
   JOURNAL = {Ann. of Math. (2)},
  FJOURNAL = {Annals of Mathematics. Second Series},
    VOLUME = {42},
      YEAR = {1941},
     PAGES = {556--590},
      ISSN = {0003-486X},
   MRCLASS = {20.0X},
  MRNUMBER = {4042},
MRREVIEWER = {G.\ de B. Robinson},
       DOI = {10.2307/1968918},
       URL = {https://doi.org/10.2307/1968918},
}

@article {MinActive,
    AUTHOR = {Craven, David A.},
     TITLE = {Groups with a {$p$}-element acting with a single non-trivial
              {J}ordan block on a simple module in characteristic {$p$}},
   JOURNAL = {J. Group Theory},
  FJOURNAL = {Journal of Group Theory},
    VOLUME = {21},
      YEAR = {2018},
    NUMBER = {5},
     PAGES = {719--787},
      ISSN = {1433-5883,1435-4446},
   MRCLASS = {20C20 (20G40)},
  MRNUMBER = {3849670},
MRREVIEWER = {A.\ S.\ Kondrat\cprime ev},
       DOI = {10.1515/jgth-2018-0014},
       URL = {https://doi.org/10.1515/jgth-2018-0014},
}

@article {Almkvist,
    AUTHOR = {Almkvist, Gert},
     TITLE = {Representations of {${\bf Z}/p{\bf Z}$}\ in characteristic
              {$p$}\ and reciprocity theorems},
   JOURNAL = {J. Algebra},
  FJOURNAL = {Journal of Algebra},
    VOLUME = {68},
      YEAR = {1981},
    NUMBER = {1},
     PAGES = {1--27},
      ISSN = {0021-8693},
   MRCLASS = {14L30 (12F10 20C20)},
  MRNUMBER = {604290},
MRREVIEWER = {R.\ M.\ Fossum},
       DOI = {10.1016/0021-8693(81)90281-7},
       URL = {https://doi.org/10.1016/0021-8693(81)90281-7},
}

@article {Zavarn,
    AUTHOR = {Zavarnitsine, A. V.},
     TITLE = {Weights of irreducible {${\rm SL}_3(q)$}-modules in the
              defining characteristic},
   JOURNAL = {Sibirsk. Mat. Zh.},
  FJOURNAL = {Rossi\u iskaya Akademiya Nauk. Sibirskoe Otdelenie. Institut
              Matematiki im. S. L. Soboleva. Sibirski\u i\ Matematicheski\u
              i\ Zhurnal},
    VOLUME = {45},
      YEAR = {2004},
    NUMBER = {2},
     PAGES = {319--328},
      ISSN = {0037-4474},
   MRCLASS = {20C33 (20D20 20G40)},
  MRNUMBER = {2061413},
MRREVIEWER = {A.\ S.\ Kondrat\cprime ev},
       DOI = {10.1023/B:SIMJ.0000021282.92576.f8},
       URL = {https://doi.org/10.1023/B:SIMJ.0000021282.92576.f8},
}

@article {Suprunenko,
    AUTHOR = {Suprunenko, I. D.},
     TITLE = {Unipotent elements of nonprime order in representations of the
              classical algebraic groups: two big {J}ordan blocks},
   JOURNAL = {Zap. Nauchn. Sem. S.-Peterburg. Otdel. Mat. Inst. Steklov.
              (POMI)},
  FJOURNAL = {Rossi\u iskaya Akademiya Nauk. Sankt-Peterburgskoe Otdelenie.
              Matematicheski\u i\ Institut im. V. A. Steklova. Zapiski
              Nauchnykh Seminarov (POMI)},
    VOLUME = {414},
      YEAR = {2013},
     PAGES = {193--241},
      ISSN = {0373-2703},
   MRCLASS = {20G05 (20G07)},
  MRNUMBER = {3470603},
MRREVIEWER = {Christopher\ P.\ Bendel},
       DOI = {10.1007/s10958-014-1863-6},
       URL = {https://doi.org/10.1007/s10958-014-1863-6},
}

@article {TestZale,
    AUTHOR = {Testerman, Donna M. and Zalesski, Alexandre E.},
     TITLE = {Irreducible representations of simple algebraic groups in
              which a unipotent element is represented by a matrix with a
              single non-trivial {J}ordan block},
   JOURNAL = {J. Group Theory},
  FJOURNAL = {Journal of Group Theory},
    VOLUME = {21},
      YEAR = {2018},
    NUMBER = {1},
     PAGES = {1--20},
      ISSN = {1433-5883,1435-4446},
   MRCLASS = {20G05 (20G41)},
  MRNUMBER = {3739341},
MRREVIEWER = {Adam\ R.\ Thomas},
       DOI = {10.1515/jgth-2017-0019},
       URL = {https://doi.org/10.1515/jgth-2017-0019},
}

@article {Mikko,
    AUTHOR = {Korhonen, Mikko},
     TITLE = {Jordan blocks of unipotent elements in some irreducible
              representations of classical groups in good characteristic},
   JOURNAL = {Proc. Amer. Math. Soc.},
  FJOURNAL = {Proceedings of the American Mathematical Society},
    VOLUME = {147},
      YEAR = {2019},
    NUMBER = {10},
     PAGES = {4205--4219},
      ISSN = {0002-9939,1088-6826},
   MRCLASS = {20G05 (20G40)},
  MRNUMBER = {4002536},
MRREVIEWER = {Paul\ D.\ Levy},
       DOI = {10.1090/proc/14570},
       URL = {https://doi.org/10.1090/proc/14570},
}

@article {GeneralFF,
    AUTHOR = {Meierfrankenfeld, U. and Stellmacher, B.},
     TITLE = {The general {FF}-module theorem},
   JOURNAL = {J. Algebra},
  FJOURNAL = {Journal of Algebra},
    VOLUME = {351},
      YEAR = {2012},
     PAGES = {1--63},
      ISSN = {0021-8693,1090-266X},
   MRCLASS = {20C20 (20D99)},
  MRNUMBER = {2862198},
MRREVIEWER = {Gernot\ Stroth},
       DOI = {10.1016/j.jalgebra.2011.10.029},
       URL = {https://doi.org/10.1016/j.jalgebra.2011.10.029},
}

@article{PellegriniZalesski,
    AUTHOR = {Pellegrini, Marco A. and Zalesski, Alexandre E.},
     TITLE = {Minimal generation of finite simple groups of {L}ie type by
              regular unipotent elements},
   JOURNAL = {Eur. J. Math.},
  FJOURNAL = {European Journal of Mathematics},
    VOLUME = {12},
      YEAR = {2026},
    NUMBER = {2},
     PAGES = {Paper No. 23, 33},
      ISSN = {2199-675X,2199-6768},
   MRCLASS = {20F05 (20D06)},
  MRNUMBER = {5076690},
       DOI = {10.1007/s40879-026-00895-4},
       URL = {https://doi.org/10.1007/s40879-026-00895-4},
}

@article{NormalAbelianI,
AUTHOR = {Grazian, Valentina and Lynd, Justin and Oliver, Bob and Parker, Chris and Semeraro, Jason and van Beek, Martin},
TITLE = {Fusion systems with a weakly closed abelian essential subgroup {I}, {II}},
JOURNAL =  {In preparation}
}

@article {AlperinFusion,
    AUTHOR = {Alperin, J. L.},
     TITLE = {Sylow intersections and fusion},
   JOURNAL = {J. Algebra},
  FJOURNAL = {Journal of Algebra},
    VOLUME = {6},
      YEAR = {1967},
     PAGES = {222--241},
      ISSN = {0021-8693},
   MRCLASS = {20.25},
  MRNUMBER = {215913},
MRREVIEWER = {G.\ E.\ Wall},
       DOI = {10.1016/0021-8693(67)90005-1},
       URL = {https://doi.org/10.1016/0021-8693(67)90005-1},
}

@article {GoldschmidtFusion,
    AUTHOR = {Goldschmidt, David M.},
     TITLE = {A conjugation family for finite groups},
   JOURNAL = {J. Algebra},
  FJOURNAL = {Journal of Algebra},
    VOLUME = {16},
      YEAR = {1970},
     PAGES = {138--142},
      ISSN = {0021-8693},
   MRCLASS = {20.40},
  MRNUMBER = {260869},
MRREVIEWER = {J.\ L.\ Alperin},
       DOI = {10.1016/0021-8693(70)90046-3},
       URL = {https://doi.org/10.1016/0021-8693(70)90046-3},
}

\end{document}